\documentclass{tac}
\thirdleveltheorems
\usepackage{color}
\usepackage[usenames,dvipsnames,svgnames,table]{xcolor}
\usepackage[arrow, curve, frame, matrix, graph, tips]{xy}
\usepackage{amssymb}
\usepackage{amsmath}
\usepackage[colorlinks=true,urlcolor=blue,linkcolor=blue,citecolor=Maroon]{hyperref}
\usepackage{tabularx}

\author{Mark Weber}
\thanks{}
\address{Department of Mathematics, Macquarie University}

\title{Internal algebra classifiers as codescent objects of crossed internal categories}
\copyrightyear{2015}
\eaddress{mark.weber.math@gmail.com}

\newtheorem{thm}{\bf Theorem}
\newtheorem{prop}{\bf Proposition}
\newtheorem{lem}{\bf Lemma}
\newtheorem{cor}{\bf Corollary}

\newtheoremrm{defn}{\bf Definition}
\newtheoremrm{rem}{\bf Remark}
\newtheoremrm{rems}{\bf Remarks}
\newtheoremrm{exam}{\bf Example}
\newtheoremrm{exams}{\bf Examples}
\newtheoremrm{const}{\bf Construction}
\newtheoremrm{notn}{\bf Notation}

\newcommand{\tn}[1]{\textnormal{#1}}
\newcommand{\tnb}[1]{\textnormal{\bf #1}}

\newcommand{\tensor}{\otimes}

\newcommand{\C}{\mathbb{C}}
\newcommand{\R}{\mathbb{R}}

\newcommand{\N}{\mathbb{N}}
\newcommand{\comp}{\circ}
\newcommand{\id}{\tn{id}}
\newcommand{\ca}[1]{\mathcal{#1}}
\newcommand{\ladj}{\dashv}
\newcommand{\iso}{\cong}
\newcommand{\catequiv}{\simeq}
\newcommand{\Set}{\tnb{Set}}

\newcommand{\Cat}{\tnb{Cat}}

\newcommand{\CAT}{\tnb{CAT}}

\newcommand{\op}{\tn{op}}

\newcommand{\TwoCAT}{\mathbf{2}{\tnb{-CAT}}}

\newcommand{\Span}[1]{{\tnb{Span}_{#1}}}

\newcommand{\Polyc}[1]{{\tnb{Poly}_{#1}}}
\newcommand{\PFun}[1]{{\tnb{P}_{#1}}}

\newcommand{\PsAlg}[1]{\tn{Ps-}{#1}\tn{-Alg}}
\newcommand{\PsAlgs}[1]{\tn{Ps-}{#1}\tn{-Alg}_{\tn{s}}}
\newcommand{\PsAlgl}[1]{\tn{Ps-}{#1}\tn{-Alg}_{\tn{l}}}
\newcommand{\Algl}[1]{{#1}\tn{-Alg}_{\tn{l}}}

\newcommand{\Alg}[1]{{#1}\tn{-Alg}}
\newcommand{\Algs}[1]{{#1}\tn{-Alg}_{\tn{s}}}

\newcommand{\Cart}{\tnb{CAT}_{\tn{pb}}}

\newcommand{\CoDesc}{\tn{CoDesc}}
\newcommand{\CrIntCat}[1]{\tn{CrIntCat}(#1)}

\newcommand{\Cnr}{\tn{Cnr}}

\newcommand{\PolyMnd}[1]{\tnb{PolyMnd}_{#1}}

\newcommand{\PolyEnd}[1]{\tnb{PolyEnd}_{#1}}
\renewcommand{\P}{\mathbb{P}}
\newcommand{\B}{\mathbb{B}}
\renewcommand{\S}{\mathbb{S}}

\newcommand{\CatAsOp}{\underline{\Cat}}
\newcommand{\Com}{\mathsf{Com}}
\newcommand{\Ass}{\mathsf{Ass}}
\newcommand{\BrCom}{\mathsf{BCom}}

\begin{document}

\maketitle

\begin{abstract}
Inspired by recent work of Batanin and Berger on the homotopy theory of operads, a general monad-theoretic context for speaking about structures within structures is presented, and the problem of constructing the universal ambient structure containing the prescribed internal structure is studied. Following the work of Lack, these universal objects must be constructed from simplicial objects arising from our monad-theoretic framework, as certain 2-categorical colimits called codescent objects. We isolate the extra structure present on these simplicial objects which enable their codescent objects to be computed. These are the crossed internal categories of the title, and generalise the crossed simplicial groups of Loday and Fiedorowicz. The most general results of this article are concerned with how to compute such codescent objects in 2-categories of internal categories, and on isolating conditions on the monad-theoretic situation which enable these results to apply. Combined with earlier work of the author in which operads are seen as polynomial 2-monads, our results are then applied to the theory of non-symmetric, symmetric and braided operads. In particular, the well-known construction of a PROP from an operad is recovered, as an illustration of our techniques.
\end{abstract}


\section{Introduction}
\label{sec:intro}

A major theme of category theory is that of studying structures within an ambient structure. Then for a given ambient structure, and a given internal structure expressable therein, one is interested in understanding the universal ambient structure containing the prescribed internal structure. This theme manifests itself throughout mathematics, from basic category theory, to topos theory, to the study of Topological Quantum Field Theories (TQFT's), and most recently, thanks to the work of Batanin and Berger \cite{BataninBerger-HtyThyOfAlgOfPolyMnd}, to the homotopy theory of operads. The formalism that we pursue here, traces its roots to the seminal work of Batanin \cite{Batanin-SymmetrisationNOperads-Compactification, Batanin-EckmannHilton, Batanin-LocConstNOpsAsHigherBraidedOps} in which the ambient-internal theme is expressed in an operadic context, and brought to bear on the problem of recognising $n$-fold loop spaces.

Our setting can be applied to clarify certain basic issues coming out of Costello's work on Topological Conformal Field Theories (TCFT's). In \cite{Costello-AInftyOpAndModuliSpaceCurves} Costello sketched a characterisation of TCFT's via a derived analogue of the characterisation of 2-dimensional TQFT's as commutative Frobenius algebras \cite{Abrams-2dTQFT-Frob, KockJ-FrobAlg-2dTQFT}. Various algebraic, homotopical and geometric issues from \cite{Costello-AInftyOpAndModuliSpaceCurves} remain to be clarified, and the compelling vision it presents has, at least partially, motivated a number of works to that end  \cite{BorisovManin-GenOpsInnerCohom, Getzler-OperadsRevisited, KaufmannWard-FeynmanCats}. On the algebraic side one interested in the interplay between two types of operadic structures -- cyclic operads and modular operads, and in particular, in the \emph{modular envelope} construction, in which a modular operad is freely generated from a cyclic operad.

Preliminary to understanding these algebraic issues, is the question of how one defines modular (resp. cyclic) operads within symmetric monoidal categories in the first place. Modular operads were originally defined by Getzler-Kapranov in \cite{GetzlerKapranov-ModularOperads}. Intuitively, if the basic operations of a classical operad are described graphically in terms of corollas, and iterated operations as trees, then modular operads are a more elaborate notion in which more general graphs may appear as iterated operations. In \cite{GetzlerKapranov-ModularOperads} this intuition was formalised directly. On the other hand in \cite{Costello-AInftyOpAndModuliSpaceCurves}, Costello directly defined a symmetric strict monoidal category here denoted as $\ca M$, which is to be regarded as the universal symmetric monoidal category containing a modular operad. Modular operads within a general symmetric monoidal category $\ca V$ are then \emph{defined} to be symmetric strong monoidal functors $\ca M \to \ca V$. Costello does the same for cyclic operads defining directly the universal ambient $\ca C$, there is an inclusion $i : \ca C \to \ca M$, and the modular envelope construction is then given by left Kan extension along $i$.

Though to some extent it is intuitively clear that the definitions of modular operad of \cite{Costello-AInftyOpAndModuliSpaceCurves} and \cite{GetzlerKapranov-ModularOperads} ``more or less'' agree, there is no proof of this given in \cite{Costello-AInftyOpAndModuliSpaceCurves}. The direct definition of the morphisms of $\ca M$, which are graphs of a certain kind, is a little combinatorially involved. Modifying the definition of ``graph'' used, one can define subtly different variants of the category $\ca M$, and so it is desirable to have a framework within which these basic issues are clarified.

There are many operadic notions in contemporary mathematics \cite{LodayVallette-AlgebraicOperads, MarklShniderStasheff-OperadBook, Markl-OperadsPROPs}. In this work, as in its predecessor \cite{Weber-OpPoly2Mnd}, we use the unadorned name \emph{operad} to refer to coloured symmetric operads, which are also known as symmetric multicategories. An important discovery of \cite{BataninBerger-HtyThyOfAlgOfPolyMnd} is that many contemporary operadic notions, such as modular and cyclic operads, are themselves algebras of operads, which are in fact $\Sigma$-free. This fact is established by exhibiting directly, in each case, the finitary polynomial monad defined over $\Set$ whose algebras are the contemporary operadic notion in question. One then appeals to the correspondence \cite{KockJ-PolyFunTrees, SzawielZawadowski-TheoriesOfAnalyticMonads} between finitary polynomial monads and $\Sigma$-free operads. In particular if $\mathsf{MdOp}$ is the operad whose algebras are modular operads, then Costello's $\ca M$ should be the corresponding (coloured) PROP.

In \cite{Weber-AlgKan} a general 2-categorical framework is presented, which in particular explains why the modular envelope is correctly described via left Kan extension along $i : \ca C \to \ca M$. For the sake of these subsequent developments, it is desirable that $\ca M$ be described as an \emph{internal algebra classifier}, in the sense of Section 5.5 of \cite{BataninBerger-HtyThyOfAlgOfPolyMnd}, involving $\mathsf{MdOp}$. Since the ambient structure in this case is that of ``symmetric monoidal category'', and the 2-monad $\tnb{S}$ on $\Cat$ for symmetric monoidal categories is outside the scope of \cite{BataninBerger-HtyThyOfAlgOfPolyMnd}, this is not possible within the framework of \cite{BataninBerger-HtyThyOfAlgOfPolyMnd} as it currently stands.

There is, however, a different way to regard operads as polynomial monads, and this was described in \cite{Weber-OpPoly2Mnd}. In this viewpoint one uses polynomial monads defined over $\Cat$, and all operads, not just $\Sigma$-free ones can be viewed in this way. In particular the terminal operad $\Com$ is identified with the polynomial monad $\tnb{S}$. Moreover, as we will see in Examples \ref{exams:algebras-of-operad-in-V}, the unique operad morphism $\mathsf{MdOp} \to \Com$ gives rise to the appropriate monad theoretic data enabling one to recover the notion of a ``modular operad within a symmetric monoidal category'' from the formalism.

In Section \ref{sec:IntAlg}, we generalise the developments of Section 7 of \cite{Batanin-EckmannHilton}, and provide a monad-theoretic context for the general situation, called an \emph{adjunction of 2-monads} $F : (\ca L,S) \to (\ca K,T)$. This consists of a 2-monad $S$ on a 2-category $\ca L$ meant to parametrise the type of internal structure, a 2-monad $T$ on a 2-category $\ca K$ meant to parametrise the type of ambient structure, an adjunction $F_! \ladj F^* : \ca L \to \ca K$ between the 2-categories on which these 2-monads act, and data exhibiting $F_!$ and $F^*$ as the appropriate types of morphism of 2-monads in the sense of \cite{Street-FTM}. Given such a context it is then meaningful to consider $S$-algebras internal to a $T$-algebra as we do in Definition \ref{defn:internal-algebra}. Having established this framework we define the universal strict $T$-algebra containing an internal $S$-algebra, here called the \emph{internal $S$-algebra classifier}, in Definition \ref{def:int-alg-classifier}. This is exactly the universal ambient structure containing the prescribed internal structure, in our setting.

The most basic example of a universal ambient structure, computable from monad-theoretic data, is the algebraists' simplicial category, here denoted $\Delta_+$. It was exhibited as the free strict monoidal category containing a monoid in \cite{MacLane-CWM}. As is well-known, one can recover $\Delta_+$ from the monad $M$ on $\Set$ whose algebras are monoids, since the simplicial set
\begin{equation}\label{eq:nerve-Delta-plus}
\xygraph{!{0;(2,0):(0,1)::} {...} [r(.3)] {M^31}="p0" [r] {M^21}="p1" [r] {M1}="p2"
"p2":"p1"|-{M\eta_1} "p1":@<1.5ex>"p2"^-{\mu_1} "p1":@<-1.5ex>"p2"_-{M(!)} "p0":@<1.5ex>"p1"^-{\mu_{M1}} "p0":"p1"|-{M\mu_1} "p0":@<-1.5ex>"p1"_-{M^2(!)}}
\end{equation}
in which $\eta$ is the unit and $\mu$ is the multiplication of the monad $M$, turns out to be the nerve of $\Delta_+$.

This situation is clarified by 2-dimensional monad theory. It is almost a tautology that a monoid in a monoidal category $\ca V$ is the same thing as a lax monoidal functor $1 \to \ca V$. One has the 2-monad $\tnb{M}$ on $\Cat$ whose strict algebras are strict monoidal categories, strict morphisms are strict monoidal functors and lax morphisms are lax monoidal functors. Denoting the inclusion of the strict algebras and strict maps, amongst the strict algebras and lax maps by
\[ J_{\tnb{M}} : \Algs {\tnb{M}} \longrightarrow \Algl {\tnb{M}}, \]
the universal property of $\Delta_+$ recalled above says exactly that it is the value at $1$ of a left adjoint to $J_{\tnb{M}}$. Moreover, there is a general monad-theoretic explanation for why $\Delta_+$ also enjoys a universal property with respect to all monoidal categories, not just strict ones. This universal property says that the category of monoids in a monoidal category $\ca V$ is equivalent to the category of strong monoidal functors $\Delta_+ \to \ca V$.

Such considerations lead to the desire to understand how to compute the left adjoint $(-)^{\dagger}_T$ to the analogous inclusion
\[ J_T : \Algs T \longrightarrow \Algl T \]
for a general 2-monad $T$ on a 2-category $\ca K$. In \cite{Lack-Codescent} Lack understood that $(-)^{\dagger}_T$ is computed as a certain weighted colimit, called a \emph{codescent object}, of the corresponding simplicial object in $\Algs T$ obtained by repeatedly applying $T$. The computation of $\Delta_+ = (1)^{\dagger}_{\tnb{M}}$ as a codescent object is straight forward for two reasons. The first, is that $\tnb{M}$ preserves all codescent objects, and so codescent objects in $\Algs {\tnb{M}}$ are computed as in $\Cat$. Secondly, in this case the simplicial object in question is the componentwise discrete category object (\ref{eq:nerve-Delta-plus}), and every category is the codescent object of its nerve.

The codescent objects arising from similar contexts in applications \cite{Batanin-EckmannHilton, BataninBerger-HtyThyOfAlgOfPolyMnd} are of componentwise discrete category objects, and thus are easy to compute. More generally, Bourke \cite{Bourke-Thesis} understood the computation of codescent objects of cateads in 2-categories of the form $\tn{Cat}(\ca E)$. For $\ca E$ a category with pullbacks, $\tn{Cat}(\ca E)$ is the 2-category of categories internal to $\ca E$. A \emph{catead} in a 2-category is a category object of a certain special form. However in the case where $T$ is the 2-monad $\tnb{S}$ on $\Cat$ for symmetric monoidal categories, while $(-)^{\dagger}_{\tnb S}$ is computed by taking codescent objects of category objects in $\Cat$, these category objects are not cateads.

In this article we identify extra structure present on the category objects in such examples, which enable us to compute their codescent objects. This is the structure of a \emph{crossed internal category} so named because they generalise the crossed simplicial groups of Loday-Fiedorowicz \cite{FieLoday-CrossedSimplicial}. Given a crossed internal category $X$ in $\tn{Cat}(\ca E)$ where $\ca E$ is locally cartesian closed, the computation of its codescent object then proceeds in two steps. The first is to compute an associated 2-category $\tn{Cnr}(X)$ internal to $\ca E$, and then one takes connected components of the hom categories of this internal 2-category. The general results expressing this method are given as Theorem \ref{thm:codesc-cr-int-cat} and Corollary \ref{cor:codesc-gen-cicat}.

Thus for a larger class of 2-monads $T$, one has a complete understanding of how to compute $(1)^{\dagger}_T$, which analogously to the case $T = \tnb{M}$, is the free strict $T$-algebra containing an internal $T$-algebra. We illustrate this when $T$ is the 2-monad $\tnb{S}$ on $\Cat$ for symmetric monoidal categories, in which case $(1)^{\dagger}_T$ is a skeleton of the category of finite sets, in Section \ref{ssec:functions}. In Section \ref{ssec:Vines} we perform similar calculations for the braided monoidal category 2-monad $\tnb{B}$, to exhibit the category of vines, in the sense of Lavers \cite{Lavers-Vines}, as the free braided strict monoidal category containing a commutative monoid.

By contrast with these examples, in the general situation of an adjunction of 2-monads, the type of internal structure and the type of ambient structure are not necessarily the same. Thus from the data of an adjunction of 2-monads $F : (\ca L,S) \to (\ca K,T)$, we give an associated simplicial $T$-algebra in Construction \ref{const:simplicial-objects-from-monad-adjunction}, whose codescent object is the internal $S$-algebra classifier by Proposition \ref{prop:dagger-F}, generalising \cite{Batanin-EckmannHilton} Theorem 7.3. The abstract conditions on an adjunction of 2-monads which enable our methods of codescent calculation to apply to the corresponding internal algebra classifiers are given in Proposition \ref{prop:crossed-int-cats-from-classifiers}. To summarise, our understanding of how to compute internal algebra classifiers is embodied by the three general results: Proposition \ref{prop:crossed-int-cats-from-classifiers}, Theorem \ref{thm:codesc-cr-int-cat} and Corollary \ref{cor:codesc-gen-cicat}.

Morphisms of operads provide a basic source of examples of adjunctions of 2-monads to which our results can be applied. As explained in \cite{Weber-OpPoly2Mnd} and recalled here in Section \ref{sec:background}, operads can be regarded as polynomial 2-monads over the polynomial 2-monad $\tnb{S}$ for symmetric monoidal categories. Moreover given a morphism of operads, one obtains an adjunction between the corresponding 2-monads, and this is described explicitly in Section \ref{ssec:explicit-adj2mnd-op-morphism}. We apply our general results in this situation to give, for a morphism $F : S \to T$ of operads, an explicit description of the free strict $T$-algebra containing an internal $S$-algebra, in Theorem \ref{thm:intalg-classifier-from-operad-morphism}. Taking $F$ to be the unique morphism $T \to \Com$ into the terminal operad, the corresponding internal algebra classifier describes the effect on objects of the left adjoint of the fundamental biadjunction between symmetric monoidal categories and operads. This is described in Corollary \ref{cor:smoncat-opd-adjunctions}. Other internal algebra classifiers arising naturally from an operad morphism are discussed in Section \ref{ssec:Sigma-free-operads}. The relationship between these various alternatives is understood in Theorem \ref{thm:T^S-vs-T/Sigma^S/Sigma}, and this is then used in Examples \ref{exams:Batanin-Berger-classifiers} to reconcile the general method of Theorem \ref{thm:intalg-classifier-from-operad-morphism} with the calculations of Batanin and Berger in \cite{BataninBerger-HtyThyOfAlgOfPolyMnd}.

Such applications are readily generalised in our framework. As explained in Section \ref{ssec:polynomials}, non-symmetric operads can be regarded as polynomial 2-monads over the 2-monad $\tnb{M}$ for monoidal categories, and braided operads can be regarded as polynomial 2-monads over the polynomial 2-monads $\tnb{B}$ for braided monoidal categories. Thus one obtains non-symmetric and braided analogues of our results, and these are indicated in Section \ref{ssec:braided-props}. However there are many other situations to which our methods could be applied to calculate internal algebra classifiers. These include from adjunctions of 2-monads arising from morphisms of non-symmetric, symmetric or braided $\Cat$-operads, following Remark 3.5 of \cite{Weber-OpPoly2Mnd}.

\emph{Simplicial notation}.
We use two alternative notations for finite ordinals. For each $n \in \N$ one has the ordered set $[n] = \{0 < ... < n\}$. The category whose objects are these ordered sets and order-preserving functions between them is the topologists' simplicial category $\Delta$. Alternatively, for each $n \in \N$ one has the ordered set $\underline{n} = \{1 < ... < n\}$. The category whose objects are natural numbers, and morphisms $m \to n$ are order preserving functions $\underline{m} \to \underline{n}$, is the algebraists' simplicial category $\Delta_+$.

We use the standard notation for the usual generating morphisms of $\Delta$; with the standard coface map $\delta_i : [n] \to [n+1]$ being the injection whose image doesn't include $i$, and the standard codegeneracy map $\sigma_i : [n+1] \to [n]$ being the surjection which identifies $i$ and $i+1$. A \emph{simplicial object} in a category $\ca E$ is a functor $X : \Delta^{\op} \to \ca E$, whose effect on objects and generating maps we denote as
\[ \begin{array}{rcl} {\delta_i : [n] \to [n+1]} & {\mapsto} & {d_i : X_{n+1} \to X_n} \\
{\sigma_i : [n+1] \to [n]} & {\mapsto} & {s_i : X_n \to X_{n+1}} \end{array} \]
in the usual way, the $d_i$ being the face maps and the $s_i$ the degeneracy maps of $X$. Similarly for a \emph{cosimplicial object} $X : \Delta \to \ca E$ one has coface $\delta_i : X_n \to X_{n+1}$ and codegeneracy maps $\sigma_i : X_{n+1} \to X_n$. The \emph{standard cosimplicial object} regards each ordinal $[n]$ as a category, and each order preserving map as a functor, and is denoted as $\delta : \Delta \to \Cat$.

\section{Operads and polynomial 2-monads}
\label{sec:background}

This section is devoted to background. In Section \ref{ssec:monads} we describe 2-dimensional monad theory. In Section \ref{ssec:operads} we recall the notion of operad and describe our conventions and notations regarding them. In Section \ref{ssec:polynomials} we recall the basics on polynomial 2-monads and how operads can be regarded as such.

\subsection{Monad theory.}
\label{ssec:monads}

We adopt the standard practise of referring to a monad in a bicategory $\ca B$ as a pair $(A,t)$, where $A$ is the underlying object in $\ca B$, $t:A \to A$ is the underlying endoarrow, and the unit $\eta^t : 1_A \to t$ and multiplication $\mu^t : t^2 \to t$ are left implicit. In particular we use this convention in the 2-category of 2-categories and speak of a 2-monad $(\ca K,T)$.

The underlying object and arrow of the Eilenberg-Moore object of $(A,t)$, when it exists, is denoted as $u^t : A^t \to A$. In the case $\ca K = \Cat$ this is the forgetful functor out of the category of algebras of $t$, and in general $u^t$ is a right adjoint. The underlying object and arrow of the Kleisli object of $(A,t)$, when it exists, is denoted as $f_t : A_t \to A$, and is in general a left adjoint. In the case of a 2-monad $(\ca K,T)$, $u^t$ and $f_t$ are denoted as
\[ \begin{array}{lccr} {U^T : \Algs T \longrightarrow \ca K} &&&
{F_T : \ca K \longrightarrow \tn{Kl}(T)} \end{array} \]
respectively.

In this article the monad functors and opfunctors of \cite{Street-FTM} are called \emph{lax morphisms} and \emph{colax morphisms} of monads respectively. So for us a lax morphism $f:(A,t) \to (B,s)$ in $\ca B$ consists of an arrow $f:A \to B$ in $\ca B$, and a coherence 2-cell $f^l:sf \to ft$ which satisfies $f^l(\eta^sf) = f\eta^t$ and $f^l(\mu^sf) = (f\mu^t)(f^lt)(sf^l)$. For a colax morphism $f:(A,t) \to (B,s)$ one has the coherence $f^c:ft \to sf$, which satisfies $f^c(f\eta^t) = \eta^sf$ and $f^c(f\mu^t) = (\mu^sf)(sf^c)(f^ct)$.

A 2-monad $(\ca K,T)$ has various notions of algebra and algebra morphism. Recall that for $A \in \ca K$, a \emph{pseudo $T$-algebra structure} on $A$ consists of an arrow $a:TA \to A$, invertible coherence 2-cells $\overline{a}_0:1_A \to a\eta_A$ and $\overline{a}_2:aT(a) \to a\mu_A$, satisfying the following axioms:
\[ \xygraph{*=(0,1.5){\xybox{\xygraph{!{0;(4,0):(0,.1)}
{\xybox{\xygraph{!{0;(2,0):(0,.5)::} {a}="l" [r] {a\eta_Aa}="m" [d] {a}="r" "l":"m"^-{\overline{a}_0a}:"r"^-{\overline{a}_2\eta_{TA}}:@{<-}"l"^-{\id} "l" [d(.35)r(.7)] {\scriptsize{=}}}}}
[d(.02)r]
{\xybox{\xygraph{!{0;(2.5,0):(0,.4)::} {aT(a)T^(a)}="tl" [r] {a\mu_AT^2(a)}="tr" [d] {a\mu_A\mu_{TA}}="br" [l] {aT(a)T(\mu_A)}="bl" "tl":"tr"^-{\overline{a}_2T^2(a)}:"br"^-{\overline{a}_2\mu_{TA}}:@{<-}"bl"^-{\overline{a}_2T(\mu_A)}:@{<-}"tl"^-{aT(\overline{a}_2)} "tl" [d(.5)r(.5)] {\scriptsize{=}}}}}
[u(.02)r]
{\xybox{\xygraph{!{0;(2,0):(0,.5)::} {a}="l" [l] {aT(a)T(\eta_A)}="m" [d] {a}="r" "l":"m"_-{aT(\overline{a}_0)}:"r"_-{\overline{a}_2T(\eta_A)}:@{<-}"l"_-{\id} "l" [d(.35)l(.7)] {\scriptsize{=}}}}}}}}} \]
We denote a pseudo $T$-algebra as a pair $(A,a)$ leaving the data $\overline{a}_0$ and $\overline{a}_2$ implicit. When these coherence isomorphisms are identities, $(A,a)$ is said to be a \emph{strict} $T$-algebra.

Recall that a \emph{lax morphism} $(A,a) \to (B,b)$ between pseudo $T$-algebras is a pair $(f,\overline{f})$, where $f:A \to B$ and $\overline{f}:bT(f) \to fa$, satisfying the following axioms:
\[ \xygraph{{\xybox{\xygraph{{f}="l" [dl] {bT(f)\eta_A}="m" [r(2)] {fa\eta_A}="r" "l":"m"_-{\overline{b}_0f}:"r"_-{\overline{f}\eta_A}:@{<-}"l"_-{f\overline{a}_0} "l" [d(.6)] {\scriptsize{=}}}}}
[r(5)]
{\xybox{\xygraph{!{0;(1.5,0):(0,.6667)::} {bT(b)T^2(f)}="p1" [r(2)] {b\mu_BT^2(f)}="p2" [d] {fa\mu_A}="p3" [dl] {faT(a)}="p4" [ul] {bT(fa)}="p5" "p1":"p2"^-{\overline{b}_2T^2(f)}:"p3"^-{\overline{f}\mu_A}:@{<-}"p4"^-{f\overline{a}_2}:@{<-}"p5"^-{\overline{f}T(a)}:@{<-}"p1"^-{bT(\overline{f})} "p1" [d(.8)r] {\scriptsize{=}}}}}} \]
When $\overline{f}$ is an isomorphism, $f$ is said to be a \emph{pseudomorphism}, and when $\overline{f}$ is an identity, $f$ is said to be a \emph{strict morphism} of algebras. Given lax $T$-algebra morphisms $f$ and $g:(A,a) \to (B,b)$, a $T$-algebra 2-cell $f \to g$ is a 2-cell $\phi:f \to g$ in $\ca K$ such that $\overline{g}(bT(\phi))=(\phi a)\overline{f}$. The various notions of algebra and algebra morphism form the various 2-categories of algebras of $T$, the standard notation for which is recalled in the table. In each case, the 2-cells are just the $T$-algebra 2-cells between the appropriate $T$-algebra morphisms. We denote by $J_T : \Algs T \to \Algl T$ the inclusion.
\begin{table}
\label{table:algebras-of-2-monads}
\caption{2-categories of algebras of a 2-monad $T$}
\centering
\begin{tabular}{|l|l|l|}
\hline
Name & Objects & Arrows \\ \hline \hline
$\PsAlgl{T}$ & pseudo $T$-algebras & lax morphisms \\ \hline
$\PsAlg{T}$ & pseudo $T$-algebras & pseudomorphisms \\ \hline
$\PsAlgs{T}$ & pseudo $T$-algebras & strict morphisms \\ \hline
$\Algl{T}$ & strict $T$-algebras & lax morphisms \\ \hline
$\Alg{T}$ & strict $T$-algebras & pseudomorphisms \\ \hline
$\Algs{T}$ & strict $T$-algebras & strict morphisms \\ \hline
\end{tabular}
\end{table}

\subsection{Operads.}
\label{ssec:operads}

At various stages in this article we shall be manipulating sequences of elements, and so as in \cite{Weber-Funny} we use the following notation. Most generally given a sequence of sets $(X_1, ... , X_n)$, a typical element of its cartesian product will be denoted as either $(x_1,...,x_n)$, or by the abbreviated notation $(x_k)_{1{\leq}k{\leq}n}$ or even $(x_k)_k$ when no confusion would result. Given a subset $s \subseteq \underline{n}$, the corresponding subsequence is denoted $(x_k)_{k \in s}$. In the situations when $s$ itself is identified by some logical condition, this condition may replace $k \in s$ in this notation. For instance given a function $h : \underline{n} \to \underline{m}$ and $l \in \underline{m}$, one may write $(x_k)_{hk=l}$ in the case where $s = h^{-1}\{l\}$. By an abuse of notation, a singleton sequence $(x)$ is often written as an element $x$.

One place where this notation is useful is in defining and then discussing operads. In this article as in \cite{Weber-OpPoly2Mnd}, we use the term \emph{operad} for what is commonly called either a symmetric multicategory, or a coloured (symmetric) operad. Our operadic notation and terminology is consistent with that of \cite{Weber-OpPoly2Mnd}.

As such, an \emph{operad} $T$ has an underlying \emph{collection} which consists of a set $I$ of \emph{colours} or \emph{objects}, and sets $T((i_k)_{1{\leq}k{\leq}n},i)$ of \emph{operations} or \emph{arrows}, defined for all sequences $(i_k)_{1{\leq}k{\leq}n}$ and elements $i$ from $I$. A typical element of $T((i_k)_{1{\leq}k{\leq}n},i)$ may be denoted as $\alpha : (i_k)_{1{\leq}k{\leq}n} \to i$. Denoting the $n$-th symmetric group as $\Sigma_n$, given $\rho \in \Sigma_n$ and $\alpha : (i_k)_{1{\leq}k{\leq}n} \to i$ in $T$, one has another operation $\alpha\rho : (i_{\rho k})_k \to i$ in which the source sequence has been permuted by $\rho$. The assignations $(\alpha,\rho) \mapsto \alpha\rho$ are functorial with respect to the composition of permutations.
 
In addition to this underlying collection, one also has \emph{identity} arrows $1_i : i \to i$, and an operation of \emph{composition} or \emph{substitution}, which takes the data of an arrow $\alpha : (i_k)_{1{\leq}k{\leq}m} \to i$ and a sequence $(\beta_k : (i_{kl})_{1{\leq}l{\leq}m_k} \to i_k)_k$ of arrows from $T$, and returns their composite $\alpha(\beta_k)_k : (i_{kl})_{kl} \to i$, in which the source sequence is written less tersely as $(i_{11},...,i_{1m_1}, ...,i_{k1},...,i_{km_k})$. The unit, associativity and composition-equivariance axioms of an operad can be written as
\[ \begin{array}{lccr} {1_i \comp (\alpha) = \alpha = \alpha \comp (1_{i_j})_j} &&& {\alpha \comp (\beta_j \comp (\gamma_{jk})_k)_j = (\alpha \comp (\beta_j)_j) \comp (\gamma_{jk})_{jk}} \end{array} \]
\[  (\alpha \comp (\beta_j)_j)(\rho(\rho_j)_j) = (\alpha\rho) \comp (\beta_j\rho_j)_j. \]
See section 3 of \cite{Weber-OpPoly2Mnd} for more details.

While this notation is efficient and precise, it is not intuitive, and so it is also worth remembering the standard way of depicting operations of an operad $T$ as trees. For example an operation $\alpha : (i_1,i_2,i_3,i_4) \to i$ is denoted as on the left in
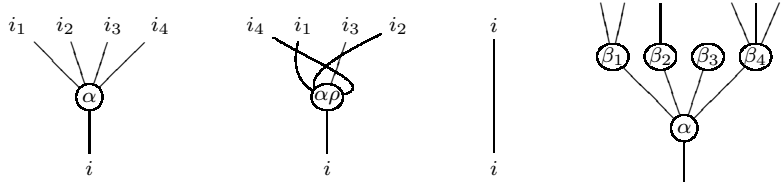
\[ \xygraph{{\xybox{\xygraph{!{0;(.5,0):(0,1.5)::}
{\scriptstyle{\alpha}} *\xycircle<5pt,5pt>{-}="p0"
(-[d] {\scriptstyle{i}},-[l(1.5)u] {\scriptstyle{i_1}},-[l(.5)u] {\scriptstyle{i_2}},-[r(.5)u] {\scriptstyle{i_3}},-[r(1.5)u] {\scriptstyle{i_4}})}}}
[r(2.5)]
{\xybox{\xygraph{!{0;(.5,0):(0,1.5)::}
{\scriptstyle{\alpha\rho}} *\xycircle<6pt,5pt>{-}="p0"
(-[d] {\scriptstyle{i}},-@`{"p0"+(-1,0.2)}[l(.5)u] {\scriptstyle{i_1}},-@`{"p0"+(-1,0.25)}[r(1.5)u] {\scriptstyle{i_2}},-[r(.5)u] {\scriptstyle{i_3}},-@`{"p0"+(1.5,0.1),"p0"+(-.5,.5)}[l(1.5)u] {\scriptstyle{i_4}})}}}
[r(1.75)]
{\xybox{\xygraph{!{0;(.5,0):(0,1.5)::}
{\scriptstyle{i}} -[d(2)] {\scriptstyle{i}}}}}
[r(2)]
{\xybox{\xygraph{!{0;(.5,0):(0,1.5)::}
{\scriptstyle{\alpha}} *\xycircle<5pt,5pt>{-}="p0"
(-[d(.75)],-[l(1.5)u] {\scriptstyle{\beta_1}} *\xycircle<6pt,5pt>{-} (-[u(.75)l(.25)],-[u(.75)r(.25)]),-[l(.5)u] {\scriptstyle{\beta_2}} *\xycircle<6pt,5pt>{-} -[u(.75)],-[r(.5)u] {\scriptstyle{\beta_3}} *\xycircle<6pt,5pt>{-},-[r(1.5)u] {\scriptstyle{\beta_4}} *\xycircle<6pt,5pt>{-} (-[u(.75)l(.5)],-[u(.75)],-[u(.75)r(.5)]))}}}} \]
and acting on it by the permutation $\rho = (142)$ is drawn as above by crossing the ``input wires'' according to $\rho$ as shown. The identity on $i$ is depicted second on the right, and the right-most diagram in the previous display is the result of forming the composite $\alpha(\beta_1,\beta_2,\beta_3,\beta_4)$.

We denote by $\tn{Br}_n$ the $n$-th braid group. A \emph{braided operad} is defined as above, except that one acts on arrows $\alpha : (i_k)_{1{\leq}k{\leq}n} \to i$ by braids $\rho \in \tn{Br}_n$ (instead of permutations) to form $\alpha\rho$. In terms of tree diagrams, the result of acting on an operation $\alpha : (i_1,i_2,i_3,i_4) \to i$ by the braid $\rho \in \tn{Br}_4$ as depicted on the left is denoted as on the right in
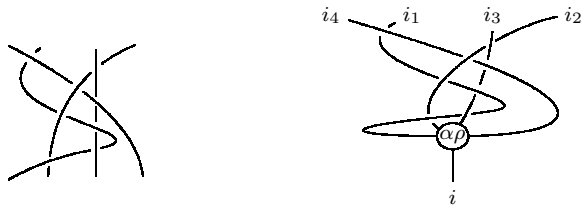
\[ \xygraph{{\xybox{\xygraph{!{0;(.5,0):(0,3)::} 
{}="t1" [r] {}="t2" [r] {}="t3" [r] {}="t4" "t1" [d] {}="b1" [r] {}="b2" [r] {}="b3" [r] {}="b4"
"t1" -@`{"t4"+(0,-.5)} "b4"|(.32)*=<4pt>{}
"t2" -@`{"t1"+(0,-.25),"t2"+(0,-.5),"b3"+(.5,.35),"b3"+(.5,.25),"b2"+(0,.2)} "b1"|(.04)*=<4pt>{}|(.298)*=<4pt>{}|(.776)*=<4pt>{}
"t3" - "b3"|(.42)*=<4pt>{}|(.6)*=<4pt>{}
"t4" -@`{"t2"+(0,-.25)} "b2"|(.288)*=<4pt>{}|(.89)*=<4pt>{}}}}
[r(4)d(.2)]
{\xybox{\xygraph{!{0;(.85,0):(0,1.5)::}
{\scriptstyle{\alpha\rho}} *\xycircle<6pt,5pt>{-}="p0"
(-[d(.5)] {\scriptstyle{i}},-@`{"p0"+(-1.25,0),"p0"+(-1.25,0.1),"p0"+(.75,.2),"p0"+(.75,.3),"p0"+(-1.5,.75)}[l(.5)u] {\scriptstyle{i_1}}|(.448)*=<4pt>{}|(.79)*=<4pt>{}|(.952)*=<4pt>{},-@`{"p0"+(-1,0.25),"p0"+(.75,1)}[r(1.5)u] {\scriptstyle{i_2}}|(.15)*=<4pt>{}|(.635)*=<4pt>{},-@`{"p0"+(.5,.5)}[r(.5)u] {\scriptstyle{i_3}}|(.38)*=<4pt>{}|(.59)*=<4pt>{},-@`{"p0"+(1.5,0),"p0"+(1.5,.4),"p0"+(-.5,.8)}[l(1.5)u] {\scriptstyle{i_4}}|(.7)*=<4pt>{})}}}} \]

\subsection{Polynomial 2-monads.}
\label{ssec:polynomials}

Our main examples of 2-monads arise by regarding operads as polynomial 2-monads as in \cite{Weber-OpPoly2Mnd}. Recall that a \emph{polynomial from $I$ to $J$} in $\Cat$ \cite{Weber-PolynomialFunctors} consists of categories and functors as on the left
\[ \xygraph{{\xybox{\xygraph{{I}="p0" [r] {E}="p1" [r] {B}="p2" [r] {J}="p3" "p0":@{<-}"p1"^-{s}:"p2"^-{p}:"p3"^-{t}}}}
[r(5)]
{\xybox{\xygraph{!{0;(1.5,0):} {\Cat/I}="p0" [r] {\Cat/E}="p1" [r] {\Cat/B}="p2" [r] {\Cat/J}="p3" "p0":"p1"^-{\Delta_s}:"p2"^-{\Pi_p}:"p3"^-{\Sigma_t}}}}} \]
in which $p$ is an exponentiable functor, and the \emph{polynomial 2-functor} it determines is the composite functor on the right, in which $\Sigma_t$ is the process of composition with $t$, $\Delta_s$ is the process of pulling back along $s$, and $\Pi_p$ is right adjoint to pulling back along $p$. Such polynomials form a 2-bicategory{\footnotemark{\footnotetext{Recall from \cite{Weber-PolynomialFunctors} that a \emph{2-bicategory} is a $\Cat$-enriched bicategory, so that homs of a 2-bicategory form 2-categories. For example strict 3-categories such as $\TwoCAT$ are examples. Ignoring the 2-cells in the homs of a 2-bicategory leaves one with an ordinary bicategory.}}} $\Polyc {\Cat}$, in which the objects are categories and a morphism $I \to J$ is a polynomial as above. The assignment of a polynomial to its associated polynomial 2-functor is the effect on arrows of a homomorphism
\[ \PFun {\Cat} : \Polyc{\Cat} \longrightarrow \TwoCAT. \]
In particular $\PFun{\Cat}$ sends a monad in $\Polyc{\Cat}$ on a category $I$, to a 2-monad on $\Cat/I$.

There are three examples of polynomial 2-monads on $\Cat$ which for us are fundamental. The first of these is the 2-monad $\tnb{M}$ for strict monoidal categories. For $A \in \Cat$, $\tnb{M}(A)$ is the category whose objects are finite sequences of objects of $A$ and morphisms are levelwise maps. In other words a morphism of $\tnb{M}(A)$ is of the form $(f_k)_{1{\leq}k{\leq}m} : (a_k)_k \to (b_k)_k$ where for each $k$, $f_k : a_k \to b_k$ is a morphism of $A$. The unit and multiplication of $\tnb{M}$ are given by the inclusion of singleton sequences and by concatenation of sequences respectively. Strict and pseudo $\tnb{M}$-algebras are strict monoidal categories and monoidal categories respectively. Strict, pseudo, lax and colax morphisms of $\tnb{M}$-algebras are strict monoidal, strong monoidal, lax and colax monoidal functors respectively. 

The polynomial underlying $\tnb{M}$ is the componentwise-discrete
\[ \xygraph{{1}="p0" [r] *!(0,.025){\xybox{\xygraph{{\N_*}}}}="p1" [r] {\N}="p2" [r] {1}="p3" "p0":@{<-}"p1"^-{}:"p2"^-{p}:"p3"^-{}} \]
in which $\N$ is the set of natural numbers (including $0$), and $p$ is the function whose fibre over $n$ has cardinality $n$. We regard  the elements of $\N_*$ as pairs $(i,n)$ where $n \in \N$ and $1 \leq i \leq n$, and $p(i,n) = n$. This polynomial is really a polynomial in $\Set$, and this polynomial monad was first considered by B\'{e}nabou in \cite{Benabou-FreeMonoids} at the generality where $\Set$ is replaced by an elementary topos with a natural numbers object.

Our second fundamental example is the 2-monad $\tnb{S}$ whose strict algebras are symmetric strict monoidal categories. General symmetric monoidal categories are the pseudo algebras of $\tnb{S}$, and as with $\tnb{M}$, the various types of symmetric monoidal functor match up as expected with the various types of $\tnb{S}$-algebra morphism. The polynomial underlying $\tnb{S}$ is
\[ \xygraph{{1}="p0" [r] *!(0,.025){\xybox{\xygraph{{\P_*}}}}="p1" [r] {\P}="p2" [r] {1}="p3" "p0":@{<-}"p1"^-{}:"p2"^-{}:"p3"^-{}} \]
in which $\P$ is the permutation category, whose objects are natural numbers, and morphisms are permutations $\rho : n \to n$ (there are no morphisms $n \to m$ for $n \neq m$), so that the endomorphism monoid $\P(n,n)$ is the $n$-th symmetric group $\Sigma_n$. The objects of $\P_*$ are the elements of $\N_*$, and a morphism $(i,n) \to (j,n)$ is a permutation $\rho \in \Sigma_n$ such that $\rho i = j$ (i.e the morphisms are ``base-point preserving permutations''). An explicit description of $\tnb{S}$ is the same as for $\tnb{M}$ except that $\tnb{S}(A)$ has more morphisms, a morphism $(a_k)_{1{\leq}k{\leq}n} \to (b_k)_{1{\leq}k{\leq}n}$ being a pair $(\rho,(f_{k})_k)$, where for all $k$, $f_k : a_k \to b_{\rho k}$. More intuitively, such a morphism is a permutation labelled by the arrows of $A$ as in the diagram
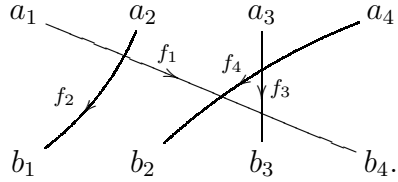
\[ \xygraph{!{0;(1.25,0):(0,1.25)::} {a_1}="t1" [r] {a_2}="t2" [r] {a_3}="t3" [r] {a_4}="t4"
"t1" [d] {b_1}="b1" [r] {b_2}="b2" [r] {b_3}="b3" [r] {b_4.}="b4"
"t1"-"b4"|(.43)@{>}^(.38){f_1} "t2"-@/^{.5pc}/"b1"|(.6)@{>}_(.6){f_2} "t3"-"b3"|(.57)@{>}^(.5){f_3} "t4"-@/_{.5pc}/"b2"|(.57)@{>}_(.54){f_4}} \]

Replacing permutations by braids in the previous paragraph brings us to our third example, the 2-monad $\tnb{B}$ on $\Cat$ whose strict algebras are braided strict monoidal categories. Its underlying polynomial is
\[ \xygraph{{1}="p0" [r] *!(0,.025){\xybox{\xygraph{{\B_*}}}}="p1" [r] {\B}="p2" [r] {1}="p3" "p0":@{<-}"p1"^-{}:"p2"^-{}:"p3"^-{}} \]
where $\B$ is the braid category, with natural numbers as objects, all morphisms are endomorphisms, and the endomorphism monoid $\B(n,n)$ is the $n$-th braid group $\tn{Br}_n$. A morphism of $\tnb{B}(A)$ is a braid labelled by the arrows of $A$ as in
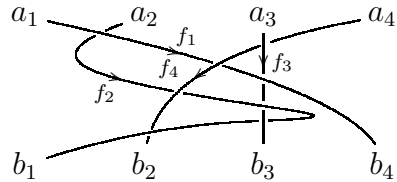
\[ \xygraph{!{0;(1.25,0):(0,1.25)::} 
{a_1}="t1" [r] {a_2}="t2" [r] {a_3}="t3" [r] {a_4}="t4"
"t1" [d] {b_1}="b1" [r] {b_2}="b2" [r] {b_3}="b3" [r] {b_4}="b4"
"t1" -@`{"t4"+(0,-.5)} "b4"|(.25)@{>}^(.25){f_1}|(.32)*=<4pt>{}
"t2" -@`{"t1"+(0,-.25),"t2"+(0,-.5),"b3"+(.5,.35),"b3"+(.5,.3),"b2"+(0,.28)} "b1"|(.04)*=<4pt>{}|(.22)@{>}_(.2){f_2}|(.298)*=<4pt>{}|(.776)*=<4pt>{}
"t3" - "b3"|(.35)@{>}^(.31){f_3}|(.42)*=<4pt>{}|(.6)*=<4pt>{}
"t4" -@`{"t2"+(0,-.25)} "b2"|(.288)*=<4pt>{}|(.55)@{>}_(.59){f_4}|(.85)*=<3pt>{}} \]
That the 2-monad for braided strict monoidal categories really is described in this way was established in \cite{JoyalStreet-GeometryTensorCalculus}, where it was denoted as $\mathbb{B} \wr (-)$.

Monads in $\Polyc{\Cat}$ are the objects of a category $\PolyMnd {\Cat}$, in which the data of a morphism $(I,s,p,t) \to (J,s',p',t')$ is a commutative diagram
\[ \xygraph{!{0;(1.5,0):(0,.6667)::} {I}="p0" [r] {E}="p1" [r] {B}="p2" [r] {I}="p3" [d] {J}="p4" [l] {B'}="p5" [l] {E'}="p6" [l] {J}="p7" "p0":@{<-}"p1"^-{s}:"p2"^-{p}:"p3"^-{t}:"p4"^-{f}:@{<-}"p5"^-{t'}:@{<-}"p6"^-{p'}:"p7"^-{s'}:@{<-}"p0"^-{f} "p1":"p6"_-{f_2} "p2":"p5"^-{f_1} "p1":@{}"p5"|-{\tn{pb}}} \]
compatible in the appropriate sense with the monad structures on $(s,p,t)$ and $(s',p',t')$. As was explained in \cite{Weber-OpPoly2Mnd}, an operad $T$ with set of colours $I$ can be identified as a morphism
\[ \xygraph{!{0;(1.5,0):(0,.6667)::} {I}="p0" [r] {E}="p1" [r] {B}="p2" [r] {I}="p3" [d] {1}="p4" [l] {\P}="p5" [l] *!(0,.025){\xybox{\xygraph{{\P_*}}}}="p6" [l] {1}="p7" "p0":@{<-}"p1"^-{}:"p2"^-{}:"p3"^-{}:"p4"^-{}:@{<-}"p5"^-{}:@{<-}"p6"^-{}:"p7"^-{}:@{<-}"p0"^-{} "p1":"p6"_-{} "p2":"p5"^-{b} "p1":@{}"p5"|-{\tn{pb}}} \]
of polynomial monads in which $b$ is a discrete fibration, and the objects of $B$ are the operations of $T$. Moreover, one can recover $\Cat$-operads as such polynomial monad morphisms in which $b$ has the structure of a split fibration.

The corresponding polynomial 2-monad on $\Cat/I$ is also denoted by $T$. Given a category $X \to I$ over $I$, which we also regard as an $I$-indexed family $(X_i)_{i{\in}I}$ of categories, $(TX)_i$ is the following category by Lemma 3.10 of \cite{Weber-OpPoly2Mnd}. Its objects are pairs $(\alpha,(x_j)_j)$, where $\alpha : (i_j)_j \to i$ is an arrow of $T$, and $x_j \in X_{i_j}$, and this data can be pictured more intuitively as a labelled operation
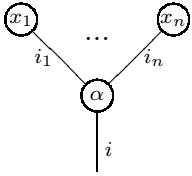
\[ \xygraph{!{0;(.8,0):(0,1)::} 
{\scriptstyle{\alpha}} *\xycircle<6pt,6pt>{-}="p0" [ul]
{\scriptstyle{x_1}} *\xycircle<6pt,6pt>{-}="p1" [r(2)]
{\scriptstyle{x_n}} *\xycircle<6pt,6pt>{-}="p2"
"p0" (-"p1",-"p2",-[d],[u(.75)] {...},[u(.5)l(.7)] {\scriptstyle{i_1}},[u(.5)r(.75)] {\scriptstyle{i_n}},[d(.7)r(.15)] {\scriptstyle{i}})} \]
A morphism $(\alpha,(x_j)_j) \to (\beta,(y_j)_j)$ is a pair $(\rho,(\gamma_j)_j)$ where $\rho$ is a permutation such that $\alpha = \beta\rho$, and $\gamma_j : x_j \to y_{\rho j}$ is a morphism of $X_{\rho j}$ for each $j$.

There are variations on this theme for non-symmetric and for braided operads. For non-symmetric operads one repeats the development of \cite{Weber-OpPoly2Mnd} but omits any mention of permutations, to exhibit non-symmetric operads as morphisms of polynomial monads
\[ \xygraph{!{0;(1.5,0):(0,.6667)::} {I}="p0" [r] {E}="p1" [r] {B}="p2" [r] {I}="p3" [d] {1}="p4" [l] {\N}="p5" [l] *!(0,.025){\xybox{\xygraph{{\N_*}}}}="p6" [l] {1}="p7" "p0":@{<-}"p1"^-{}:"p2"^-{}:"p3"^-{}:"p4"^-{}:@{<-}"p5"^-{}:@{<-}"p6"^-{}:"p7"^-{}:@{<-}"p0"^-{} "p1":"p6"_-{} "p2":"p5"^-{b} "p1":@{}"p5"|-{\tn{pb}}} \]
in which $b$ is a discrete fibration. A general such polynomial monad morphism, that is under no conditions on $b$, is a non-symmetric $\Cat$-operad, since by the discreteness of $\N$, $b$ is automatically a split fibration. In the braided case one imitates the development of \cite{Weber-OpPoly2Mnd} again, but this time replacing permutations by braids, to exhibit braided operads as morphisms of polynomial monads
\[ \xygraph{!{0;(1.5,0):(0,.6667)::} {I}="p0" [r] {E}="p1" [r] {B}="p2" [r] {I}="p3" [d] {1}="p4" [l] {\B}="p5" [l] *!(0,.025){\xybox{\xygraph{{\B_*}}}}="p6" [l] {1}="p7" "p0":@{<-}"p1"^-{}:"p2"^-{}:"p3"^-{}:"p4"^-{}:@{<-}"p5"^-{}:@{<-}"p6"^-{}:"p7"^-{}:@{<-}"p0"^-{} "p1":"p6"_-{} "p2":"p5"^-{b} "p1":@{}"p5"|-{\tn{pb}}} \]
in which $b$ is a discrete fibration and braided $\Cat$-operads as such morphisms together with a cleavage on $b$ making it a split fibration.

\section{Internal algebras}
\label{sec:IntAlg}

The general context of an adjunction of 2-monads, for discussing structures within structures, is defined and discussed in Section \ref{ssec:monad-adjunctions}. Then in Section \ref{ssec:operad-morphism->monad-adjunction} we explain how a morphism of operads gives rise to an adjunction of 2-monads. In Section \ref{ssec:explicit-adj2mnd-op-morphism} we give a more explicit description of this in elementary terms. This last part is somewhat more technical, uses the theory of polynomial functors \cite{GambinoKock-PolynomialFunctors, Weber-PolynomialFunctors} heavily, and plays a role in the proof of Theorem \ref{thm:intalg-classifier-from-operad-morphism} below.

\subsection{Adjunctions of 2-monads.}
\label{ssec:monad-adjunctions}
The idea of considering structures within structures is a fundamental theme in category theory. In \cite{Batanin-EckmannHilton} certain $\Cat$-operads universally possessing certain internal structure, were found to provide some insight into the combinatorics of iterated loop spaces. 

The context giving rise to the internal algebras of \cite{Batanin-EckmannHilton} was monad theoretic. Given a 2-monad on a 2-category $\ca K$ with a terminal object $1$, and a pseudo algebra $A$ of $T$, a \emph{$T$-algebra internal} to $A$ is by definition a lax morphism of $T$-algebras $1 \to A$. When $T = \tnb{M}$, $A$ is a monoidal category, and an $\tnb{M}$-algebra therein is a lax monoidal functor $1 \to A$, which is the same thing as a monoid in $A$. When $T = \tnb{S}$ (resp. $\tnb{B}$), $A$ is a symmetric (resp. braided) monoidal category, and an internal algebra is a commutative monoid in $A$.

With just a single 2-monad $T$ one is restricted to considering structures within structures of the ``same type''. However as is clear from mathematical practise, one would like a viewpoint to encompass a wider variety of internal structure with respect to a given ambient structure. For example, there are many types of structures that may be considered internal to a symmetric monoidal category $\ca V$. Indeed for any operad, one can consider the algebras of it in $\ca V$.

To encompass these situations, the correct monad theoretic context is a general type of morphism $S \to T$ of 2-monads, in which $T$ parametrises the type of ambient structure, and $S$ parametrises the type of internal structure. In fact, the analogous situation for monads in an arbitrary bicategory $\ca B$ will be of use, and so we make
\begin{defn}\label{def:adjunction-of-2-monads}
Let $\ca B$ be a bicategory and $(L,s)$ and $(K,t)$ be monads therein. An \emph{adjunction of monads} $f : (L,s) \to (K,t)$
consists of
\begin{enumerate}
\item an arrow $f_! : L \to K$,
\label{adj-mnd-left-adjoint}
\item a 2-cell $f^c : f_!s \to tf_!$ providing the coherence of a colax monad morphism, and
\label{adj-mnd-colax-coh}
\item a right adjoint $f^* : K \to L$ of $f_!$.
\label{adj-mnd-right-adjoint}
\end{enumerate}
An adjunction of monads $F : (\ca L,S) \to (\ca K,T)$ in $\TwoCAT$ is called an \emph{adjunction of 2-monads}.
\end{defn}
Given the adjunction $f_! \ladj f^*$, 2-cells $f^c : f_!s \to tf_!$ are in bijection with 2-cells $f^l : sf^* \to f^*t$, and $f^c$ satisfies the axioms making $(f_!,f^c)$ a colax monad morphism iff $f^l$ satisfies the axioms making $(f^*,f^l)$ a lax monad morphism. A key feature of Definition \ref{def:adjunction-of-2-monads} is that the monads $s$ and $t$ can act on (possibly) different objects $L$ and $K$ of $\ca B$. However an important special case is given by
\begin{exams}\label{exams:morphisms-of-2-monads}
For an adjunction of monads $f : (L,s) \to (K,t)$ in which $L = K$ and $f_! \ladj f^*$ is the identity adjunction, $f^c = f^l : s \to t$, and the colax monad morphism axioms for $f^c$ say that this common 2-cell $s \to t$ underlies a morphism of monoids in the monoidal category $\ca B(K,K)$ of endomorphisms of $K$.
\end{exams}
In order to exhibit an adjunction of 2-monads $F : (\ca L,S) \to (\ca K,T)$ as the context within which it makes sense to discuss ``$S$-algebras internal to a $T$-algebra'', we require the preliminary
\begin{rem}\label{rem:adjmnd-data-allowing-defn-of-internal-algebra}
Let $F : (\ca L,S) \to (\ca K,T)$ be an adjunction of 2-monads. As we saw above, the colax coherence data $F^c : F_!S \to TF_!$ and the lax coherence data $F^l : SF^* \to F^*T$ determine each other uniquely. By the formal theory of monads \cite{Street-FTM} $F^c$ and $F^l$ are in turn in bijection with extensions $\underline{F}$ of $F_!$ as on the left
\[ \xygraph{{\xybox{\xygraph{!{0;(1.5,0):(0,.6667)::} {\tn{Kl}(S)}="p0" [r] {\tn{Kl}(T)}="p1" [d] {\ca K}="p2" [l] {\ca L}="p3" "p0" :@{.>}"p1"^-{\underline{F}} :@{<-}"p2"^-{F_T} :@{<-}"p3"^-{F_!}  :"p0"^-{F_S}:@{}"p2"|-{=}}}}
[r(4)]
{\xybox{\xygraph{!{0;(1.5,0):(0,.6667)::} {\Algs T}="p0" [r] {\Algs S}="p1" [d] {\ca L}="p2" [l] {\ca K}="p3"
"p0" :@{.>}"p1"^-{\overline{F}} :"p2"^-{U^S} :@{<-}"p3"^-{F^*} :@{<-}"p0"^-{U^T} :@{}"p2"|-{=}}}}} \]
and also with liftings $\overline{F}$ of $F^*$ as on the right. Given a strict $T$-algebra $(A,a)$, the $S$-algebra action $SF^*A \to F^*A$ for $\overline{F}(A,a)$ is given in explicit terms by the composite $F^*(a)F^l_A$. With the effect of $\overline{F}$ on arrows and 2-cells similarly easy to describe, one may verify directly that the lifting $\overline{F}$ extends to any of the other 2-categories of algebras compatibly with the inclusions amongst them.
\end{rem}
\begin{defn}\label{defn:internal-algebra}
Let $F : (\ca L,S) \to (\ca K,T)$ be an adjunction of 2-monads, suppose that $\ca L$ has a terminal object $1$, and let $A$ be a pseudo $T$-algebra. An \emph{$S$-algebra internal to $A$} (relative to $F$) is a lax morphism $1 \to \overline{F}A$ of $S$-algebras. The category of $S$-algebras internal to $A$ is defined to be $\PsAlgl S(1,\overline{F}A)$.
\end{defn}
\begin{exams}\label{exams:M-Sm-Br-intalg}
In the case where $F$ is the identity morphism on $\tnb{M}$, $\tnb{B}$ or $\tnb{S}$, Definition \ref{defn:internal-algebra} gives the category of monoids in a monoidal category, or of commutative monoids in a braided or symmetric monoidal category respectively. There is an evident morphism $\tnb{M} \to \tnb{B}$ (resp. $\tnb{M} \to \tnb{S}$) of 2-monads, and with respect to this Definition \ref{defn:internal-algebra} gives the category of monoids in a braided (resp. symmetric) monoidal category. There is also a morphism $\tnb{B} \to \tnb{S}$ arising from the process of taking the underlying permutation of a braid, and thus a category of $\tnb{B}$-algebras internal to a symmetric monoidal category $\ca V$. However since the induced forgetful functors between 2-categories of $\tnb{S}$ and $\tnb{B}$-algebras are 2-fully-faithful, the category of $\tnb{B}$-algebras internal to $\ca V$ is the same as the category of $\tnb{S}$-algebras internal to $\ca V$, and thus is the category of commutative monoids in $\ca V$.
\end{exams}
There are many important examples in which $\ca L$ and $\ca K$ are different. See for instance Examples \ref{exams:algebras-of-operad-in-V} and \ref{exams:adj-monad-op-morphism} below.

\subsection{Morphisms of operads.}
\label{ssec:operad-morphism->monad-adjunction}
More interesting examples of Definitions \ref{def:adjunction-of-2-monads} and \ref{defn:internal-algebra} arise from operad morphisms. We describe these examples in this section, after recalling some further required background from the theory of polynomial functors \cite{GambinoKock-PolynomialFunctors, Weber-PolynomialFunctors, Weber-OpPoly2Mnd}.

Let $\ca E$ be a category with pullbacks. The bicategory $\Polyc{\ca E}$ has objects those of $\ca E$, and an arrow $I \to J$ in $\Polyc {\ca E}$ is a \emph{polynomial} in $\ca E$ from $I$ to $J$, which by definition is a diagram as on the left
\[ \xygraph{{\xybox{\xygraph{{I}="p0" [r] {E}="p1" [r] {B}="p2" [r] {J}="p3" "p0":@{<-}"p1"^-{s}:"p2"^-{p}:"p3"^-{t}}}}
[r(5)d(.05)]
{\xybox{\xygraph{!{0;(1.5,0):(0,.5)::} {I}="p0" [ur] {E_1}="p1" [r] {B_1}="p2" [dr] {J}="p3" [dl] {B_2}="p4" [l] {E_2}="p5" "p0":@{<-}"p1"^-{s_1}:"p2"^-{p_1}:"p3"^-{t_1}:@{<-}"p4"^-{t_2}:@{<-}"p5"^-{p_2}:"p0"^-{s_2}
"p1":"p5"_-{f_2} "p2":"p4"^-{f_1}
"p1":@{}"p4"|-{\tn{pb}} "p0" [r(.5)] {\scriptstyle{=}} "p3" [l(.5)] {\scriptstyle{=}}}}}} \]
in $\ca E$ in which the middle map $p$ is exponentiable. A 2-cell $f:(s_1,p_1,t_1) \to (s_2,p_2,t_2)$ in $\Polyc {\ca E}$ is a diagram as on the right in the previous display. In elementary terms the process of forming the horizontal composite $(s_3,p_3,t_3) = (s_2,p_2,t_2) \comp (s_1,p_1,t_1)$ of polynomials is encapsulated by the commutative diagram
\[ \xygraph{{I}="b1" [r] {E_1}="b2" [r] {B_1}="b3" [r] {J}="b4" [r] {E_2}="b5" [r] {B_2}="b6" [r] {K.}="b7" "b4" [u] {B_1 \times_J E_2}="p1" [u] {F}="dl" ([r(1.5)] {B_3}="dr", [l(1.5)] {E_3}="p2")
"b1":@{<-}"b2"_-{s_1}:"b3"_-{p_1}:"b4"_-{t_1}:@{<-}"b5"_-{s_2}:"b6"_-{p_2}:"b7"_-{t_2} "dl":"p1"_-{}(:"b3"_-{},:"b5"^-{}) "b2":@{<-}"p2"_-{}:"dl"_-{}:"dr"_-{}:"b6"^(.7){} "b1":@{<-}"p2"^-{s_3} "dr":"b7"^-{t_3} "p2":@/^{1pc}/"dr"^-{p_3}
"b3" [u(1.25)] {\scriptstyle{\tn{pb}}} "b5" [u(1.25)] {\scriptstyle{\tn{dpb}}} "b4" [u(.5)] {\scriptstyle{\tn{pb}}}} \]
The regions labelled ``pb'' are pullbacks, and those labelled ``dpb'' are distributivity pullbacks in the sense of \cite{Weber-PolynomialFunctors}. The universal properties enjoyed by pullbacks and distributivity pullbacks are used in exhibiting aspects of $\Polyc{\ca E}$'s bicategory structure, as well as an explicit description of the homomorphism
\[ \begin{array}{lccr} {\PFun {\ca E} : \Polyc {\ca E} \longrightarrow \CAT}
&&& {I \mapsto \ca E/I} \end{array} \]
of bicategories with object map as indicated. The effect of $\PFun {\ca E}$ on arrows is to send the polynomial $(s,p,t)$ to the composite functor $\Sigma_t\Pi_p\Delta_s : \ca E/I \longrightarrow \ca E/J$.

In particular a \emph{span} is a polynomial whose middle map is an identity, and the composition of such polynomials coincides with the usual pullback-composition of the bicategory $\Span {\ca E}$ of spans in $\ca E$. For $f : I \to J$ in $\ca E$, the spans $f^{\bullet}$ and $f_{\bullet}$ are the polynomials
\[ \xygraph{{\xybox{\xygraph{{I}="p0" [r] {I}="p1" [r] {I}="p2" [r] {J}="p3" "p0":@{<-}"p1"^-{1_I}:"p2"^-{1_I}:"p3"^-{f}}}}
[r(4)]
{\xybox{\xygraph{{J}="p0" [r] {I}="p1" [r] {I}="p2" [r] {I}="p3" "p0":@{<-}"p1"^-{f}:"p2"^-{1_I}:"p3"^-{1_I}}}}} \]
and one has $f^{\bullet} \ladj f_{\bullet}$. The category $\PolyMnd {\ca E}$ had monads in $\Polyc{\ca E}$ as objects, and morphisms are adjunctions of monads in $\Polyc{\ca E}$ in the sense of Definition \ref{def:adjunction-of-2-monads}, in which the underlying adjunction is of the form $f^{\bullet} \ladj f_{\bullet}$. In elementary terms, such a morphism $(I,s,p,t) \to (J,s',p',t')$ amounts to a commutative diagram
\[ \xygraph{!{0;(1.5,0):(0,.6667)::} {I}="p0" [r] {E}="p1" [r] {B}="p2" [r] {I}="p3" [d] {J}="p4" [l] {B'}="p5" [l] {E'}="p6" [l] {J}="p7" "p0":@{<-}"p1"^-{s}:"p2"^-{p}:"p3"^-{t}:"p4"^-{f}:@{<-}"p5"^-{t'}:@{<-}"p6"^-{p'}:"p7"^-{s'}:@{<-}"p0"^-{f} "p1":"p6"_-{f_2} "p2":"p5"^-{f_1} "p1":@{}"p5"|-{\tn{pb}}} \]
compatible in the appropriate sense with the monad structures on $(s,p,t)$ and $(s',p',t')$. Strictly speaking the data of this last diagram is that of a 2-cell
\[ \phi : f^{\bullet} \comp (s_S,p_S,t_S) \comp f_{\bullet} \longrightarrow (s_T,p_T,t_T), \]
in $\Polyc{\ca E}$, and the colax and lax coherence data of the corresponding adjunction of monads are mates of $\phi$ via $f^{\bullet} \ladj f_{\bullet}$. Applying $\PFun {\ca E}$ to such an adjunction of monads in $\Polyc{\ca E}$, produces an adjunction of monads in $\CAT$ whose underlying adjunction is $\Sigma_f \ladj \Delta_f$.

As explained in \cite{Weber-PolynomialFunctors} section 4, the theory of polynomial functors admits an evident 2-categorical analogue. Given a 2-category $\ca K$ with pullbacks, polynomials in $\ca K$ are the 1-cells of the 2-bicategory $\Polyc{\ca K}$. A 2-bicategory is a degenerate sort of tricategory, which is just like a bicategory except that the homs are 2-categories instead of categories. Any notion, such as that of a monad, that makes sense internal to a bicategory also does so internal to a 2-bicategory $\ca M$, since $\ca M$ has an underlying bicategory obtained by forgetting the 3-cells. Taking the associated polynomial 2-functor of a polynomial in $\ca K$ is the effect on 1-cells of a homomorphism $\PFun{\ca K} : \Polyc{\ca K} \to \TwoCAT$. All of our examples in this work take place in the 2-category $\ca K = \Cat$.
\begin{exams}\label{exams:algebras-of-operad-in-V}
As recalled in Section \ref{ssec:polynomials}, given an operad $T$ with set of objects $I$ one has an associated polynomial monad
\[ \xygraph{!{0;(1.5,0):(0,.6667)::} {I}="p0" [r] {E_T}="p1" [r] {B_T}="p2" [r] {I}="p3" "p0":@{<-}"p1"^-{s_T}:"p2"^-{p_T}:"p3"^-{t_T}} \]
and the effect of $\PFun{\Cat}$ on this is a 2-monad on $\Cat/I$ which is also denoted as $T$. This polynomial monad comes with an adjunction of monads
\[ \xygraph{!{0;(1.5,0):(0,.6667)::} {I}="p0" [r] {E_T}="p1" [r] {B_T}="p2" [r] {I}="p3" [d] {1}="p4" [l] {\P}="p5" [l] *!(0,.025){\xybox{\xygraph{{\P_*}}}}="p6" [l] {1}="p7" "p0":@{<-}"p1"^-{}:"p2"^-{}:"p3"^-{}:"p4"^-{}:@{<-}"p5"^-{}:@{<-}"p6"^-{}:"p7"^-{}:@{<-}"p0"^-{} "p1":"p6"_-{} "p2":"p5"^-{b_T} "p0":@{}"p6"|-{=} "p1":@{}"p5"|-{\tn{pb}} "p2":@{}"p4"|-{=}} \]
(in which $b_T$ is a discrete fibration) and the effect of $\PFun{\Cat}$ on this is an adjunction of 2-monads
\[ \ca I_T : (\Cat/I,T) \longrightarrow (\Cat,\tnb{S}). \]
The 2-functor $\ca I_T^* : \Cat \to \Cat/I$ sends $X$ to the constant $I$-indexed family on $X$ which we denote as $X_{\bullet}$. Similarly for a symmetric monoidal category ($=$ pseudo-$\tnb{S}$-algebra) $\ca V$, consistently with \cite{Weber-OpPoly2Mnd} Example 4.6, we denote by $\ca V_{\bullet}$ the pseudo-$T$-algebra $\overline{\ca I}_T(\ca V)$. By Corollary 4.18 of \cite{Weber-OpPoly2Mnd} the category of algebras of $T$ internal to $\ca V$ in the sense of Definition \ref{defn:internal-algebra}, is isomorphic to the category of algebras of the operad $T$ in the symmetric monoidal category $\ca V$ in the usual sense.
\end{exams}
\begin{exams}\label{exams:Batanin-Berger-monad-adj}
The monad morphisms of \cite{BataninBerger-HtyThyOfAlgOfPolyMnd} Section 5.5 are adjunctions of monads in $\CAT$ in which the monads are finitary and the categories on which they act are cocomplete. The examples all arise by applying $\PFun{\Set}$ to adjunctions of monads
\[ \xygraph{!{0;(1.5,0):(0,.6667)::} {I}="p0" [r] {E}="p1" [r] {B}="p2" [r] {I}="p3" [d] {J}="p4" [l] {B'}="p5" [l] {E'}="p6" [l] {J}="p7" "p0":@{<-}"p1"^-{s}:"p2"^-{p}:"p3"^-{t}:"p4"^-{f}:@{<-}"p5"^-{t'}:@{<-}"p6"^-{p'}:"p7"^-{s'}:@{<-}"p0"^-{f} "p1":"p6"_-{f_2} "p2":"p5"^-{f_1} "p0":@{}"p6"|-{=} "p1":@{}"p5"|-{\tn{pb}} "p2":@{}"p4"|-{=}} \]
in $\Polyc{\Set}$. One obtains finitary monads on $\Set/I$ and $\Set/J$ in such examples because the middle maps $p$ and $p'$ have finite fibres. These adjunctions of monads are then regarded as adjunctions of 2-monads via the 2-functorial process $\ca E \mapsto \tn{Cat}(\ca E)$ which sends a category $\ca E$ with pullbacks to the 2-category of categories internal to $\ca E$.
\end{exams}
\begin{rem}\label{rem:Bat-Ber-alg-in-smc}
In Example \ref{exams:Batanin-Berger-monad-adj} one applied $\PFun{\Set}$ and then $\ca E \mapsto \tn{Cat}(\ca E)$ to a morphism of polynomial monads in $\Set$. However,
another way to see this two stage process is as the application of
\[ \xygraph{!{0;(2,0):(0,1)::} {\Polyc{\Set}}="p0" [r] {\Polyc{\Cat}}="p1" [r] {\TwoCAT}="p2" "p0":"p1"^-{}:"p2"^-{\PFun{\Cat}}} \]
in which the first arrow denotes the inclusion of polynomials in $\Set$ as componentwise-discrete polynomials in $\Cat$. However to consider algebras of the examples of \cite{BataninBerger-HtyThyOfAlgOfPolyMnd} in a symmetric monoidal category, one must use a different process. Namely, by \cite{KockJ-PolyFunTrees, SzawielZawadowski-TheoriesOfAnalyticMonads} a polynomial monad in $\Set$ whose middle map has finite fibres can be identified as a $\Sigma$-free operad, which is then interpretted as a categorical polynomial monad as in Example \ref{exams:algebras-of-operad-in-V}. For the relation between these two viewpoints, see \cite{Weber-OpPoly2Mnd} section 6.
\end{rem}
\begin{exams}\label{exams:adj-monad-op-morphism}
Let $S$ be an operad with object set $I$, $T$ be an operad with object set $J$, and $F : S \to T$ be a morphism of operads with underlying object function $f : I \to J$. Applying the functor $\overline{\ca N}$ of \cite{Weber-OpPoly2Mnd} Proposition 3.2 gives a morphism of polynomial monads
\[ \xygraph{!{0;(1.5,0):(0,.6667)::} {I}="p0" [r] {E_S}="p1" [r] {B_S}="p2" [r] {I}="p3" [d] {J.}="p4" [l] {B_T}="p5" [l] {E_T}="p6" [l] {J}="p7" "p0":@{<-}"p1"^-{s_S}:"p2"^-{p_S}:"p3"^-{t_S}:"p4"^-{f}:@{<-}"p5"^-{t_T}:@{<-}"p6"^-{p_T}:"p7"^-{s_T}:@{<-}"p0"^-{f} "p1":"p6"_-{f_2} "p2":"p5"^-{f_1} "p0":@{}"p6"|-{=} "p1":@{}"p5"|-{\tn{pb}} "p2":@{}"p4"|-{=}} \]
and thus an adjunction
\[ (\Cat/I,S) \longrightarrow (\Cat/J,T) \]
of 2-monads with underlying adjunction $\Sigma_f \ladj \Delta_f$. Thus for any morphism $F : S \to T$ of operads, Definition \ref{defn:internal-algebra} gives a notion of $S$-algebra internal to $T$. Example \ref{exams:algebras-of-operad-in-V} is the special case where $T$ is the terminal operad $\Com$.
\end{exams}
\begin{rem}\label{rem:nonsym-braided-internal-algebras}
Examples \ref{exams:algebras-of-operad-in-V} and \ref{exams:adj-monad-op-morphism} have evident non-symmetric and braided analogues by working instead over the polynomial monads
\[ \xygraph{{\xybox{\xygraph{{1}="p0" [r] *!(0,.025){\xybox{\xygraph{{\N_*}}}}="p1" [r] {\N}="p2" [r] {1}="p3" "p0":@{<-}"p1"^-{}:"p2"^-{}:"p3"^-{}}}}
[r(4)]
{\xybox{\xygraph{{1}="p0" [r] *!(0,.025){\xybox{\xygraph{{\B_*}}}}="p1" [r] {\B}="p2" [r] {1}="p3" "p0":@{<-}"p1"^-{}:"p2"^-{}:"p3"^-{}}}}} \]
as at the end of Section \ref{ssec:polynomials}.
\end{rem}

\subsection{Explicit description of the adjunction of 2-monads from an operad morphism.}
\label{ssec:explicit-adj2mnd-op-morphism}
As recalled in Section \ref{ssec:polynomials}, for an operad $T$ with set of colours $J$, following notation 3.7 of \cite{Weber-OpPoly2Mnd}, we also denote by $T$ the induced 2-monad on $\Cat/J$. For a morphism $F : S \to T$ of operads with object map $f : I \to J$ as in Example \ref{exams:adj-monad-op-morphism}, we shall use the analogous
\begin{notn}\label{notn:associated-adjunction-of-2-monads}
Given a morphism $F : S \to T$ of operads with object map $f : I \to J$, we also denote the associated adjunction of 2-monads described in Examples \ref{exams:adj-monad-op-morphism} as
\[ F : (\Cat/I,S) \longrightarrow (\Cat/J,T). \]
Thus in particular, $F_! = \Sigma_f$ and $F^* = \Delta_f$.
\end{notn}
A complete explicit description of the 2-monad associated to an operad was given at the end of Section 3 of \cite{Weber-OpPoly2Mnd}. We shall now extend these calculations to give an explicit description of the adjunction of 2-monads $F$ of Notation \ref{notn:associated-adjunction-of-2-monads}. This task is to describe the colax $F^c : \Sigma_fS \to T\Sigma_f$ and lax $F^l : S\Delta_f \to \Delta_fT$ monad morphism coherence 2-natural transformations, in terms of the data of the operad morphism $F : S \to T$. These calculations will be used in Section \ref{sec:examples}.

The outcome of these calculations, described below in Lemmas \ref{lem:Fc-explicit} and \ref{lem:Fl-explicit}, are as one would predict given the explicit descriptions obtained in \cite{Weber-OpPoly2Mnd} Section 3. We present these results first, and then embark on their proofs in a fairly technical discussion in which the material and notations of \cite{Weber-PolynomialFunctors} and Section 2 of \cite{Weber-OpPoly2Mnd} are treated as assumed knowledge. This more technical discussion, which is here for the sake of rigour, can be omitted without affecting the understandability of the rest of the paper.

Given $X \in \Cat/I$ and $j \in J$, $(\Sigma_fSX)_j = \coprod_{fi=j} (SX)_i$. Thus an object of $(\Sigma_fSX)_j$ consists of $(\alpha,(x_k)_k)$ where $\alpha : (i_k)_{1{\leq}k{\leq}n} \to i$ is in $S$, $fi = j$ and $x_k \in X_{i_k}$. A morphism $(\alpha,(x_k)_k) \to (\alpha',(x'_k)_k)$ consists of $(\rho,\gamma_k)$ where $\rho \in \Sigma_n$, $i_k = i'_{\rho k}$ for all $k$, and $\gamma_k : x_k \to x'_{\rho k}$ is in $X_{i_k}$. An object of $(T\Sigma_fX)_j$ consists of $(\beta,(i_k,x_k)_{1{\leq}k{\leq}n})$ where $\beta : (j_k)_k \to j$ is in $T$, $i_k \in I$ such that $fi_k = j_k$, and $x_k \in X_{i_k}$. A morphism $(\beta,(i_k,x_k)_k) \to (\beta',(i'_k,x'_k)_k)$ consists of $(\rho,(\gamma_k)_k)$ where $\rho \in \Sigma_n$ such that $\beta = \beta'\rho$ and $i_k = i'_{\rho k}$, and $\gamma_k : x_k \to x'_{\rho k} \in X_{i_k}$.
\begin{lem}\label{lem:Fc-explicit}
For $X \in \Cat/I$ and $j \in J$, $F^c_{X,j}$ is given explicitly as
\[ \begin{array}{lccr} {F^c_{X,j}(\alpha,(x_k)_k) = (F\alpha,(i_k,x_k)_k)} &&&
{F^c_{X,j}(\rho,(\gamma_k)_k) = (\rho,(\gamma_k)_k)} \end{array} \]
where $\alpha : (i_k)_{1{\leq}k{\leq}n} \to i$ is in $S$, $x_k \in X_{i_k}$, $\rho \in \Sigma_n$, and $\gamma_k : x_k \to x'_{\rho k}$ is in $X_{i_k}$ for $1 \leq k \leq n$.
\end{lem}
In terms of labelled trees, an object of $(\Sigma_fSX)_j$ is as depicted on the left
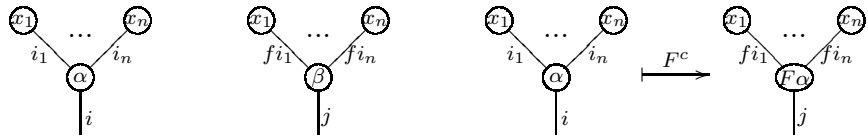
\[ \xygraph{{\xybox{\xygraph{!{0;(.6,0):(0,1)::} 
{\scriptstyle{\alpha}} *\xycircle<5pt,5pt>{-}="p0" [ul]
{\scriptstyle{x_1}} *\xycircle<5pt,5pt>{-}="p1" [r(2)]
{\scriptstyle{x_n}} *\xycircle<5pt,5pt>{-}="p2"
"p0" (-"p1",-"p2",-[d],[u(.75)] {...},[u(.4)l(.7)] {\scriptstyle{i_1}},[u(.4)r(.75)] {\scriptstyle{i_n}},[d(.7)r(.15)] {\scriptstyle{i}})}}}
[r(2.5)]
{\xybox{\xygraph{!{0;(.6,0):(0,1)::} 
{\scriptstyle{\beta}} *\xycircle<5pt,5pt>{-}="p0" [ul]
{\scriptstyle{x_1}} *\xycircle<5pt,5pt>{-}="p1" [r(2)]
{\scriptstyle{x_n}} *\xycircle<5pt,5pt>{-}="p2"
"p0" (-"p1",-"p2",-[d],[u(.75)] {...},[u(.4)l(.7)] {\scriptstyle{fi_1}},[u(.4)r(.75)] {\scriptstyle{fi_n}},[d(.7)r(.15)] {\scriptstyle{j}})}}}
[r(2.5)]
{\xybox{\xygraph{!{0;(.6,0):(0,1)::} 
{\scriptstyle{\alpha}} *\xycircle<5pt,5pt>{-}="p0" [ul]
{\scriptstyle{x_1}} *\xycircle<5pt,5pt>{-}="p1" [r(2)]
{\scriptstyle{x_n}} *\xycircle<5pt,5pt>{-}="p2"
"p0" (-"p1",-"p2",-[d],[u(.75)] {...},[u(.4)l(.7)] {\scriptstyle{i_1}},[u(.4)r(.75)] {\scriptstyle{i_n}},[d(.7)r(.15)] {\scriptstyle{i}})}}}
:@{|->}[r(2.5)]
{\xybox{\xygraph{!{0;(.6,0):(0,1)::} 
{\scriptstyle{F\alpha}} *\xycircle<7pt,5pt>{-}="p0" [ul]
{\scriptstyle{x_1}} *\xycircle<5pt,5pt>{-}="p1" [r(2)]
{\scriptstyle{x_n}} *\xycircle<5pt,5pt>{-}="p2"
"p0" (-"p1",-"p2",-[d],[u(.75)] {...},[u(.4)l(.7)] {\scriptstyle{fi_1}},[u(.4)r(.75)] {\scriptstyle{fi_n}},[d(.7)r(.15)] {\scriptstyle{j}})}}}^-{F^c}} \]
where $fi=j$, an object of $(T\Sigma_fX)_j$ is as depicted in the middle, and the effect of $F^c$ on objects is as depicted on the right. In this display, $x_k \in X_{i_k}$ throughout.

Given $Y \in \Cat/J$ and $i \in I$, an object of $(S\Delta_fY)_i$ consists of $(\alpha,(y_k)_k)$ where $\alpha : (i_k)_{1{\leq}k{\leq}n} \to i$ is in $S$ and $y_k \in Y_{fi_k}$. A morphism $(\alpha,(y_k)_k) \to (\alpha',(y'_k)_k)$ consists of $(\rho,(\delta_k)_k)$ where $\rho \in \Sigma_n$ and $\delta_k : y_k \to y'_{\rho k} \in Y_{fi_k}$. Since $(\Delta_fTY)_i = (TY)_{fi}$, an object of $(\Delta_fTY)_i$ consists of $(\beta,(y_k)_k)$ where $\beta : (j_k)_k \to fi$ is in $T$ and $y_k \in Y_{j_k}$. A morphism $(\beta,(y_k)_k) \to (\beta',(y'_k)_k)$ consists of $(\rho,(\delta_k)_k)$, where $\rho \in \Sigma_n$ is such that $\beta = \beta'\rho$, and $\delta_k : y_k \to y'_{\rho k}$ is in $Y_{j_k}$ for all $k$.
\begin{lem}\label{lem:Fl-explicit}
For $Y \in \Cat/J$ and $i \in I$, $F^l_{Y,i}$ is given explicitly as
\[ \begin{array}{lccr} {F^l_{Y,i}(\alpha,(y_k)_k) = (F\alpha,(y_k)_k)} &&& {F^l_{Y,i}(\rho,(\delta_k)_k) = (\rho,(\delta_k)_k)} \end{array} \]
where $\alpha : (i_k)_{1{\leq}k{\leq}n} \to i$ is in $S$, $y_k \in Y_{i_k}$, $\rho \in \Sigma_n$ and $\delta_k : y_k \to y'_{\rho k} \in Y_{i_k}$.
\end{lem}
In terms of labelled trees, an object of $(S\Delta_fY)_i$ is as depicted on the left
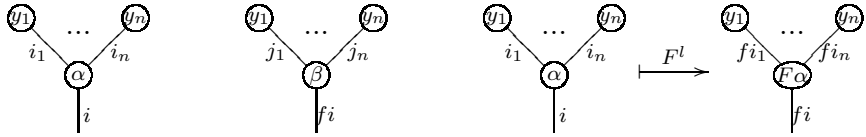
\[ \xygraph{{\xybox{\xygraph{!{0;(.6,0):(0,1)::} 
{\scriptstyle{\alpha}} *\xycircle<5pt,5pt>{-}="p0" [ul]
{\scriptstyle{y_1}} *\xycircle<5pt,5pt>{-}="p1" [r(2)]
{\scriptstyle{y_n}} *\xycircle<5pt,5pt>{-}="p2"
"p0" (-"p1",-"p2",-[d],[u(.75)] {...},[u(.4)l(.7)] {\scriptstyle{i_1}},[u(.4)r(.75)] {\scriptstyle{i_n}},[d(.7)r(.15)] {\scriptstyle{i}})}}}
[r(2.5)]
{\xybox{\xygraph{!{0;(.6,0):(0,1)::} 
{\scriptstyle{\beta}} *\xycircle<5pt,5pt>{-}="p0" [ul]
{\scriptstyle{y_1}} *\xycircle<5pt,5pt>{-}="p1" [r(2)]
{\scriptstyle{y_n}} *\xycircle<5pt,5pt>{-}="p2"
"p0" (-"p1",-"p2",-[d],[u(.75)] {...},[u(.4)l(.7)] {\scriptstyle{j_1}},[u(.4)r(.75)] {\scriptstyle{j_n}},[d(.7)r(.15)] {\scriptstyle{fi}})}}}
[r(2.5)]
{\xybox{\xygraph{!{0;(.6,0):(0,1)::} 
{\scriptstyle{\alpha}} *\xycircle<5pt,5pt>{-}="p0" [ul]
{\scriptstyle{y_1}} *\xycircle<5pt,5pt>{-}="p1" [r(2)]
{\scriptstyle{y_n}} *\xycircle<5pt,5pt>{-}="p2"
"p0" (-"p1",-"p2",-[d],[u(.75)] {...},[u(.4)l(.7)] {\scriptstyle{i_1}},[u(.4)r(.75)] {\scriptstyle{i_n}},[d(.7)r(.15)] {\scriptstyle{i}})}}}
:@{|->}[r(2.5)]
{\xybox{\xygraph{!{0;(.6,0):(0,1)::} 
{\scriptstyle{F\alpha}} *\xycircle<7pt,5pt>{-}="p0" [ul]
{\scriptstyle{y_1}} *\xycircle<5pt,5pt>{-}="p1" [r(2)]
{\scriptstyle{y_n}} *\xycircle<5pt,5pt>{-}="p2"
"p0" (-"p1",-"p2",-[d],[u(.75)] {...},[u(.4)l(.7)] {\scriptstyle{fi_1}},[u(.4)r(.75)] {\scriptstyle{fi_n}},[d(.7)r(.15)] {\scriptstyle{fi}})}}}^-{F^l}} \]
where $y_k \in Y_{fi_k}$, an object of $(T\Sigma_fX)_j$ is as depicted in the middle where $y_k \in Y_{j_k}$, and the effect of $F^l$ on objects is as depicted on the right.

We turn now to our more technical discussion. Let $\ca E$ be a category with pullbacks. Given polynomial endomorphisms $P = (s,p,t)$ and $Q = (s',p',t')$ on $I$ and $J$, recall \cite{GambinoKock-PolynomialFunctors, Weber-OpPoly2Mnd} that a morphism between them in the category $\PolyEnd{\ca E}$, is a pair $(f,\phi)$ where $f : I \to J$ is in $\ca E$ and $\phi : f^{\bullet} \comp P \comp f_{\bullet} \to Q$ is in $\Polyc{\ca E}$, and that in more elementary terms this amounts to a commutative diagram
\begin{equation}\label{diag:PolyEnd-morphism} 
\xygraph{!{0;(1.5,0):(0,.6667)::} {I}="p0" [r] {E}="p1" [r] {B}="p2" [r] {I}="p3" [d] {J.}="p4" [l] {B'}="p5" [l] {E'}="p6" [l] {J}="p7" "p0":@{<-}"p1"^-{s}:"p2"^-{p}:"p3"^-{t}:"p4"^-{f}:@{<-}"p5"^-{t'}:@{<-}"p6"^-{p'}:"p7"^-{s'}:@{<-}"p0"^-{f} "p1":"p6"_-{f_2} "p2":"p5"^-{f_1} "p1":@{}"p5"|-{\tn{pb}}} 
\end{equation}
We denote by
\[ \begin{array}{lccr} {\phi^c : f^{\bullet} \comp P \to Q \comp f^{\bullet}} &&&
{\phi^l : P \comp f_{\bullet} \to f_{\bullet} \comp Q} \end{array} \]
the 2-cells which correspond to $\phi$ via the adjunction $f^{\bullet} \ladj f_{\bullet}$. The problem of verifying Lemmas \ref{lem:Fc-explicit} and \ref{lem:Fl-explicit} comes down have an explicit description of
\[ \begin{array}{lccr} {\PFun{\ca E}(\phi^c) : \Sigma_f\PFun{\ca E}(P) \to \PFun{\ca E}(Q)\Sigma_f} &&&
{\PFun{\ca E}(\phi^l) : \PFun{\ca E}(P)\Delta_f \to \Delta_f\PFun{\ca E}(Q)} \end{array} \]
in the cases of interest for us. We shall obtain such an explicit description at this generality in Lemma \ref{lem:explicit-colax-lax-monad-morphism-coherence}, in terms of morphisms in $\ca E$ induced by pullbacks and distributivity pullbacks.

Some preliminary remarks regarding distributivity pullbacks are required. Given a pullback square in $\ca E$ as on the left
\[ \xygraph{{\xybox{\xygraph{{P}="p0" [r] {B}="p1" [d] {C}="p2" [l] {A}="p3" "p0":"p1"^-{q}:"p2"^-{g}:@{<-}"p3"^-{f}:@{<-}"p0"^-{p}:@{}"p2"|-{\tn{pb}}}}}
[r(3)]
{\xybox{\xygraph{!{0;(1.25,0):(0,.8)::} {\ca E/P}="p0" [r] {\ca E/B}="p1" [d] {\ca E/C}="p2" [l] {\ca E/A}="p3" "p0":@{<-}"p1"^-{\Delta_q}:"p2"^-{\Sigma_g}:"p3"^-{\Delta_f}:@{<-}"p0"^-{\Sigma_p}:@{}"p2"|-*{\iso}}}}
[r(3)]
{\xybox{\xygraph{!{0;(1.25,0):(0,.8)::} {\ca E/P}="p0" [r] {\ca E/B}="p1" [d] {\ca E/C}="p2" [l] {\ca E/A}="p3" "p0":"p1"^-{\Pi_q}:@{<-}"p2"^-{\Delta_g}:@{<-}"p3"^-{\Pi_f}:"p0"^-{\Delta_p}:@{}"p2"|-*{\iso}}}}} \]
one has the left and right Beck-Chevalley isomorphisms $\Sigma_p\Delta_q \iso \Delta_f\Sigma_g$ and $\Pi_q\Delta_p \iso \Delta_g\Pi_f$. In elementary terms $\Sigma_p\Delta_q \iso \Delta_f\Sigma_g$ witnesses the fact that if one pastes a square on top of the original pullback, then this square is a pullback iff the composite of this square with the original pullback is a pullback. In other words the left Beck-Chevalley isomorphisms are a reformulation of the elementary composability of pullbacks in a category $\ca E$ which admits all pullbacks. Similarly
\begin{lem}\label{lem:right-BC-elementary}
In a category with pullbacks, given the solid parts of the diagram on the left
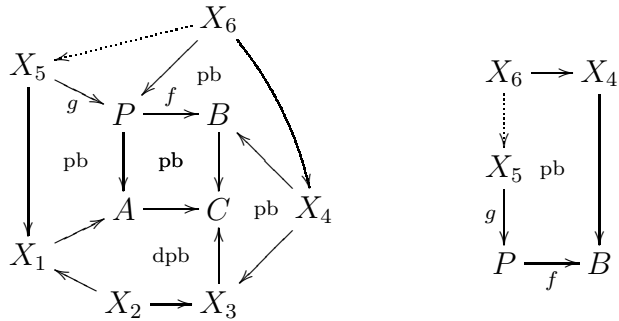
\[ \xygraph{{\xybox{\xygraph{{P}="p0" [r] {B}="p1" [d] {C}="p2" [l] {A}="p3" "p0":"p1"^-{f}:"p2"^-{}:@{<-}"p3"^-{}:@{<-}"p0"^-{}:@{}"p2"|-{\tn{pb}}
"p3" [d(.5)l(1)] {X_1}="q0" [d(.5)r(1)] {X_2}="q1" [r] {X_3}="q2" "p2" [r] {X_4}="q3" "p0" [u(.5)l(1)] {X_5}="q4" "p1" [u] {X_6}="q5"
"q1"(:"q0":"p3",:"q2":"p2") "q3"(:"q2",:"p1") "q5"(:"p0",:@/^{.5pc}/"q3") "q4"(:"p0"_-{g},:"q0")
"p0" [d(.5)r(.5)] {\scriptsize{\tn{pb}}} "p3" [d(.5)r(.5)] {\scriptsize{\tn{dpb}}} "q3" [l(.5)] {\scriptsize{\tn{pb}}} "p3" [u(.5)l(.5)] {\scriptsize{\tn{pb}}} "p1" [u(.4)l(.1)] {\scriptsize{\tn{pb}}} "q5":@{.>}"q4"}}}
[r(4)]
{\xybox{\xygraph{{X_6}="p0" [r] {X_4}="p1" [d(2)] {B}="p2" [l] {P}="p3" [u] {X_5}="p4" "p0":"p1"^-{}:"p2"^-{}:@{<-}"p3"^-{f}:@{<-}"p4"^-{g}:@{<.}"p0"^-{}:@{}"p2"|-{\tn{pb}}}}}} \]
one can factor the morphism $X_6 \to P$ through $g$ as shown in such a way as to make the pullback on the right into a distributivity pullback around $(f,g)$.
\end{lem}
\noindent is an elementary formulation of the right Beck-Chevalley isomorphisms. We leave the task of giving a direct proof of this lemma using the universal properties of pullbacks and distributivity pullbacks as an instructive exercise for the reader. Another elementary fact which we shall use below is
\begin{lem}\label{lem:vertical-comp-dpb}
If in a category with pullbacks one has
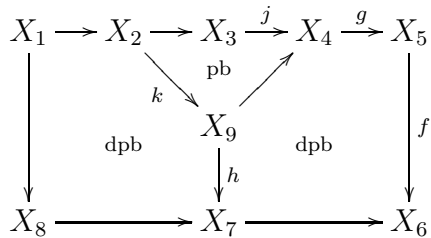
\[ \xygraph{{X_1}="p0" [r] {X_2}="p1" [r] {X_3}="p2" [r] {X_4}="p3" [r] {X_5}="p4" [d(2)] {X_6}="p5" [l(2)] {X_7}="p6" [l(2)] {X_8}="p7" [r(2)u] {X_9}="p8" "p0":"p1"^-{}:"p2"^-{}:"p3"^-{j}:"p4"^-{g}:"p5"^-{f}:@{<-}"p6"^-{}:@{<-}"p7"^-{}:@{<-}"p0"^-{}
"p8" (:@{<-}"p1"^-{k},:"p3",:"p6"^-{h},[d(.2)l] {\scriptsize{\tn{dpb}}},[d(.2)r] {\scriptsize{\tn{dpb}}}, [u(.6)] {\scriptsize{\tn{pb}}})} \]
in which the left distributivty pullback is around $(h,k)$ and the right distributivty pullback is around $(f,g)$, then the composite diagram is a distributivity pullback around $(f,gj)$.
\end{lem}
\noindent whose proof is also a straight forward exercise. This is the elementary counterpart of the fact, described in Proposition 2.2.3 of \cite{Weber-PolynomialFunctors}, that for a distibutivity pullback
\[ \xygraph{{P}="p0" [r] {W}="p1" [r] {X}="p2" [d] {Y}="p3" [l(2)] {Q}="p4" "p0":"p1"^-{p}:"p2"^-{g}:"p3"^-{f}:@{<-}"p4"^-{r}:@{<-}"p0"^-{q}:@{}"p3"|-{\tn{dpb}}} \]
the canonical natural transformation $\Sigma_r\Pi_q\Delta_p \to \Pi_f\Sigma_g$ is an isomorphism. Other elementary facts concerning distributivity pullbacks were given in section 2.2 of \cite{Weber-PolynomialFunctors}.

To get an elementary description of $\phi^c$, form the diagram on the left
\begin{equation}\label{eq:def-phic-phil} 
\xygraph{{\xybox{\xygraph{!{0;(.8,0):(0,1)::} {I}="p0" [ur] {F_2}="p1" [r] {F_3}="p2" [d(2)r] {J}="p3" [l] {B'}="p4" [l] {E'}="p5" [l] {J}="p6" [ur] {F_1}="p7" "p0":@{<-}"p1"^-{s''}:"p2"^-{p''}:"p3"^-{t''}:@{<-}"p4"^-{t'}:@{<-}"p5"^-{p'}:"p6"^-{s'}:@{<-}"p0"^-{f} "p7"(:"p0",:"p5",:@{<-}"p1",[d(.5)l(.5)] {\scriptsize{\tn{pb}}},[r(.5)d(.25)] {\scriptsize{\tn{dpb}}}) "p2":"p4"}}}
[r(3.25)]
{\xybox{\xygraph{!{0;(.8,0):(0,1)::} {I}="p0" [ur] {F_2}="p1" [r] {F_3}="p2" [d(2)r] {J}="p3" [l] {B'}="p4" [l] {E'}="p5" [l] {J}="p6" [ur] {F_1}="p7" "p0":@{<-}"p1"^-{}:"p2"^-{}:"p3"^-{}:@{<-}"p4"^-{t'}:@{<-}"p5"^-{p'}:"p6"^-{s'}:@{<-}"p0"^-{f} "p7"(:"p0",:"p5",:@{<-}"p1",[d(.5)l(.5)] {\scriptsize{\tn{pb}}},[r(.5)d(.25)] {\scriptsize{\tn{dpb}}}) "p2":"p4"
"p1" [u(1.5)l] {I}="q0" [r] {E}="q1" [r] {B}="q2" [r] {I}="q3"
"q0":@{<-}"q1"^-{s}:"q2"^-{p}:"q3"^-{t} "q0":"p0"_-{1_I} "q1":@/_{1pc}/"p5"_(.25){f_2} "q2":@/^{1pc}/"p4"^(.25){f_1} "q3":"p3"^-{f} "q1":@{.>}"p1"|-{\phi^c_2} "q2":@{.>}"p2"|-{\phi^c_1}}}}
[r(3.25)]
{\xybox{\xygraph{!{0;(.8,0):(0,1.5)::} {I}="p0" [r] {E}="p1" [r] {B}="p2" [r] {I}="p3" [d(2)] {J}="p4" [l] {B'}="p5" [l] {E'}="p6" [l] {J}="p7" "p0":@{<-}"p1"^-{s}:"p2"^-{p}:"p3"^-{t}:"p4"^-{f}:@{<-}"p5"^-{t'}:@{<-}"p6"^-{p'}:"p7"^-{s'}:@{<-}"p0"^-{f} "p1":"p6"_-{f_2} "p2":"p5"_(.75){f_1}
"p1" [r(.5)d] {G_2}="q0" [r] {G_1}="q1" "q0"(:"p6",:"q1"(:"p5",:"p3",:@{<.}"p2"|-{\phi^l_1}),:@{<.}"p1"|-{\phi^l_2},:"p7")}}}}
\end{equation}
and it is easily verified that $Q \comp f^{\bullet}$ is the polynomial $(s'',p'',t'')$, and that the morphisms $F_2 \to E'$ and $F_3 \to B'$ are the components of $Q \comp c_f$, where $c_f$ is the counit of $f^{\bullet} \ladj f_{\bullet}$, and is given explicitly by
\[ \xygraph{!{0;(1.5,0):(0,.5)::}
{J}="p0" [ur] {I}="p1" [r] {I}="p2" [dr] {J.}="p3" [dl] {J}="p4" [l] {J}="p5" "p0":@{<-}"p1"^-{f}:"p2"^-{1_I}:"p3"^-{f}:@{<-}"p4"^-{1_J}:@{<-}"p5"^-{1_J}:"p0"^-{1_J}
"p1":"p5"_-{f} "p2":"p4"^-{f}
"p1":@{}"p4"|-{\tn{pb}} "p0" [r(.5)] {\scriptstyle{=}} "p3" [l(.5)] {\scriptstyle{=}}} \]
Since the equation $\phi = (Q \comp c_f)(\phi^c \comp f_{\bullet})$ determines $\phi^c$ uniquely, the components of $\phi^c$ are induced as in the diagram in the middle of (\ref{eq:def-phic-phil}). Similarly the components of $\phi^l$ are induced in the diagram on the right of (\ref{eq:def-phic-phil}), in which the squares with vertices $(G_1,I,J,B')$ and $(G_2,G_1,B',E')$ are pullback squares.

Given $x:X \to I$ and the data (\ref{diag:PolyEnd-morphism}) one can form the commutative diagram on the left
\begin{equation}\label{diag:monad-morphisms-coherences-explicit}
\xygraph{{\xybox{\xygraph{!{0;(1,0):(0,.6667)::} {I}="p0" [r] {E}="p1" [r] {B}="p2" [r] {I}="p3" [d] {J}="p4" [l] {B'}="p5" [l] {E'}="p6" [l] {J}="p7" "p0":@{<-}"p1"^-{s}:"p2"^-{p}:"p3"^-{t}:"p4"^-{f}:@{<-}"p5"^-{t'}:@{<-}"p6"^-{p'}:"p7"^-{s'}:@{<-}"p0"^-{f} "p1":"p6"_-{f_2} "p2":"p5"^-{f_1} "p1":@{}"p5"|-{\tn{pb}}
"p0" [u] {X}="q0" [r] {X_2}="q1" [u] {X_3}="q2" [r] {X_4}="q3"
"q2" (:"q1"(:"q0":"p0"_-{x},:"p1"),:"q3"(:"p2",:"p3"))
"q0" [d(.4)r(.5)] {\scriptsize{\tn{pb}}} "q1" [u(.1)r(.5)] {\scriptsize{\tn{dpb}}}
"p7" [d] {X}="r0" [r] {X'_2}="r1" [d] {X'_3}="r2" [r] {X'_4}="r3"
"r2" (:"r1"(:"r0":"p7"^-{fx},:"p6"),:"r3"(:"p5",:"p4"))
"r0" [u(.4)r(.5)] {\scriptsize{\tn{pb}}} "r1" [d(.1)r(.5)] {\scriptsize{\tn{dpb}}}
"q1":@{.>}@/_{1pc}/"r1" "q2":@{.>}@/^{1pc}/"r2" "q3":@{.>}@/^{1pc}/"r3"}}}
[r(5)]
{\xybox{\xygraph{!{0;(1,0):(0,.6667)::} {I}="p0" [r] {E}="p1" [r] {B}="p2" [r] {I}="p3" [d] {J}="p4" [l] {B'}="p5" [l] {E'}="p6" [l] {J}="p7" "p0":@{<-}"p1"^-{s}:"p2"^-{p}:"p3"^-{t}:"p4"_-{f}:@{<-}"p5"^-{t'}:@{<-}"p6"^-{p'}:"p7"^-{s'}:@{<-}"p0"_-{f} "p1":"p6"_-{f_2} "p2":"p5"^-{f_1} "p1":@{}"p5"|-{\tn{pb}}
"p0" [l] {Y_1}="q1" "p7" [l] {Y}="q0" "p0" [u] {Y_2}="q2" "p1" [u(2)] {Y_3}="q3" "p2" [u(2)] {Y_4}="q4" 
"q1" (:"q0":"p7"|-{y},:"p0") "q2"(:"q1",:"p1") "q3"(:"q2",:"q4"(:"p2",:"p3"))
"q1":@{}"p7"|-{\tn{pb}} "q2":@{}"p0"|-{\tn{pb}} "p1" [r(.25)u] {\scriptsize{\tn{dpb}}}
"p7" [d] {Y'_1}="r2" "p6" [d(2)] {Y'_2}="r3" "p5" [d(2)] {Y'_3}="r4" "p4" [r] {Y'_4}="r5"
"r2"(:"q0",:"p6") "r3"(:"r2",:"r4"(:"p5",:"p4")) "r5"(:"r4",:"p3")
"r2":@{}"p7"|-{\tn{pb}} "p6" [r(.25)d] {\scriptsize{\tn{dpb}}} "p4" [r(.5)] {\scriptsize{\tn{pb}}}
"q2":@/_{1pc}/@{.>}"r2" "q3":@/_{1pc}/@{.>}"r3" "q4":@/_{1pc}/@{.>}"r4" "q4":@/^{1pc}/@{.>}"r5"}}}}
\end{equation}
in which one induces $X_2 \to X'_2$ using the bottom pullback, and $X_3 \to X'_3$ and $X_4 \to X'_4$ are induced by the bottom distributivity pullback. As such, $X_4 \to X'_4$ is a morphism $\Sigma_f\PFun{\ca E}(P)(x) \to \PFun{\ca E}(Q)\Sigma_f(x)$ in $\ca E/J$. Similarly given $y : Y \to J$ one can form the commmutative diagram on the right in which the dotted arrows are induced in the evident manner, and then $Y_4 \to Y'_4$ is a morphism $\PFun{\ca E}(P)\Delta_f(y) \to \Delta_f\PFun{\ca E}(Q)(y)$ in $\ca E/I$.
\begin{lem}\label{lem:explicit-colax-lax-monad-morphism-coherence}
Given a morphism in $\PolyEnd{\ca E}$ as in (\ref{diag:PolyEnd-morphism}), $x : X \to I$ and $y : Y \to J$, then
\begin{enumerate}
\item The morphism $X_4 \to X'_4$ in (\ref{diag:monad-morphisms-coherences-explicit}) is $\PFun{\ca E}(\phi^c)_x$.
\item The morphism $Y_4 \to Y'_4$ in (\ref{diag:monad-morphisms-coherences-explicit}) is $\PFun{\ca E}(\phi^l)_y$.
\end{enumerate}
\end{lem}
\begin{proof}
By the explicit description of $\phi^c$ and that of the 2-cell map of $\PFun{\ca E}$ given above, $\PFun{\ca E}(\phi^c)_x$ is the induced arrow $X_4 \to X'_7$ in
\begin{equation}\label{diag:phic-proof}
\xygraph{!{0;(1.85,0):(0,.25)::} {I}="p0" [ur] {E}="p1" [r] {B}="p2" [d(2)] {F_3}="p4" [l] {F_2}="p5" [u(.5)l(.5)] {F_1}="p6" "p0":"p1"^-{}:"p2"^-{}:"p4"^-{}:"p5"^-{}:"p6"^-{}:"p0"^-{} "p1":"p5"_-{} "p2":"p4"^{}
"p0" [l] {X}="q0" "p1" [ul] {X_2}="q1" [ur] {X_3}="q2" [r] {X_4}="q3"
"q2" (:"q1"(:"q0":"p0"^-{x},:"p1"),:"q3":"p2")
"p6" [dl] {X'_2}="r0" "p5" [dl] {X'_5}="r1" [dr] {X'_3}="r2" [r] {X'_4}="r3"
"r2" (:"r1"(:"r0"(:"q0",:"p6"),:"p5"),:"r3":"p4")
"q1":@/^{1.5pc}/@{.>}"r1" "q2":@/^{1.5pc}/@{.>}"r2" "q3":@/^{1.5pc}/@{.>}"r3"
"p1" [r(.35)u] {\scriptsize{\tn{dpb}}} [d(2)r(.15)] {\scriptsize{\tn{pb}}} [d(2)l(.15)] {\scriptsize{\tn{dpb}}} "q1" [d] {\scriptsize{\tn{pb}}} [d(1.5)l(.25)]  {\scriptsize{\tn{pb}}} [d(.75)r(.55)] {\scriptsize{\tn{pb}}}}
\end{equation}
noting that $X'_2$ really is the result of pulling back $x$ along $F_1 \to I$ because of the diagram on the left in
\[ \xygraph{{\xybox{\xygraph{{X'_2}="p0" [r] {X}="p1" [d] {I}="p2" [d] {J}="p3" [l] {E'}="p4" [u] {F_1}="p5" "p0":"p1"^-{}:"p2"^-{x}:"p3"^-{f}:@{<-}"p4"^-{}:@{<-}"p5"^-{}:@{<-}"p0"^-{} "p5":"p2"^{} "p0":@{}"p2"|-{\tn{pb}}:@{}"p4"|-{\tn{pb}}}}}
[r(3)]
{\xybox{\xygraph{!{0;(.75,0):(0,1)::} {X'_3}="p0" [r(2)] {X'_4}="p1" [d(2)] {F_3}="p2" [d(2)] {B'}="p3" [l(2)] {E'}="p4" [u] {F_1}="p5" [u] {X'_2}="p6" [u] {X'_5}="p7" [dr] {F_2}="p8" "p0":"p1"^-{}:"p2"^-{}:"p3"^-{}:@{<-}"p4"^-{}:@{<-}"p5"^-{}:@{<-}"p6"^-{}:@{<-}"p7"^-{}:@{<-}"p0"^-{}
"p8" (:@{<-}"p7",:"p2",:"p5", [u(1)r(.25)] {\scriptsize{\tn{dpb}}}, [d(1)r(.25)] {\scriptsize{\tn{dpb}}}, [l(.5)] {\scriptsize{\tn{pb}}})}}}
[r(3)]
{\xybox{\xygraph{{X'_3}="p0" [r] {X'_4}="p1" [d(2)] {B'}="p2" [l] {E'}="p3" [u] {X'_2}="p4" "p0":"p1"^-{}:"p2"^-{}:@{<-}"p3"^-{}:@{<-}"p4"^-{}:@{<-}"p0"^-{}
"p4" [r(.5)u(.1)] {\scriptsize{\tn{dpb}}}}}}} \]
and $X'_3$ and $X'_4$ really are the result of taking a distributivity pullback around $X'_2 \to E' \to B'$, because by Lemma \ref{lem:vertical-comp-dpb}, one can identify the diagrams in the middle and right in the previous display. One can then check that the composite $X_2 \to X'_5 \to X'_2$ in (\ref{diag:phic-proof}) coincides with the morphism $X_2 \to X'_2$ of (\ref{diag:monad-morphisms-coherences-explicit}), by verifying that these coincide after composing with the pullback projections $X \leftarrow X'_2 \to E'$. From this one can then reconcile the morphisms $X_3 \to X'_3$ and $X_4 \to X'_4$ in (\ref{diag:monad-morphisms-coherences-explicit}) and (\ref{diag:phic-proof}) using the universal property of the distributivity pullback on the right in the previous display, thus establishing our explicit description of $\PFun{\ca E}(\phi^c)_x$.

By the explicit description of $\phi^l$ and that of the 2-cell map of $\PFun{\ca E}$, $\PFun{\ca E}(\phi^l)_y$ is the induced arrow
\begin{equation}\label{diag:phil-proof}
\xygraph{!{0;(1.85,0):(0,.25)::} {J}="p0" [u(.5)r(.5)] {I}="p1h" [u(.5)r(.5)] {E}="p1" [r] {B}="p2" [d(2)] {G_1}="p4" [l] {G_2}="p5" [u(.5)l(.5)] {E'}="p6" "p0":"p1h"^-{}:"p1"^-{}:"p2"^-{}:"p4"^-{}:"p5"^-{}:"p6"^-{}:"p0"^-{} "p1":"p5"_-{} "p2":"p4"^{}
"p0" [l] {Y}="q0" "p1h" [ul] {Y_1}="q1h" "p1" [ul] {Y_2}="q1" [ur] {Y_3}="q2" [r] {Y_4}="q3"
"q2" (:"q1"(:"q1h"(:"q0":"p0"^-{y},:"p1h"),:"p1"),:"q3":"p2")
"p6" [dl] {Y'_1}="r0" "p5" [dl] {Y'_5}="r1" [dr] {Y'_6}="r2" [r] {Y'_4}="r3"
"r2" (:"r1"(:"r0"(:"q0",:"p6"),:"p5"),:"r3":"p4")
"q1":@/^{1.5pc}/@{.>}"r1" "q2":@/^{1.5pc}/@{.>}"r2" "q3":@/^{1.5pc}/@{.>}"r3"
"p1" [r(.35)u] {\scriptsize{\tn{dpb}}} [d(2)r(.15)] {\scriptsize{\tn{pb}}} [d(2)l(.15)] {\scriptsize{\tn{dpb}}} "q1" [d(2.5)l(.25)] {\scriptsize{\tn{pb}}} [d(.75)r(.55)] {\scriptsize{\tn{pb}}} "p0" [u(.5)l(.25)] {\scriptsize{\tn{pb}}} "p1h" [u(.6)l(.1)] {\scriptsize{\tn{pb}}}}
\end{equation}
in which the name $Y'_4$ is compatible with the diagram (\ref{diag:monad-morphisms-coherences-explicit}) because by Lemma \ref{lem:right-BC-elementary} one has
\[ \xygraph{!{0;(1,0):(0,.65)::} {Y'_5}="p0" [ur] {Y'_6}="p1" [d(.25)r(1.5)] {Y'_4}="p2" [d(3.5)] {Y'_3}="p3" [d(.25)l(1.5)] {Y'_2}="p4" [ul] {Y'_1}="p5" "p0":@{<-}"p1"^-{}:"p2"^-{}:@/^{3.5pc}/"p3"^-{}:@{<-}"p4"^-{}:"p5"^-{}:@{<-}"p0"^-{}
"p0" [d(.5)r] {G_2}="q0" [r] {G_1}="q1" [r] {I}="q2" [d] {J}="q3" [l] {B'}="q4" [l] {E'}="q5"
"q0":"q1"^-{}:"q2"^-{}:"q3"^-{}:@{<-}"q4"^-{}:@{<-}"q5"^-{}:@{<-}"q0"^-{} "q1":"q4"^{}
"p0":"q0" "p2" (:"q1",:"q2") "p3" (:"q3",:"q4") "p5":"q5"
"q0" [d(.5)l(.5)] {\scriptsize{\tn{pb}}} [r] {\scriptsize{\tn{pb}}} [r] {\scriptsize{\tn{pb}}} [r(.8)] {\scriptsize{\tn{pb}}} "q0" [r(.25)u(.75)] {\scriptsize{\tn{dpb}}} "q5" [r(.25)d(.75)] {\scriptsize{\tn{dpb}}} "p2" [d(.75)] {\scriptsize{=}} "p3" [u(.75)] {\scriptsize{=}}} \]
and so it remains to reconcile the morphisms $Y_4 \to Y'_4$ in (\ref{diag:monad-morphisms-coherences-explicit}) and (\ref{diag:phil-proof}). To this end we collect everything into one diagram
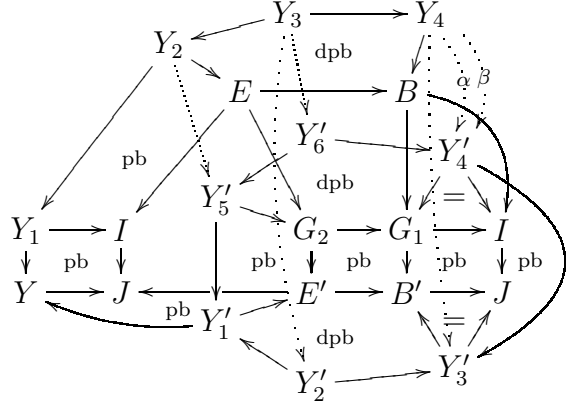
\[ \xygraph{!{0;(1,0):(0,.65)::} {Y'_5}="p0" [ur] {Y'_6}="p1" [d(.25)r(1.5)] {Y'_4}="p2" [d(3.5)] {Y'_3}="p3" [d(.25)l(1.5)] {Y'_2}="p4" [ul] {Y'_1}="p5" "p0":@{<-}"p1"^-{}:"p2"^-{}:@/^{3.5pc}/"p3"^-{}:@{<-}"p4"^-{}:"p5"^-{}:@{<-}"p0"^-{}
"p0" [d(.5)r] {G_2}="q0" [r] {G_1}="q1" [r] {I}="q2" [d] {J}="q3" [l] {B'}="q4" [l] {E'}="q5"
"q0":"q1"^-{}:"q2"^-{}:"q3"^-{}:@{<-}"q4"^-{}:@{<-}"q5"^-{}:@{<-}"q0"^-{} "q1":"q4"^{}
"p0":"q0" "p2" (:"q1",:"q2") "p3" (:"q3",:"q4") "p5":"q5"
"q0" [d(.5)l(.5)] {\scriptsize{\tn{pb}}} [r] {\scriptsize{\tn{pb}}} [r] {\scriptsize{\tn{pb}}} [r(.8)] {\scriptsize{\tn{pb}}} "q0" [r(.25)u(.75)] {\scriptsize{\tn{dpb}}} "q5" [r(.25)d(.75)] {\scriptsize{\tn{dpb}}} "p2" [d(.75)] {\scriptsize{=}} "p3" [u(.75)] {\scriptsize{=}}
"q5" [l(2)] {J}="r0" [l] {Y}="r1" [u] {Y_1}="r2" [r] {I}="r3" "q0" [u(2.25)l(.75)] {E}="r4" "q1" [u(2.25)] {B}="r5"
"q5":"r0"|(.5)*=<5pt>{}:@{<-}"r1"^-{}:@{<-}"r2"^-{}:"r3"^-{}:"r0"^-{} "r3":@{<-}"r4":"r5":@/^{1.5pc}/"q2" "r4":"q0" "r5":"q1" "p5":@/^{.5pc}/"r1"
"r4" [l(.75)u(.75)] {Y_2}="r6" [u(.5)r(1.25)] {Y_3}="r7" [r(1.5)] {Y_4}="r8"
"r7" (:"r6"(:"r2",:"r4"),:"r8":"r5")
"p5":@{}"r0"|(.4){\tn{pb}}:@{}"r2"|-{\tn{pb}}:@{}"r4"|-{\tn{pb}}:@{}"r8"|-{\tn{dpb}}
"r6":@{.>}"p0" "r7" (:@{.>}"p1",:@{.>}@/_{.75pc}/"p4") "r8" (:@{.>}@/_{.3pc}/"p3",:@{.>}@<-.75ex>@/^{1pc}/"p2"|-{\alpha},:@{.>}@<.75ex>@/^{1pc}/"p2"|-{\beta})} \]
in which $\alpha$ is the arrow induced in (\ref{diag:phil-proof}), $\beta$ is the arrow induced in (\ref{diag:monad-morphisms-coherences-explicit}), and so our task is to show that $\alpha = \beta$. By definition $\alpha$ is induced by the distributivity pullback with vertices $(Y'_5,G_2,G_1,Y'_4,Y'_6)$, and $\beta$ is induced by the pullback $(Y'_4,Y'_4,J,I)$. Using the above diagram, one checks $\alpha = \beta$ by showing that they are identified by composing with the pullback projections $I \leftarrow Y'_4 \to Y'_3$.
\end{proof}
\begin{proof}
(\emph{of Lemma \ref{lem:Fc-explicit}}).
By Lemma \ref{lem:explicit-colax-lax-monad-morphism-coherence} $F^c_X$ is the morphism $SX \to T\Sigma_fX$ over $J$ induced in
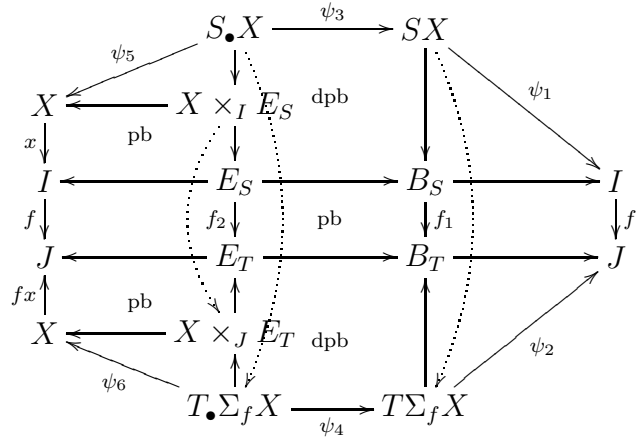
\[ \xygraph{!{0;(2,0):(0,.4)::} {I}="p0" [r] {E_S}="p1" [r] {B_S}="p2" [r] {I}="p3" [d] {J}="p4" [l] {B_T}="p5" [l] {E_T}="p6" [l] {J}="p7" "p0":@{<-}"p1"^-{}:"p2"^-{}:"p3"^-{}:"p4"^-{f}:@{<-}"p5"^-{}:@{<-}"p6"^-{}:"p7"^-{}:@{<-}"p0"^-{f} "p1":"p6"_-{f_2} "p2":"p5"^-{f_1} "p1":@{}"p5"|-{\tn{pb}}
"p0" [u] {X}="q0" [r] {X \times_I E_S}="q1" [u] {S_{\bullet}X}="q2" [r] {SX}="q3"
"q2" (:"q0"_-{\psi_5},:"q1"(:"q0":"p0"_-{x},:"p1"),:"q3"^-{\psi_3}(:"p2",:"p3"^-{\psi_1}))
"q0" [d(.4)r(.5)] {\scriptsize{\tn{pb}}} "q1" [u(.1)r(.5)] {\scriptsize{\tn{dpb}}}
"p7" [d] {X}="r0" [r] {X \times_J E_T}="r1" [d] {T_{\bullet}\Sigma_fX}="r2" [r] {T\Sigma_fX}="r3"
"r2" (:"r0"^-{\psi_6},:"r1"(:"r0":"p7"^-{fx},:"p6"),:"r3"_-{\psi_4}(:"p5",:"p4"_-{\psi_2}))
"r0" [u(.4)r(.5)] {\scriptsize{\tn{pb}}} "r1" [d(.1)r(.5)] {\scriptsize{\tn{dpb}}}
"q1":@{.>}@/_{1.5pc}/"r1" "q2":@{.>}@/^{1.5pc}/"r2" "q3":@{.>}@/^{1.5pc}/"r3"} \]
Let us define $\phi_1 : SX \to T\Sigma_fX$ as what we seek $F^c_X$ to be, namely as
\[ \begin{array}{lccr} {\phi_1(\alpha,(x_k)_k) = (F\alpha,(i_k,x_k)_k)} &&&
{\phi_1(\rho,(\gamma_k)_k) = (\rho,(\gamma_k)_k).} \end{array} \]
By construction and the explicit descriptions of $SX$ (and $TX$) established above, one may identify the objects of $S_{\bullet}X$ as comprising the data of an object of $SX$ together with a chosen input for the underlying operation of $S$. That is, one may an object of $S_{\bullet}X$ as $(\alpha,(x_k)_k,l)$ where $\alpha : (i_k)_{1{\leq}k{\leq}n} \to i$ is in $S$, $x_k \in X_{i_k}$ and $1 \leq l \leq n$. Similarly, a morphism $(\alpha,(x_k)_k,l) \to (\beta,(y_k)_k,m)$ of $S_{\bullet}X$ consists of a morphism $(\rho,(\gamma)_k) : (\alpha,(x_k)_k) \to (\beta,(y_k)_k)$ of $SX$ such that $\rho l = m$. Now define $\phi_2 : S_{\bullet}X \to T_{\bullet}\Sigma_fX$ to be
\[ \begin{array}{lccr} {\phi_2(\alpha,(x_k)_k,l) = (F\alpha,(i_k,x_k)_k,l)} &&&
{\phi_2(\rho,(\gamma_k)_k) = (\rho,(\gamma_k)_k).} \end{array} \]
By the universal property of the bottom distributivity pullback, and the bottom pullback in the above diagram, it suffices to verify that
\[ \begin{array}{lcccccr} {\psi_2\phi_1 = f\psi_1} && {\psi_3\phi_1 = \psi_4\phi_2} && {fx\psi_6\phi_2 = x\psi_5} && {\psi_8\phi_2 = f_2\psi_7} \end{array} \]
where $\psi_7$ and $\psi_8$ are the composites $S_{\bullet}X \to X \times_I E_S \to E_S$ and $T_{\bullet}\Sigma_fX \to X \times_J E_T \to E_T$ respectively, which is entirely straight forward since all the functors in our present situation now have such explicit descriptions.
\end{proof}
\begin{proof}
(\emph{of Lemma \ref{lem:Fl-explicit}}).
By Lemma \ref{lem:explicit-colax-lax-monad-morphism-coherence} $F^l_Y$ is the morphism $S\Delta_fY \to \Delta_fTY$ over $I$ induced in
\[ \xygraph{!{0;(1.5,0):(0,.5)::} {I}="p0" [r] {E_T}="p1" [r] {B_T}="p2" [r] {I}="p3" [d] {J}="p4" [l] {B_S}="p5" [l] {E_S}="p6" [l] {J}="p7" "p0":@{<-}"p1"^-{}:"p2"^-{}:"p3"^-{}:"p4"_-{f}:@{<-}"p5"^-{}:@{<-}"p6"^-{}:"p7"^-{}:@{<-}"p0"_-{f} "p1":"p6"_-{f_2} "p2":"p5"^-{f_1} "p1":@{}"p5"|-{\tn{pb}}
"p0" [l] {Y \times_J I}="q1" "p7" [l] {Y}="q0" "p0" [u] {Y \times_J E}="q2" "p1" [u(2)] {S_{\bullet}\Delta_fY}="q3" "p2" [u(2)] {S\Delta_fY}="q4" 
"q1" (:"q0"_-{\psi_8}:"p7",:"p0") "q2"(:"q1"_-{\psi_7},:"p1"^-{\psi_9}) "q3"(:"q2"^-{\psi_6},:"q4"(:"p2"^-{\psi_4},:"p3"^-{\psi_1}))
"q1":@{}"p7"|-{\tn{pb}} "q2":@{}"p0"|-{\tn{pb}} "p1" [r(.25)u] {\scriptsize{\tn{dpb}}}
"p7" [d] {Y \times_J E_T}="r2" "p6" [d(2)] {T_{\bullet}Y}="r3" "p5" [d(2)] {TY}="r4" "p4" [r] {\Delta_fTY}="r5"
"r2"(:"q0"^-{\psi_{11}},:"p6"_-{\psi_{12}}) "r3"(:"r2"^-{\psi_{10}},:"r4"(:"p5"_-{\psi_5},:"p4")) "r5"(:"r4"^-{\psi_3},:"p3"_-{\psi_2})
"r2":@{}"p7"|-{\tn{pb}} "p6" [r(.25)d] {\scriptsize{\tn{dpb}}} "p4" [r(.5)] {\scriptsize{\tn{pb}}}
"q2":@/_{1.25pc}/@{.>}"r2" "q3":@/_{1.25pc}/@{.>}"r3" "q4":@/_{1.25pc}/@{.>}"r4" "q4":@/^{1.25pc}/@{.>}"r5"} \]
Let us define $\phi_1 : S\Delta_fY \to \Delta_fTY$ as what we seek $F^l_Y$ to be, namely as
\[ \begin{array}{lccr} {\phi_1(\alpha,(y_k)_k) = (F\alpha,(y_k)_k)} &&& {\phi_1(\rho,(\delta_k)_k) = (\rho,(\delta_k)_k).} \end{array} \]
We define $\phi_2 : S\Delta_fY \to TY$ by the same equations
\[ \begin{array}{lccr} {\phi_2(\alpha,(y_k)_k) = (F\alpha,(y_k)_k)} &&& {\phi_2(\rho,(\delta_k)_k) = (\rho,(\delta_k)_k)} \end{array} \]
the difference with $\phi_1$ being that $\phi_2$ now lives over $J$. As with the proof of Lemma \ref{lem:Fc-explicit}, we define $\phi_3 : S_{\bullet}\Delta_fY \to T_{\bullet}Y$ by
\[ \begin{array}{lccr} {\phi_3(\alpha,(y_k)_k,l) = (F\alpha,(y_k)_k,l)} &&& {\phi_3(\rho,(\delta_k)_k) = (\rho,(\delta_k)_k)} \end{array} \]
in view of the explicit descriptions of $S_{\bullet}\Delta_fY$ and $T_{\bullet}Y$. In order to identify $\phi_1$ with the above induced map, it suffices by the universal properties of the bottom distributivity pullback and those of the pullbacks defining $Y \times_J E_T$ and $\Delta_fTY$, to verify that
\[ \begin{array}{lllll} {\phi_1\psi_2 = \psi_1} && {\psi_3\phi_1 = \phi_2} && {\psi_5\phi_2 = f_1\psi_4} \\ {\psi_{10}\psi_{12}\phi_3 = f_2\psi_9\psi_6} && {\psi_{11}\psi_{10}\phi_3 = \psi_8\psi_7\psi_6.} && \end{array} \]
As in the proof of Lemma \ref{lem:Fc-explicit}, this is straight forward because the functors participating in these equations now have completely explicit descriptions.
\end{proof}

\section{Internal algebra classifiers and codescent objects}
\label{sec:IntAlgClass}

Internal algebra classifiers are defined formally in Section \ref{ssec:int-alg-classifiers} and their flexibility is established. We review codescent objects and cateads in Section \ref{ssec:codescent-review}, and present the example of a simplicial object of which we wish to take codescent, but which is not a catead. Then in Section \ref{ssec:IntAlgClass-as-codescent} we describe internal algebra classifiers as codescent objects. In Section \ref{ssec:IntAlgClass-first-examples} we explain why in the examples of interest for us, the simplicial objects of which we take codescent are category objects.

\subsection{Internal algebra classifiers.}
\label{ssec:int-alg-classifiers}
With the notion of $S$-algebra internal to a $T$-algebra provided by Definition \ref{defn:internal-algebra}, it is natural to wonder about the ``free $T$-algebra generated by an internal $S$-algebra''. In various of our examples there is a well-established understanding of what this object should be. For instance as explained in MacLane \cite{MacLane-CWM}, the algebraists' simplicial category $\Delta_+$ is the free strict monoidal category containing a monoid. Similarly as is well known, the category $\mathbb{S}$ of finite ordinals and all functions between them is the free symmetric strict monoidal category containing a commutative monoid{\footnotemark{\footnotetext{We give a proof of this result in Theorem \ref{thm:free-smc-containing-a-commutative-monoid} below as an illustration of the general methods of this article.}}}.

In this section we give a general definition of such universal objects, called \emph{internal algebra classifiers} at the generality of Definition \ref{defn:internal-algebra}. Under mild conditions internal algebra classifiers also enjoy a universal property of a bicategorical nature with respect to pseudo algebras. Applied to $\Delta_+$ (resp. $\mathbb{S}$) this bicategorical universal property gives an equivalence of categories between the category of monoids (resp. commutative monoids) in a monoidal (resp. symmetric monoidal) category $\ca V$, and the category of strong monoidal functors $\Delta_+ \to \ca V$ (resp. symmetric strong monoidal functors $\mathbb{S} \to \ca V$) and monoidal natural transformations between them.
\begin{rem}\label{rem:J-F}
Given an adjunction $F : (\ca L,S) \to (\ca K,T)$ of 2-monads, one has commutative squares
\[ \xygraph{!{0;(2,0):(0,.5)::} {\ca K}="p1" [r] {\Algs{T}}="p2" [r] {\Algl{T}}="p3" [d] {\Algl{S}}="p4" [l] {\Algs{S}}="p5" [l] {\ca L}="p6"
"p1":@{<-}"p2"^-{U^T}:"p3"^-{J_T}:"p4"^-{\overline{F}}:@{<-}"p5"^-{J_S}:"p6"^-{U^S}:@{<-}"p1"^-{F^*} "p2":"p5"^-{\overline{F}}
"p3" [r] {\PsAlgl{T}}="p7" [d] {\PsAlgl{S}}="p8" "p3":"p7":"p8"^{\overline{F}}:@{<-}"p4"} \]
in which the unlabelled arrows are the inclusions (see Remark \ref{rem:adjmnd-data-allowing-defn-of-internal-algebra}). We denote by
\[ J_F : \Algs T \longrightarrow \Algl S \]
the diagonal of the middle square in the above display, and by $(-)^{\dagger}_F$ the left adjoint to $J_F$, when it exists.
\end{rem}
\begin{defn}\label{def:int-alg-classifier}
Let $F : (\ca L,S) \to (\ca K,T)$ be an adjunction of 2-monads, suppose that $\ca L$ has a terminal object $1$, and let $A$ be a strict $T$-algebra. An internal $S$-algebra $s : 1 \to \overline{F}A$ exhibits $A$ as the \emph{internal $S$-algebra classifier} (relative to $F$) when for all strict $T$-algebras $B$, the functor
\[ \Algs{T}(A,B) \longrightarrow \Algl{S}(1,\overline{F}B) \]
given by precomposition with $s$ is an isomorphism of categories.
\end{defn}
\begin{rem}\label{rem:T-to-the-S}
The defining universal property of a internal $S$-algebra classifier $s : 1 \to \overline{F}A$ is exactly that of the component of the unit of the adjunction $(-)^{\dagger}_F \ladj J_F$ at $1$. When a choice of left adjoint $(-)^{\dagger}_F$ to $J_F$ is given, we denote by
\[ g_T^S : 1 \longrightarrow \overline{F}T^S \]
the unit component at $1$, which exhibits $T^S$ as an internal $S$-algebra classifier.
\end{rem}
The universal property that $T^S$ enjoys with respect to pseudo algebras is due in part to the flexibility of $T^S$ as a $T$-algebra, and also to general coherence theorems that are available in the situations of interest. We now recall the required 2-dimensional monad theory.

A functor or 2-functor is said to \emph{have rank} when it preserves $\lambda$-filtered colimits for some regular cardinal $\lambda$, and \emph{finitary} when $\lambda$ can be taken to be the cardinality of $\N$. The study of coherence issues for 2-monads, and of internal algebras, begins with the fact that for a 2-monad $T$ with rank on a complete and cocomplete 2-category $\ca K$, the inclusions
\[ \begin{array}{lccr} {J:{\Algs{T}} \longrightarrow {\Alg{T}}} &&& {J_{\tn{p}}:{\Algs{T}} \longrightarrow {\PsAlg{T}}} \end{array} \]
have left adjoints, which we shall denote as $Q$ and $(-)'$ respectively. The importance of such adjunctions was first recognised in \cite{BWellKellyPower-2DMndThy}. In \cite{Lack-Codescent} the left adjoints were constructed explicitly using certain 2-categorical colimits called codescent objects leading to even more general conditions on their existence, namely just that these colimits exist in $\Algs{T}$. We shall discuss codescent objects further in Section \ref{ssec:codescent-review}.

The adjunction $Q \ladj J$ allows us to understand coherence for pseudo morphisms at this general level. It is not true in general that the inclusion
\[ J_{A,B}:\Algs{T}(A,B) \longrightarrow \Alg{T}(A,B) \]
of the strict $T$-algebra morphisms amongst the pseudo morphisms is an equivalence for all strict algebras $A$ and $B$, but this is true when $A$ is flexible as we now recall. 

Following \cite{BWellKellyPower-2DMndThy} section 4, the components of the unit and counit of $Q \ladj J$ are denoted as $p_A:A \to QA$ and $q_A:QA \to A$ respectively. In general one has, by Theorem 4.2 of \cite{BWellKellyPower-2DMndThy}, that $p_A \ladj q_A$ in $\Alg{T}$ with unit an identity and counit invertible. Thus $p_A$ is a section and a pseudo inverse to $q_A$, but note that $p_A$ is a pseudo morphism. The strict $T$-algebra $A$ is said to be \emph{flexible} when it has a section in $\Algs{T}$. By \cite{BWellKellyPower-2DMndThy} Theorem 4.2 any such section is also a pseudo inverse in $\Algs{T}$, and moreover this condition is also equivalent to $A$ being a retract of $QC$ for some $C$. The result that for flexible $A$, $J_{A,B}$ is an equivalence is contained in Theorem 4.7 of \cite{BWellKellyPower-2DMndThy}. In other words for flexible $A$, every pseudo morphism $A \to B$ is isomorphic in $\PsAlg{T}$ to a strict one.

To understand when one has a coherence Theorem for pseudo algebras one considers the adjunction $(-)' \ladj J_{\tn{p}}$. For a given pseudo $T$-algebra $A$, the unit of this adjunction is pseudo morphism $s_A:A \to A'$, where $A'$ is by definition a strict $T$-algebra, and one may then ask when $s_A$ is an equivalence. The 2-monads of interest in this article are all 2-monads $T$ with rank on 2-categories $\tn{Cat}(\ca E)$ of categories internal to a locally presentable category $\ca E$. For such $\ca E$, $\tn{Cat}(\ca E)$ is complete and cocomplete, and if moreover $T$ preserves internal functors whose object maps are invertible, then the general coherence result of Power \cite{Power-GeneralCoherenceResult}, as formulated still more generally by Lack in \cite{Lack-Codescent} Theorem 4.10, gives a natural situation in which all the components $s_A$ are equivalences. Thus for such 2-monads every pseudo algebra is equivalent to a strict one in a canonical way. This applies in particular to the case where $\ca E$ is $\Set$, and $T$ is $\tnb{M}$, $\tnb{S}$ or $\tnb{B}$, to give the coherence Theorems for monoidal, symmetric monoidal and braided monoidal categories respectively.
\begin{prop}\label{prop:uprop-intalg-pseudo}
Let $F:(\ca L,S) \to (\ca K, T)$ be an adjunction of 2-monads, $S$ and $T$ have rank, and $\ca K$ and $\ca L$ have all limits and colimits. Then $T^S$ exists and is flexible. If moreover $\ca K$ is of the form $\Cat(\ca E)$ for some category $\ca E$ with pullbacks, and $T$ preserves internal functors whose object maps are invertible, then one has equivalences
\[ \PsAlg{T}(T^S,A) \catequiv \PsAlgl{S}(1,\overline{F}A) \]
pseudo naturally in $A \in \PsAlg{T}$.
\end{prop}
\begin{proof}
Under the given hypotheses $J_S$ has a left adjoint \cite{BWellKellyPower-2DMndThy}, and by a standard application of the Dubuc adjoint triangle Theorem \cite{Dubuc-KanExtensions}, so does $\overline{F} : \Algs T \to \Algs S$. Thus $(-)^{\dagger}_F$ and hence $T^S$ exist.

To establish the flexibility of $T^S$ we apply Theorem 5.1 of \cite{BWellKellyPower-2DMndThy} which says the following. Given a 2-monad $T$ with rank on a complete and cocomplete 2-category $\ca K$, a 2-functor $G:\Alg{T} \to \ca M$, and a left adjoint $H$ to $GJ$ with unit $s:1 \to GJH$; then
\begin{enumerate}
\item For all $M \in \ca M$, $HM$ is flexible.\label{BKP-flex}
\item For all $M \in \ca M$ and $A \in \Alg{T}$,
\[ G(-) \comp s_M : \Alg{T}(JHM,A) \longrightarrow \ca M(M,GA) \]
is a surjective equivalence.\label{BKP-bicat-uprop}
\end{enumerate}
For our application of this result we take $\ca M = \Algl S$ and note that $J_F$ may be factored as
\[ \xygraph{!{0;(1.5,0):} {\Algs T}="p0" [r] {\Alg T}="p1" [r] {\Algl T}="p2" [r] {\Algl S}="p3" "p0":"p1"^-{J}:"p2"^-{J'}:"p3"^-{\overline{F}}} \]
in which $J$ and $J'$ are the inclusions, we take $G = \overline{F}J'$ and so $H$ is the left adjoint $(-)^{\dagger}_F$ whose existence we established above. Thus by (\ref{BKP-flex}) with $M = 1$, $T^S$ is indeed flexible.

We now assume $\ca K$ is of the form $\tn{Cat}(\ca E)$ for some category $\ca E$ with pullbacks and that $T$ preserves internal functors whose object maps are invertible. As recalled above for any pseudo $T$-algebra $A$ one has a strict $T$-algebra $A'$ and pseudomorphism $s_A:A \to A'$ which is an equivalence in $\PsAlg{T}$. Then for any strict $S$-algebra $B$ we have
\[ \begin{array}{lll} {\PsAlg{T}(B^{\dagger}_F,A)} & {\catequiv} & {\Alg{T}(B^{\dagger}_F,A')} \\ {} & {\iso} & {\Algl{S}(B,\overline{F}A')} \\ {} & {\catequiv} & {\PsAlgl{S}(B,\overline{F}A)} \end{array} \]
in which the isomorphism is from the definition of $B^{\dagger}_F$, the first equivalence is given by composing with the equivalence $s_A$, and the other equivalence is given by composing with the equivalence $\overline{F}(s_A)$. All this is clearly pseudo natural in $A$, and putting $B = 1$ gives the result.
\end{proof}

\subsection{Review of codescent objects.}
\label{ssec:codescent-review}
There are various 2-categorical colimits that are called codescent objects in the literature \cite{Street-CorrFibrationInBicats, Lack-Codescent, Bourke-Thesis}. In this article we shall just consider the variant which arises by considering how to compute the left adjoint $(-)^{\dagger}_F$ to $J_F$ in the context of Remark \ref{rem:J-F}. Recall the standard cosimplicial object $\delta : \Delta \to \Cat$ recalled in the introduction, and the notions of weighted limit and colimit from enriched category theory \cite{Kelly-EnrichedCatsBook}.
\begin{defn}\label{def:codescent-object}
\cite{Street-CorrFibrationInBicats}
Let $\ca K$ be a 2-category, $X \in [\Delta^{\op},\ca K]$ and $Y \in [\Delta,\ca K]$.
\begin{enumerate}
\item A \emph{descent object} of $Y$ is a limit of $Y$ weighted by $\delta$.
\item A \emph{codescent object} of $X$ is a colimit of $X$ weighted by $\delta$.
\end{enumerate}
\end{defn}
As such the descent object of $Y$ consists of an object $\tn{Desc}(Y)$ of $\ca K$ together with isomorphisms as on the left
\[ \begin{array}{lccr} {\ca K(Z,\tn{Desc}(Y)) \iso [\Delta,\Cat](\delta,\ca K(Z,Y-))} &&& {\tn{Desc}(Y) \iso [\Delta,\Cat](\delta,Y)} \end{array} \]
\[  \]
2-natural in $Z$. When $\ca K = \Cat$ the descent object of $Y$ exists in general and is computed as on the right in the previous display. In this case one has the following simple description of the category $\tn{Desc}(Y)$ since the data of a natural transformation $f:\delta \to Y$ is determined by its components $f_{[0]}$ and $f_{[1]}$:
\begin{itemize}
\item An object consists of an object $y_0 \in Y_0$, together with a morphism $y_1:\delta_1y_0 \to \delta_0y_0$ in $Y_1$, such that $\sigma_0y_1 = 1_{y_0}$ and $\delta_1y_1 = (\delta_0y_1)(\delta_2y_1)$.
\item A morphism $(y_0,y_1) \to (z_0,z_1)$ consists of $f:y_0 \to z_0$ in $Y_0$ such that $(\delta_0f)y_1 = z_1(\delta_1f)$.
\end{itemize}

The codescent object of $X$ consists of an object $\tn{CoDesc}(X)$ of $\ca K$ together with isomorphisms
\begin{equation}\label{eq:codesc-defining-isos}
\ca K(\tn{CoDesc}(X),Z) \iso \tn{Desc}(\ca K(X-,Z))
\end{equation}
of categories 2-natural in $Z$. An object of the category on the right is called a \emph{cocone} for $X$ with vertex $Z \in \ca K$, which in explicit terms is a pair $(f_0,f_1)$, where $f_0:X_0 \to Z$ and $f_1:f_0d_1 \to f_0d_0$ are in $\ca K$, and satisfy $f_1s_0=1_{f_0}$ and $(f_1d_0)(f_1d_2)=f_1d_1$. A morphism of cocones $\phi : (f_0,f_1) \to (g_0,g_1)$ is a 2-cell $\phi : f_0 \to g_0$ such that $(\phi d_0)f_1 = g_1(\phi d_1)$. We denote by $(q_0,q_1)$ the cocone corresponding under (\ref{eq:codesc-defining-isos}) to $1_{\tn{Codesc}(X)}$, whose universal property can be used to give a more hands-on definition of the notion of codescent object, as we now recall.

To give a codescent object for $X \in [\Delta^{\op},\ca K]$ is to give an object $\tn{CoDesc}(X)$ of $\ca K$ together with a cocone $(q_0,q_1)$ with for $X$ vertex $\tn{CoDesc}(X)$ satisfying the following universal properties.
\begin{enumerate}
\item (1-dimensional universal property): For any cocone $(f_0,f_1)$ with vertex $Z$, there is a unique $f':\tn{CoDesc}(X) \to Y$ such that $f'q_0 = f_0$ and $f'q_1 = f_1$.\label{1-dim-uprop-codesc}
\item (2-dimensional universal property): For any morphism of cocones $\phi:(f_0,f_1) \to (g_0,g_1)$ for $X$ with vertex $Z$, there is a unique $\phi':f' \to g'$ such that $\phi'q_0 = \phi$.\label{2-dim-uprop-codesc}
\end{enumerate}
Recall that when $\ca K$ admits cotensors with $[1]$, the 1-dimensional universal property implies the 2-dimensional one.
\begin{exam}\label{ex:cat-as-codesc-of-nerve}
Let $X$ be a category and regard it as a simplicial set by taking its nerve
\[ \xygraph{!{0;(2,0):(0,1)::} {X_2}="p0" [r] {X_1}="p1" [r] {X_0}="p2"
"p2":"p1"|-{s_0} "p1":@<1.5ex>"p2"^-{d_1} "p1":@<-1.5ex>"p2"_-{d_0} "p0":@<1.5ex>"p1"^-{d_2} "p0":"p1"|-{d_1} "p0":@<-1.5ex>"p1"_-{d_0}} \]
and then as a simplicial object in $\Cat$ be regarding the $X_n$ as discrete categories. Taking $q_0 : X_0 \to X$ to be the inclusion of objects, and $q_1 : q_0d_1 \to q_0d_0$ to be the natural transformation with components given by $(q_1)_f = f$, exhibits $X$ as the codescent object of its nerve.
\end{exam}
\begin{exam}\label{ex:internal-nerve}
Recall that given a category $\ca E$ with pullbacks, a \emph{category object} in $\ca E$ is a simplicial object $X \in [\Delta^{\op},\ca E]$ such that for all $n \in \N$ the square
\begin{equation}\label{eq:pb-square-for-cat-object}
\xygraph{!{0;(1.5,0):(0,.6667)::} {X_{n+2}}="tl" [r] {X_{n+1}}="tr" [d] {X_n}="br" [l] {X_{n+1}}="bl" "tl":"tr"^-{d_{n+2}}:"br"^-{d_0}:@{<-}"bl"^-{d_{n+1}}:@{<-}"tl"^-{d_0}}
\end{equation}
is a pullback. One has a 2-category $\tn{Cat}(\ca E)$ in which the objects are internal categories and whose arrows $f:X \to Y$ are just morphisms in $[\Delta^{\op},\ca E]$ and are called \emph{internal functors}. A 2-cell $\phi:f \to g$ has underlying data an arrow $\phi_0:X_0 \to Y_1$ in $\ca E$ such that $d_1\phi_0 = f_0$ and $d_0\phi_0 = g_0$, given such $\phi_0$ one induces $\phi_1$, $\phi_2:X_1 \to Y_2$ in $\ca E$ unique such that
\[ \begin{array}{lcccccr} {d_2\phi_1 = \phi_0d_1} && {d_0\phi_1 = g_1} && {d_2\phi_2 = f_1} && {d_0\phi_2 = \phi_0d_0} \end{array} \]
and then one demands that $d_1\phi_1 = d_1\phi_2$. As explained in \cite{Bourke-Thesis}, Example \ref{ex:cat-as-codesc-of-nerve} generalises to internal categories. Thus given a simplicial object $X$ in $\tn{Cat}(\ca E)$, if it is componentwise discrete and a category object, then to ``calculate'' its codescent object is just to interpret the simplicial diagram $X$ as an internal category, that is to say, as an object of $\tn{Cat}(\ca E)$.
\end{exam}
In \cite{Bourke-Thesis} 2-categories of the form $\tn{Cat}(\ca E)$, for $\ca E$ a category with pullbacks, were characterised in terms of the existence and well-behavedness of codescent objects of cateads. A \emph{catead} in a 2-category $\ca K$ is a category object $X : \Delta^{\op} \to \ca K$ such that the span
\[ \xygraph{{X_0}="p0" [r] {X_1}="p1" [r] {X_0}="p2" "p0":@{<-}"p1"^-{d_1}:"p2"^-{d_0}} \]
is a 2-sided discrete fibration in the sense of \cite{Street-FibrationIn2cats}. In particular for any category $\ca E$ with pullbacks, the 2-category $\tn{Cat}(\ca E)$ of categories internal to $\ca E$ admits codescent objects of all cateads, and these are preserved by 2-functors of the form $\tn{Cat}(F) : \tn{Cat}(\ca E) \to \tn{Cat}(\ca F)$, where $F : \ca E \to \ca F$ is a pullback preserving functor.

For $\ca E$ a category with pullbacks, (the underlying category of) $\tn{Cat}(\ca E)$ is a category with pullbacks, and the functor $\tn{ob}_{\ca E} : \tn{Cat}(\ca E) \to \ca E$ which sends an internal category to its object of objects is pullback preserving, so that one has a 2-functor
\[ \tn{Cat}(\tn{ob}_{\ca E}) : \tn{Cat}(\tn{Cat}(\ca E)) \longrightarrow \tn{Cat}(\ca E). \]
The straight forward computation of codescent objects of cateads in $\tn{Cat}(\ca E)$ is described by
%
%
\begin{prop}\label{prop:Bourke-codesc-catead}
\cite{Bourke-Thesis}
Let $\ca E$ be a category with pullbacks and $X$ be a catead in $\tn{Cat}(\ca E)$. Then
\[ \CoDesc(X) =  \tn{Cat}(\tn{ob}_{\ca E})(X). \]
\end{prop}
In particular when $X$ is componentwise discrete it is a catead, and so Proposition \ref{prop:Bourke-codesc-catead} generalises Example \ref{ex:internal-nerve}. However the most interesting examples for us are not cateads, such as
\begin{exam}\label{exam:bar-of-Sm-not-a-catead}
An instance of the bar construction (see Section \ref{ssec:IntAlgClass-as-codescent} below) in the case of the 2-monad $\tnb{S}$ gives rise to the simplicial object
\[ \xygraph{!{0;(3,0):(0,1)::}
{\tnb{S}^3(1)}="p0" [r] {\tnb{S}^2(1)}="p1" [r] {\tnb{S}(1)}="p2"
"p2":"p1"|-{\eta_1} "p1":@<1.5ex>"p2"^-{\mu_1} "p1":@<-1.5ex>"p2"_-{\tnb{S}(t_{\tnb{S}(1)})} "p0":@<1.5ex>"p1"^-{\mu_{\tnb{S}(1)}} "p0":"p1"|-{\tnb{S}(\mu_1)} "p0":@<-1.5ex>"p1"_-{\tnb{S}(t_{\tnb{S}(1)})}} \]
which, as we shall see now, is not a catead. If it were a catead then $d_1 = \mu_1$, as the left leg of a 2-sided discrete fibration, would be a fibration by \cite{Weber-2Toposes} Theorem 2.11(3). In explicit terms $\tnb{S}(1)$ is the permutation category $\P$, the effect of $\mu_1$ on objects is $(m_k)_{1{\leq}k{\leq}n} \mapsto m_1 + ... + m_k$, and on arrows sends the morphism $(\rho,(\rho_k)_k) : (m_k)_k \to (m'_k)_k$ of $\tnb{S}(\P)$ to the result
\[ \rho(\rho_k)_k : \sum_{k=1}^n m_k \longrightarrow \sum_{k=1}^n m'_k \]
of substituting the permutations $\rho_k$ into the permutation $\rho$. If $\mu_1$ is a fibration, then for any $(m'_k)_{1{\leq}k{\leq}n}$ from $\tnb{S}(\P)$, and any permutation $\psi : m \to \sum_k m'_k$, one can find $\rho$ and $(\rho_k)_k$ such that $\psi = \rho(\rho_k)_k$. Let us call any subset $s \subseteq \underline{m}$ an \emph{interval} if given $i_1$ and $i_2 \in s$, and $i_3 \in m$ such that $i_1 < i_3 < i_2$, then $i_3 \in s$. The condition that $\psi$ is of the form $\rho(\rho_k)_k$ implies that for any $k$, $\psi^{-1}m'_k \subseteq m$ is an interval, where $m'_k$ is regarded in the evident way as a subinterval of $\underline{m}$. For a counter example take $n=2$, $m'_1=m'_2=2$ and $\psi = (124)$. Thus as a subinterval of $\underline{4}$, $m'_1$ is $\{1,2\}$, and so $\psi^{-1}m'_1 = \{1,4\}$ which is not an interval. 
\end{exam}

\subsection{Internal algebra classifiers as codescent objects.}
\label{ssec:IntAlgClass-as-codescent}
The construction of the left adjoint $(-)^{\dagger}_T$ to $J_T : \Algs T \to \Algl T$ in terms of codescent objects was discussed in \cite{Bourke-Thesis, Lack-Codescent}. In this section we generalise this from the setting of a single 2-monad $T$ to that of an adjunction $F:(\ca L,S) \to (\ca K,T)$ of 2-monads. Instances of this general construction play a central role in \cite{Batanin-EckmannHilton, BataninBerger-HtyThyOfAlgOfPolyMnd}.

Let $(\ca K,T)$ be a 2-monad and denote by $C_T : \Delta_+^{\op} \to [\Algs T,\Algs T]$ the strict monoidal functor corresponding to the comonad generated by the Eilenberg-Moore adjunction $F^T \ladj U^T$. We define the 2-functor $\ca R_T$ as on the left in
\[ \begin{array}{lccr} {\ca R_T : \Algs T \to [\Delta^{\op},\Algs T]}
&&& {\Delta^{\op} \hookrightarrow \Delta_+^{\op} \xrightarrow{C_T} [\Algs T,\Algs T]} \end{array} \]
as the adjoint transpose of the composite given on the right in the previous display. In more explicit terms $\ca R_TX$ is given on objects by
\[ (\ca R_TX)_n = T^{n+1}X \]
and has face and degeneracy maps given by the formulae
\[ \begin{array}{lccr}
{d_i^{n+1} = \left\{
\begin{array}{lll} {T^{n+1}x} && {i = 0} \\
{T^{n+1-i}\mu^T_{T^{i-1}X}} && {1 \leq i \leq n+1} \end{array}\right.}
&&& {s_i^{n+1} = T^{n+1-i}\eta^T_{T^iX}.}
\end{array} \]
where $x:TX \to X$ denotes the $T$-algebra structure. Thus the part of $\ca R_TX$ which affects the value of its codescent object is
\[ \xygraph{!{0;(2,0):(0,1)::} {T^3X}="p0" [r] {T^2X}="p1" [r] {TX.}="p2"
"p2":"p1"|-{T\eta^T_X} "p1":@<1.5ex>"p2"^-{\mu^T_X} "p1":@<-1.5ex>"p2"_-{Tx} "p0":@<1.5ex>"p1"^-{\mu^T_{TX}} "p0":"p1"|-{T\mu^T_X} "p0":@<-1.5ex>"p1"_-{T^2x}} \]
For instance $\ca R_{\tnb{S}}1$ is the simplicial object considered in Example \ref{exam:bar-of-Sm-not-a-catead}. The simplicial object $\ca R_TX$ is well known and is usually called the \emph{bar construction}. We now give a generalisation of this construction for adjunctions of 2-monads.
\begin{const}\label{const:simplicial-objects-from-monad-adjunction}
Given an adjunction of 2-monads $F : (\ca L,S) \to (\ca K,T)$ we now construct a 2-functor
\[ \ca R_F : \Algs S \longrightarrow [\Delta^{\op},\Algs T] \]
which in the case $F = 1_{(\ca K,T)}$ is $\ca R_T$. For $X \in \Algs S$, $Y \in \Algs T$ and $n \in \N$, one has isomorphisms
\[ \Algs S(S^{n+1}X,\overline{F}Y) \iso \ca L(S^nX,F^*Y) \iso \ca K(F_!S^nX,Y) \iso \Algs T(TF_!S^nX,Y) \]
2-natural in $X$ and $Y$. By the Yoneda lemma, we define $\ca R_FX \in [\Delta^{\op},\Algs T]$ as unique such that $(\ca R_FX)_n = TF_!S^nX$, and
\begin{equation}\label{eq:RF-defining-nat-iso}
\Algs S(\ca R_SX,\overline{F}Y) \iso \Algs T(\ca R_FX,Y)
\end{equation}
2-naturally in $Y$. The effect of $\ca R_F$ on arrows and 2-cells is determined uniquely when we ask that the isomorphisms (\ref{eq:RF-defining-nat-iso}) be 2-natural in $X$.
\end{const}
Below we shall be computing codescent objects of simplicial objects of the form $\ca R_FX$. Thus it will be helpful to have a more explicit handle on their face and degeneracy maps, and this is provided by
\begin{lem}\label{lem:face-degen-R-F-X}
Given an adjunction $F:(\ca L,S) \to (\ca K,T)$ of 2-monads and a strict $S$-algebra $(X,x)$, the face and degeneracy maps of $\ca R_FX$ are given by the formulae
\[ \begin{array}{lccr}
{d_i^{n+1} = \left\{\begin{array}{lll}
{TF_!S^nx} && {i = 0} \\
{TF_!S^{n-i}\mu^S_{S^{i-1}X}} && {1 \leq i \leq n} \\
{\mu^T_{F_!S^nX}T(F^c_{S^nX})} && {i = n+1}
\end{array}\right.}
&&& {s_i^{n+1} = TF_!S^{n-i}\eta^S_{S^iX}.}
\end{array} \]
Thus the part of $\ca R_FX$ which affects the value of its codescent object is
\[ \xygraph{!{0;(3,0):(0,1)::} {TF_!S^2X}="p0" [r] {TF_!SX}="p1" [r] {TF_!X.}="p2"
"p2":"p1"|-{TF_!\eta^S_X} "p1":@<1.5ex>"p2"^-{\mu^T_{F_!X}T(F^c_X)} "p1":@<-1.5ex>"p2"_-{TF_!x} "p0":@<1.5ex>"p1"^-{\mu^T_{F_!SX}T(F^c_{SX})} "p0":"p1"|-{TF_!\mu^S_X} "p0":@<-1.5ex>"p1"_-{TF_!Sx}} \]
\end{lem}
\begin{proof}
We use the naturality of (\ref{eq:RF-defining-nat-iso}) and the explicit formulae for the face and degeneracy maps of $\ca R_SX$ to read off the formulae for the face and degeneracy maps of $\ca R_FX$. Let $(Y,y)$ be a strict $T$-algebra. The underlying object of $\overline{F}Y$ is by definition $F^*Y$ and its $S$-algebra structure is the composite $F^*(y)F^l_Y$. We denote by
\[ \varphi_{X,Y,n} : \Algs S(S^{n+1}X,\overline{F}Y) \longrightarrow \Algs T(TF_!S^nX,Y) \]
the components of (\ref{eq:RF-defining-nat-iso}). Recall that this isomorphism was a composite of various adjunction isomorphisms, and expressing these in terms of the units and counits of the participating adjunctions, one obtains for $f:S^{n+1}X \to \overline{F}Y$ the explicit formula
\[ \varphi_{X,Y,n}(f) = yT(\varepsilon^F_Y)TF_!(f)TF_!(\eta^S_{S^nX}). \]
Doing the same for $\varphi_{X,Y,n}^{-1}$ gives for $g:TF_!S^nX \to Y$, the explicit formula
\[ \varphi_{X,Y,n}^{-1}(g) = F^*(y)F^l_YSF^*(g)SF^*(\eta^T_{F_!S^nX})S(\eta^F_{S^nX}).  \]
The naturality of $\varphi_{X,TF_!S^nX}$ ensures that 
\[ \xygraph{!{0;(5,0):(0,.2)::} {\Algs T(TF_!S^nX,TF_!S^nX)}="p0" [r] {\Algs T(TF_!S^{n+1}X,TF_!S^nX)}="p1" [d] {\Algs S(S^{n+2}X,\overline{F}TF_!S^nX)}="p2" [l] {\Algs S(S^{n+1}X,\overline{F}TF_!S^nX)}="p3" "p0":"p1"^-{(-) \comp d_i^{n+1}}:@{<-}"p2"^-{\varphi}:@{<-}"p3"^-{(-) \comp d_i^{n+1}}:@{<-}"p0"^-{\varphi^{-1}}} \]
commutes, and the face map $d_i^{n+1} \in \Algs T(TF_!S^{n+1}X,TF_!S^nX)$ is the effect of the common composite on the identity $1_{TF_!S^nX}$. Thus applying the composite around the bottom to $1_{TF_!S^nX}$ gives a description of the face maps of $\ca R_FX$ in terms of those of $\ca R_SX$. Similarly, one expresses the degeneracy maps of $\ca R_FX$ in terms of those of $\ca R_SX$, and then the explicit formulae follow from straight forward calculations which are left to the reader.
\end{proof}
Having defined $\ca R_FX$ and given an explicit description of it, we now establish that its codescent object is the value on objects of the left adjoint $(-)^{\dagger}_F$ to $J_F$.
\begin{prop}\label{prop:dagger-F}
Let $F:(\ca L,S) \to (\ca K,T)$ be an adjunction of 2-monads and suppose that $\Algs T$ has all codescent objects. Then the left adjoint $(-)^{\dagger}_F$ to $J_F$ exists and is given on objects by the formula
\begin{equation}\label{eq:dagger-F-codesc-formula}
X^{\dagger}_F = \tn{CoDesc}(\ca R_FX).
\end{equation}
\end{prop}
\begin{proof}
Given any 2-monad $(\ca K,T)$, strict $T$-algebras $(X,x)$ and $(Y,y)$, and a lax $T$-morphism $f:X \to Y$ with coherence datum $\overline{f}:yT(f) \to fx$, determines the object $(yT(f),yT(\overline{f}))$ of the category $\tn{Desc}(\Algs T(\ca R_TX,Y))$. Conversely, given an object
\[ \begin{array}{lccr} {f_0:TX \to Y} &&& {f_1:f_0\mu^T_X \to f_0T(x)} \end{array} \]
of $\tn{Desc}(\Algs T(\ca R_TX,Y))$, the data
\[ \begin{array}{lccr} {f_0T(\eta_X):X \to Y} &&& {f_1\eta^T_{TX}:yT(f_0\eta^T_X) \to f_0\eta^T_Xx} \end{array} \]
is that of a lax $T$-morphism $X \to Y$. As observed by Street in \cite{Street-CorrFibrationInBicats} the processes just described are the effect on objects of an isomorphism of categories
\begin{equation}\label{eq:homs-TAlgl-via-Desc}
\Algl T(X,Y) \iso \tn{Desc}(\Algs T(\ca R_TX,Y)).
\end{equation}
In the present situation, given a strict $S$-algebra $X$ and a strict $T$-algebra $Y$, we thus have natural isomorphisms
\begin{equation}\label{eq:adjunctions-for-dagger-F} 
\begin{array}{rcl} {\Algl S(X,\overline{F}Y)}
& {\iso} & {\tn{Desc}(\Algs S(\ca R_SX,\overline{F}Y))} \\
& {\iso} & {\tn{Desc}(\Algs T(\ca R_FX,Y))} \\
& {\iso} & {\Algs T(\tn{CoDesc}(\ca R_FX),Y)} \end{array}
\end{equation}
the first of which is from (\ref{eq:homs-TAlgl-via-Desc}), the second is from the definition of $\ca R_F$, and the third is from the definition of codescent objects.
\end{proof}
It will also be useful to have an explicit description of the component
\[ g^F_X : X \longrightarrow \overline{F}X^{\dagger}_F \]
at a strict $T$-algebra $(X,x)$, of the unit of the adjunction $(-)^{\dagger}_F \ladj \overline{F}J_T$ in terms of a given universal codescent cocone for $\ca R_FX$. As a lax $T$-morphism, its underlying $1$-cell and coherence data will be of the form
\[ \begin{array}{lccr} {g^F_X : X \to F^*X^{\dagger}_F} &&& {\overline{g^F_X} : F^*(x^{\dagger}_F)F^l_{X^{\dagger}_F}T(g^F_X) \to g^F_Xx} \end{array} \]
where $x^{\dagger}_F:TX^{\dagger}_X \to X^{\dagger}_F$ denotes the $T$-algebra structure. So we suppose that a universal codescent cocone for $\ca R_FX$
\[ \begin{array}{lccr} {q^F_{X,0} : TF_!X \to X^{\dagger}_F} &&& {q^F_{X,1} : q^F_{X,0}\mu^T_{F_!X}T(F^c_X) \to q^F_{X,0}TF_!(x)} \end{array} \]
is given, and then we have
\begin{lem}\label{lem:unit-in-terms-of-codescent}
Let $T$ be a 2-monad on a 2-category $\ca K$ such that $\Algs T$ admits codescent objects. Suppose one has a universal codescent cocone for $\ca R_FX$ as above defining the effect on a strict $T$-algebra $(X,x)$ of the left adjoint $(-)^{\dagger}_F$ of Proposition(\ref{prop:dagger-F}). Then
\[ \begin{array}{lccr} {g^F_X = F^*(q^F_{X,0})F^*(\eta^T_{F_!X})\eta^F_X} &&& {\overline{g^F_X} = F^*(q^F_{X,1})F^*(\eta^T_{F_!SX})\eta^F_{SX}} \end{array} \]
then describes the data of the unit component $g^F_X:X \to \overline{F}X^{\dagger}_F$ in $\Algl S$.
\end{lem}
\begin{proof}
The required unit is by definition the effect of the composite isomorphisms of (\ref{eq:adjunctions-for-dagger-F}) on $1_{X^{\dagger}_F} \in \Algs T(X^{\dagger}_F,X^{\dagger}_F)$, and the third of these isomorphisms sends $1_{X^{\dagger}_F}$ to the given universal codescent cocone $(q^F_{X,0},q^F_{X,1})$ for $\ca R_FX$. Applying the second isomorphism of (\ref{eq:adjunctions-for-dagger-F}) to this, gives a codescent cocone for $\ca R_SX$ with vertex $\overline{F}X^{\dagger}_F$ which we now describe explicitly. Note that this isomorphism came from (\ref{eq:RF-defining-nat-iso}) above, which in turn can be written explicitly in terms of the adjunctions $F^S \ladj U^S$, $F_! \ladj F^*$ and $F^T \ladj U^T$. Doing so gives the $1$-cell datum of the $\ca R_SX$-codescent cocone as the upper most composite in the commutative diagram
\[ \xygraph{!{0;(2.5,0):(0,.4)::} {SX}="p0" [r] {SF^*F_!X}="p1" [r] {SF^*TF_!X}="p2" [r] {SF^*X^{\dagger}_F}="p3" [d] {F^*TX^{\dagger}_F}="p4" [d] {F^*X^{\dagger}_F}="p5" [l] {F^*TF_!X}="p6" [ul] {F^*TF_!X}="p7" [r] {F^*T^2F_!X}="p8"
"p0":"p1"^-{S\eta^F}:"p2"^-{SF^*\eta^T}:"p3"^-{SF^*q^F_{X,0}}:"p4"^-{F^l}:"p5"^-{F^*x^{\dagger}_F}:@{<-}"p6"^-{F^*q^F_{X,0}}:@{<-}"p7"^-{1}:@{<-}"p1"^-{F^l}
"p8"(:@{<-}"p7"_-{F^*T\eta^T},:@{<-}"p2"_-{F^l},:"p4"^-{F^*Tq^F_{X,0}},:"p6"^-{F^*\mu^T_{F_!X}})} \]
and so the $1$-cell part of the required cocone is, more simply, $F^*(q^F_{X,0})F^l_{F_!X}S(\eta^F_X)$. Similarly, the required 2-cell datum is $F^*(q^F_{X,1})F^l_{F_!SX}S(\eta^F_{SX})$. Thus, the data $(g^F_X,\overline{g^F_X})$ is the result of applying the first isomorphism of (\ref{eq:adjunctions-for-dagger-F}) to this last cocone. The effect of this isomorphism was described explicitly in the first part of the proof of Proposition \ref{prop:dagger-F}, and using those details one obtains
\[ \begin{array}{lccr} {g^F_X = F^*(q^F_{X,0})F^l_{F_!X}S(\eta^F_X)\eta^S_X} &&& {\overline{g^F_X} = F^*(q^F_{X,1})F^l_{F_!SX}S(\eta^F_{SX})\eta^S_{SX}.} \end{array} \]
Noting that for all $X \in \ca L$ one has
\[ \xygraph{!{0;(2,0):(0,.5)::} {X}="p0" [r] {F^*F_!X}="p1" [dr] {F^*TF_!X}="p2" [l] {SF^*F_!X}="p3" [l] {SX}="p4" "p0":"p1"^-{\eta^F_X}:"p2"^-{F^*\eta^T_{F_!X}}:@{<-}"p3"^-{F^l_{F_!X}}:@{<-}"p4"^-{S\eta^F_X}:@{<-}"p0"^-{\eta^S_X} "p1":"p3"_-{\eta^S_{F^*F_!X}}} \]
commutative, the result follows.
\end{proof} 
\begin{rem}\label{rem:universal-int-alg-explicit}
Given an adjunction $F : (\ca K,S) \to (\ca L,T)$ of 2-monads such that $\Algs T$ admits all codescent objects and $\ca L$ has a terminal object $1$, Lemma \ref{lem:unit-in-terms-of-codescent} in the case $X = 1$ is an explicit description of the universal $S$-algebra internal to a $T$-algebra $g^S_T : 1 \to \overline{F}T^S$.
\end{rem}
The codescent object appearing in (\ref{eq:dagger-F-codesc-formula}) is a particular 2-categorical colimit in $\Algs T$. In the absence of further assumptions, an explicit computation of such a colimit in terms of colimits in $\ca K$ is difficult. It will involve a codescent object in $\ca K$, as well as other colimits in organised into a transfinite construction of the sort considered in \cite{Kelly-Transfinite}. Thankfully in the examples of interest for us, the 2-monad $T$ actually preserves the codescent objects that we care about, and so at least those codescent objects in $\Algs T$ are computed as in $\ca K$.

Let us write $T^S$ now for the underlying object in $\ca K$ of $1^{\dagger}_F$, and denote by $a^S : T(T^S) \to T^S$ its $T$-algebra structure. Denoting by $\sigma_T:TU^T \to U^T$ the 2-cell datum of the Eilenberg-Moore object of $T$, which in explicit terms is the 2-natural transformation with components $(\sigma_T)_{(X,x)} = x$, we obtain the following immediate corollary of Proposition \ref{prop:dagger-F}.
\begin{cor}\label{cor:alg-class-explicit}
Let $F:(\ca L,S) \to (\ca K, T)$ be an adjunction of 2-monads, $T$ have rank, $\ca K$ have all codescent objects, and $\ca L$ have a terminal object $1$. If $T$ and $T^2$ preserve the codescent object of $U^T\ca R_F1$, then
\[ \begin{array}{lccr} {T^S = \tn{CoDesc}(U^T\ca R_F1)} &&& {a^S = \tn{CoDesc}(\sigma_T\ca R_F1).} \end{array} \]
\end{cor}
In other words, applying $\tn{CoDesc} : [\Delta^{\op},\ca K] \to \ca K$ to the morphism $\sigma_T\ca R_F1$ of simplicial objects in $\ca K$ whose codescent-relevant parts are depicted in
\[ \xygraph{!{0;(3.5,0):(0,.5)::}
{T^2F_!S^21}="pt0" [r] {T^2F_!S1}="pt1" [r] {T^2F_!1}="pt2"
"pt2":"pt1"|-{T^2F_!\eta^S_1} "pt1":@<1.5ex>"pt2"^-{T\mu^T_{F_!1}T^2(F^c_1)} "pt1":@<-1.5ex>"pt2"_-{T^2F_!(t_{S1})} "pt0":@<1.5ex>"pt1"^-{T\mu^T_{F_!S1}T^2(F^c_{S1})} "pt0":"pt1"|-{T^2F_!\mu^S_1} "pt0":@<-1.5ex>"pt1"_-{T^2F_!S(t_{S1})}
"pt0" [d] {TF_!S^21}="pb0" [r] {TF_!S1}="pb1" [r] {TF_!1}="pb2"
"pb2":"pb1"|-{TF_!\eta^S_1} "pb1":@<1.5ex>"pb2"^-{\mu^T_{F_!1}T(F^c_1)} "pb1":@<-1.5ex>"pb2"_-{TF_!(t_{S1})} "pb0":@<1.5ex>"pb1"^-{\mu^T_{F_!S1}T(F^c_{S1})} "pb0":"pb1"|-{TF_!\mu^S_1} "pb0":@<-1.5ex>"pb1"_-{TF_!S(t_{S1})}
"pt0":"pb0"|-{\mu^T_{F_!S^21}} "pt1":"pb1"|-{\mu^T_{F_!S1}} "pt2":"pb2"|-{\mu^T_{F_!1}}} \]
where $t_{S1} : S1 \to 1$ is the unique map, gives the $T$-algebra structure map $a^S$.
\begin{exams}\label{exams:easy-actions-operadic-examples}
The middle maps of the polynomials underlying $\tnb{M}$, $\tnb{B}$ and $\tnb{S}$ are discrete fibrations with finite fibres. Thus by \cite{Weber-PolynomialFunctors} Theorem 4.5.1 these 2-monads preserve sifted colimits, and thus in particular codescent objects by \cite{Lack-Codescent} Proposition 4.3. Moreover since the property of being a discrete fibration with finite fibres is stable by pullback along arbitrary functors, the 2-monad corresponding to any operad as in \cite{Weber-OpPoly2Mnd} also preserves sifted colimits. Similarly for 2-monads arising from non-symmetric or braided operads in the manner discussed in Section \ref{ssec:polynomials}. Thus the adjunction of 2-monads arising from any morphism of operads as in Examples \ref{exams:adj-monad-op-morphism} or its variants of Remark \ref{rem:nonsym-braided-internal-algebras} conform to Corollary \ref{cor:alg-class-explicit}.
\end{exams}

\subsection{Category objects from bar constructions.}
\label{ssec:IntAlgClass-first-examples}
All of the examples of simplicial objects of which we take codescent are in fact category objects by
\begin{prop}\label{prop:RF-cat-object}
If $F : (\ca L,S) \to (\ca K,T)$ is an adjunction of 2-monads such that the naturality squares of $\mu^T$ and $F^c$ are pullbacks, $T$ preserves pullbacks, and $(A,a)$ is a strict $S$-algebra, then $\ca R_FA$ is a category object.
\end{prop}
\begin{proof}
As we recalled in Example \ref{ex:internal-nerve} a category object in $\ca K$ is a simplicial object $Y \in [\Delta^{\op},\ca K]$ such that for all $n \in \N$ the square
\[ \xygraph{{\xybox{\xygraph{!{0;(1.5,0):(0,.6667)::} {Y_{n+2}}="tl" [r] {Y_{n+1}}="tr" [d] {Y_n}="br" [l] {Y_{n+1}}="bl" "tl":"tr"^-{d_{n+2}}:"br"^-{d_0}:@{<-}"bl"^-{d_{n+1}}:@{<-}"tl"^-{d_0}}}}
[r(5)u(.03)]
{\xybox{\xygraph{!{0;(2.5,0):(0,.4)::} {TF_!S^{n+2}A}="p0" [r] {T^2F_!S^{n+1}A}="p1" [r] {TF_!S^{n+1}A}="p2" [d] {TF_!S^nA}="p3" [l] {T^2F_!S^nA}="p4" [l] {TF_!S^{n+1}A}="p5" "p0":"p1"^-{T(F^c_{S^{n+1}A})}:"p2"^-{\mu^T_{F_!S^{n+1}A}}:"p3"^-{TF_!S^na}:@{<-}"p4"^-{\mu^T_{F_!S^nA}}:@{<-}"p5"^-{T(F^c_{S^nA})}:@{<-}"p0"^-{TF_!S^{n+1}a} "p1":"p4"^{T^2F_!S^na}}}}} \]
on the left is a pullback. In the case of $\ca R_FA$ these squares decompose horizontally as on the right in the previous display.
\end{proof}
\begin{rem}\label{rem:basic-bar-as-cat-object}
Applying Proposition \ref{prop:RF-cat-object} when $F$ is an identity one recovers the well known fact that for any cartesian 2-monad $T$ and strict $T$-algebra $A$, $\ca R_TA$ is a category object.
\end{rem}
\begin{exams}\label{exams:easy-actions-internalised-examples}
Many examples, such as those of \cite{Batanin-EckmannHilton, BataninBerger-HtyThyOfAlgOfPolyMnd}, arise in the following way. One begins with an adjunction of monads $F : (\ca D,S) \to (\ca E,T)$ in the 2-category $\Cart$, of categories with pullbacks, pullback preserving functors and cartesian natural transformations. Applying the 2-functor $\ca E \mapsto \tn{Cat}(-)$ to this gives an adjunction of 2-monads, which for the sake of brevity we denote as $F : (\tn{Cat}(\ca D),S) \to (\tn{Cat}(\ca E),T)$. In this situation $\ca R_F1$ is a componentwise discrete, a category object by Proposition \ref{prop:RF-cat-object}, and so its codescent object is obtained as in Example \ref{ex:internal-nerve}, and is clearly preserved by composites of $T$. Thus such situations also conform to Corollary \ref{cor:alg-class-explicit}, and the underlying object of $T^S$ is
\[ \xygraph{!{0;(3,0):(0,1)::} {TF_!S^21}="p0" [r] {TF_!S1}="p1" [r] {TF_!1}="p2"
"p2":"p1"|-{TF_!\eta^S_1} "p1":@<1.5ex>"p2"^-{\mu^T_{F_!1}T(F^c_1)} "p1":@<-1.5ex>"p2"_-{TF_!t_{S1}} "p0":@<1.5ex>"p1"^-{\mu^T_{F_!S1}T(F^c_{S1})} "p0":"p1"|-{TF_!\mu^S_1} "p0":@<-1.5ex>"p1"_-{TF_!St_{S1}}} \]
viewed as an object of $\tn{Cat}(\ca E)$.
\end{exams}
We recall the most basic instance of Example \ref{exams:easy-actions-internalised-examples} in
\begin{exam}\label{exam:delta-plus}
For $\tnb{M}$ the 2-monad for monoidal categories recalled in Section \ref{ssec:polynomials}, the strict monoidal structure $\tnb{M}(\tnb{M}^{\tnb{M}}) \to \tnb{M}^{\tnb{M}}$ is the functor which is sent by the nerve functor to the morphism
\[ \xygraph{!{0;(3,0):(0,.4)::}
{\tnb{M}^41}="t0" [r] {\tnb{M}^31}="t1" [r] {\tnb{M}^21}="t2"
"t2":"t1"|-{\tnb{M}^2\eta_1} "t1":@<1.5ex>"t2"^-{\tnb{M}\mu_1} "t1":@<-1.5ex>"t2"_-{\tnb{M}^2t_{M1}} "t0":@<1.5ex>"t1"^-{\tnb{M}\mu_{\tnb{M}1}} "t0":"t1"|-{\tnb{M}^2\mu_1} "t0":@<-1.5ex>"t1"_-{\tnb{M}^3t_{M1}}
"t0" [d] {\tnb{M}^31}="b0" [r] {\tnb{M}^21}="b1" [r] {\tnb{M}1}="b2"
"b2":"b1"|-{\tnb{M}\eta_1} "b1":@<1.5ex>"b2"^-{\mu_1} "b1":@<-1.5ex>"b2"_-{\tnb{M}t_{M1}} "b0":@<1.5ex>"b1"^-{\mu_{\tnb{M}1}} "b0":"b1"|-{\tnb{M}\mu_1} "b0":@<-1.5ex>"b1"_-{\tnb{M}^2t_{M1}}
"t0":"b0"_-{\mu_{\tnb{M}^21}} "t1":"b1"^-{\mu_{\tnb{M}1}} "t2":"b2"^-{\mu_1}} \]
of simplicial sets. Thus an object of $\tnb{M}^{\tnb{M}}$ is an element of $\tnb{M}1$, which is a natural number. A morphism of $\tnb{M}^{\tnb{M}}$ is a sequence of natural numbers $(m_1,...,m_n)$, the source of which is $m = m_1 + ... + m_n$, and the target of which is $n$. Thus we identify $(m_1,...,m_n)$ with the order preserving function $\underline{m} \to \underline{n}$ whose fibre over $k \in n$ has cardinality $m_k$. In these terms $\tnb{M}\eta_1$ and $\tnb{M}\mu_1$ correspond to the identities and composition of order preserving functions, and so the bottom row exhibits the underlying object of $\tnb{M}^{\tnb{M}}$ as identifiable with $\Delta_+$. The effect of the action $\tnb{M}(\tnb{M}^{\tnb{M}}) \to \tnb{M}^{\tnb{M}}$ on objects is $\mu_1$, and so corresponds with ordinal sum, and similarly on morphisms. Thus one recovers $\Delta_+$ as the monoid classifier.
\end{exam}
If $\ca K = \Cat$ in the situation of Remark \ref{rem:basic-bar-as-cat-object}, then $\ca U^TR_TA$ is a category object in $\Cat$. A category object $X$ in $\Cat$ is commonly known as a \emph{double category}, and we now fix our double categorical terminology and conventions. The structure of a double category $X$ includes categories and functors as in
\[ \xygraph{!{0;(2,0):(0,1)::} {X_2}="p0" [r] {X_1}="p1" [r] {X_0}="p2"
"p2":"p1"|-{s_0} "p1":@<1.5ex>"p2"^-{d_1} "p1":@<-1.5ex>"p2"_-{d_0} "p0":@<1.5ex>"p1"^-{d_2} "p0":"p1"|-{d_1} "p0":@<-1.5ex>"p1"_-{d_0}} \]
and so in elementary terms, $X$ consists of (1) \emph{objects} -- which are the objects of $X_0$, (2) \emph{vertical arrows} -- which are the arrows of $X_0$, (3) \emph{horizontal arrows} -- which are the objects of $X_1$, and (4) \emph{squares} -- which are the arrows of $X_1$. A typical square $\alpha$ of $X$ is drawn as
\[ \xygraph{{x_1}="tl" [r] {x_2}="tr" [d] {x_4}="br" [l] {x_3}="bl" "tl":"tr"^-{h_1}:"br"^-{v_2}:@{<-}"bl"^-{h_2}:@{<-}"tl"^-{v_1}:@{}"br"|*{\alpha}} \]
and has vertical source and target arrows -- in this case $d_1\alpha = v_1$ and $d_0\alpha = v_2$, and horizontal source and target arrows -- in this case $h_1$ and $h_2$, which are the source and target of $\alpha$ as an arrow of $X_1$.

For each object $x$, the category structure of $X_0$ provides us with a vertical identity arrow $1^{\tn{v}}_x$ on $x$, and for any composable pair $(f,g)$ of vertical arrows, a composite $g \comp_{\tn{v}} f$. The horizontal arrow $s_0x$ is denoted $1^{\tn{h}}_x$ and called the horizontal identity on $x$, and for any object of $X_2$, which is a composable pair $(f,g)$ of horizontal arrows, $d_1(f,g)$ is denoted $g \comp_{\tn{h}} f$. We shall write either of these composition operations as concatenation when no confusion would result from doing so.

For any vertical arrow $f$, the identity square $s_0f$ on $f$ is drawn as on the left
\[ \xygraph{{\xybox{\xygraph{{x}="p0" [r] {x}="p1" [d] {y}="p2" [l] {y}="p3" "p0":"p1"^-{1^{\tn{h}}_x}:"p2"^-{f}:@{<-}"p3"^-{1^{\tn{h}}_y}:@{<-}"p0"^-{f}:@{}"p2"|-*{\id^{\tn{h}}_f}}}}
[r(3)]
{\xybox{\xygraph{{x}="p0" [r] {y}="p1" [d] {y}="p2" [l] {x}="p3" "p0":"p1"^-{g}:"p2"^-{1^{\tn{v}}_y}:@{<-}"p3"^-{g}:@{<-}"p0"^-{1_x^{\tn{v}}}:@{}"p2"|-*{\id^{\tn{v}}_g}}}}} \]
and for any horizontal arrow $g$, the identity on $g$ in the category $X_1$ is denoted as on the right. We shall omit the $(-)^{\tn{h}}$ and $(-)^{\tn{v}}$ superscripts when no confusion would arise from doing so.

A composable pair of arrows of $X_1$ is a configuration as in the first diagram of
\[ \xygraph{{\xybox{\xygraph{!{0;(.8,0):(0,1.25)::} {x_1}="p0" [r] {x_2}="p1" [d] {x_4}="p2" [d] {x_6}="p3" [l] {x_5}="p4" [u] {x_3}="p5" "p0":"p1"^-{h_1}:"p2"^-{v_2}:"p3"^-{v_4}:@{<-}"p4"^-{h_3}:@{<-}"p5"^-{v_3}:@{<-}"p0"^-{v_1} "p5":"p2"^{h_2} "p2" (:@{}"p0"|-*{\alpha},:@{}"p4"|-*{\beta})}}}
[r(2)]
{\xybox{\xygraph{!{0;(1.25,0):(0,.8)::} {x_1}="p0" [r] {x_2}="p1" [d] {x_6}="p2" [l] {x_5}="p3" "p0":"p1"^-{h_1}:"p2"^-{v_4v_2}:@{<-}"p3"^-{h_3}:@{<-}"p0"^-{v_3v_1}:@{}"p2"|-*{\beta \comp_{\tn{v}} \alpha}}}}
[r(2.5)]
{\xybox{\xygraph{!{0;(.8,0):(0,1.25)::} {x_1}="p0" [r] {x_2}="p1" [r] {x_3}="p2" [d] {x_6}="p3" [l] {x_5}="p4" [l] {x_4}="p5" "p0":"p1"^-{h_1}:"p2"^-{h_2}:"p3"^-{v_3}:@{<-}"p4"^-{h_4}:@{<-}"p5"^-{h_3}:@{<-}"p0"^-{v_1} "p1":"p4"|-{v_2} "p1" (:@{}"p5"|-*{\alpha},:@{}"p3"|-*{\beta})}}}
[r(2.5)]
{\xybox{\xygraph{!{0;(1.25,0):(0,.8)::} {x_1}="p0" [r] {x_3}="p1" [d] {x_6}="p2" [l] {x_4}="p3" "p0":"p1"^-{h_2h_1}:"p2"^-{v_3}:@{<-}"p3"^-{h_4h_3}:@{<-}"p0"^-{v_1} :@{}"p2"|-*{\beta \comp_{\tn{h}} \alpha}}}}} \]
and its composite is as denoted in the second diagram. In diagramatic contexts, we shall speak of the composite embodied by the left-most diagram, leaving the notation $\beta \comp_{\tn{v}} \alpha$ mostly for equations. Similarly, a morphism of $X_2$ is a configuration as in the third diagram of the previous display, and its composite, which is the effect of $d_1$ on this arrow of $X_2$, is as denoted on the right. The functoriality of $d_1 : X_2 \to X_1$ implies that the value of the composite
\[ \xygraph{{x_1}="p0" [r] {x_2}="p1" [r] {x_3}="p2" [d] {x_6}="p3" [d] {x_9}="p4" [l] {x_8}="p5" [l] {x_7}="p6" [u] {x_4}="p7" [r] {x_5}="p8" "p0":"p1"^-{h_1}:"p2"^-{h_2}:"p3"^-{v_3}:"p4"^-{v_6}:@{<-}"p5"^-{h_6}:@{<-}"p6"^-{h_5}:@{<-}"p7"^-{v_4}:@{<-}"p0"^-{v_1} "p8" (:@{<-}"p7"^-{h_3},:@{<-}"p1"^-{v_2},:"p3"^-{h_4},:"p5"^-{v_5}, :@{}"p0"|-*{\alpha}, :@{}"p2"|-*{\beta}, :@{}"p4"|-*{\delta}, :@{}"p6"|-*{\gamma})} \]
is unambiguous, or in equational terms, that
\[ (\delta \comp_{\tn{h}} \gamma) \comp_{\tn{v}} (\beta \comp_{\tn{h}} \alpha) = (\delta \comp_{\tn{v}} \beta) \comp_{\tn{h}} (\gamma \comp_{\tn{v}} \alpha). \]

An internal functor $f:X \to Y$ in $\Cat$ is commonly known as a \emph{double functor}. In explicit terms $f$ sends objects, vertical arrows, horizontal arrows and squares in $X$ to objects, vertical arrows, horizontal arrows and squares in $Y$, in a manner compatible with sources, targets, identities and compositions. Having recalled these preliminaries, we come to the last example of this section.
\begin{exam}\label{exam:dblcat-for-cmon-classifier}
The double category $U^{\tnb{S}}\ca R_{\tnb{S}}1$, whose codescent-relevant parts we write as
\[ \xygraph{!{0;(3,0):(0,1)::} {\tnb{S}^2(\mathbb{P})}="p0" [r] {\tnb{S}(\mathbb{P})}="p1" [r] {\mathbb{P}}="p2"
"p2":"p1"|-{\tnb{S}(\eta_1)} "p1":@<1.5ex>"p2"^-{\mu_1} "p1":@<-1.5ex>"p2"_-{\tnb{S}(t_{\mathbb{P}})} "p0":@<1.5ex>"p1"^-{\mu_{\mathbb{P}}} "p0":"p1"|-{\tnb{S}(\mu_1)} "p0":@<-1.5ex>"p1"_-{\tnb{S}^2(t_{\mathbb{P}})}} \]
is described as follows. At the level of objects we are in the situation of Example \ref{exam:delta-plus}, that is $\tn{ob}\,U^{\tnb{S}}\ca R_{\tnb{S}}1 = U^{\tnb{M}}\ca R_{\tnb M}1$. Thus the category of objects and horizontal arrows of $U^{\tnb{S}}\ca R_{\tnb{S}}1$ is $\Delta_+$. By definition the category of objects and vertical maps of $U^{\tnb{S}}\ca R_{\tnb{S}}1$ is $\P$. To give a square in $U^{\tnb{S}}\ca R_{\tnb{S}}1$ is to give a morphism of $\tnb{S}(\mathbb{P})$, which is to give $n, m_1, ..., m_n \in \N$, $\rho \in \Sigma_n$, and $\rho_k \in \Sigma_{m_k}$ for $1 \leq k \leq n$. We depict such a square as on the left
\[ \xygraph{{\xybox{\xygraph{!{0;(2.5,0):(0,.4)::} {\underline{m}}="p0" [r] {\underline{n}}="p1" [d] {\underline{n}}="p2" [l] {\underline{m}}="p3" "p0":"p1"^-{(m_{\rho 1},...,m_{\rho n})} :"p2"^-{\rho}:@{<-}"p3"^-{(m_1,...,m_n)}:@{<-}"p0"^-{\rho(\rho_k)_k}:@{}"p2"|-*{(\rho,(\rho_k)_k)}}}}
[r(5)]
{\xybox{\xygraph{!{0;(2.5,0):(0,.4)::} {\underline{m}}="p0" [r] {\underline{n}}="p1" [d] {\underline{n}}="p2" [l] {\underline{m}}="p3" "p0":"p1"^-{f} :"p2"^-{\rho}:@{<-}"p3"^-{g}:@{<-}"p0"^-{\rho'}:@{}"p2"|-*{=}}}}} \]
so as to make explicit its horizontal and vertical sources and targets. By the definition of $\rho(\rho_k)_k$ one has
\[ \rho \comp (m_{\rho 1},...,m_{\rho n}) = (m_1,...,m_n) \comp \rho(\rho_k)_k \]
in $\Set$. Conversely given monotone functions $f$ and $g$, and permutations $\rho$ and $\rho'$ making the square on the right commute, one obtains $\rho_k \in \Sigma_{|f^{-1}\{k\}|}$ by restricting $\rho'$ to the fibre $f^{-1}\{k\}$, and then $\rho' = \rho(\rho_k)_k$. Thus a square of $U^{\tnb{S}}\ca R_{\tnb{S}}1$ is determined uniquely by its boundary, and a square will exist with the boundary $(f,\rho,g,\rho')$ as on the right in the previous display iff $\rho f = g \rho'$ in $\Set$. Moreover one easily checks that horizontal and vertical composition of squares in $U^{\tnb{S}}\ca R_{\tnb{S}}1$ corresponds to horizontal and vertical pasting of such squares.
\end{exam}

\section{Codescent for crossed internal categories.}

In this section we present our general method of calculation of codescent objects and internal algebra classifiers. In Section \ref{ssec:CrIntCats} crossed internal categories are defined and explained. Then in Section \ref{ssec:cic-from-adj2mnd} abstract conditions on an adjunction of 2-monads are given, that ensure that the simplicial object participating in the calculation of the corresponding internal algebra classifier, is a crossed internal category. As it turns out later in the section, the calculation of codescent objects of crossed internal categories is done in two steps, the first of which produces an internal 2-category, called the 2-category of corners. This 2-category is described in Section \ref{ssec:Corners}, and in Section \ref{ssec:CodescComp} the general results on the calculation of codescent objects are given.

\subsection{Crossed internal categories.}
\label{ssec:CrIntCats}
Obtaining $\Delta_+$ as in Example \ref{exam:delta-plus} from the 2-monad $\tnb{M}$ via the computation of codescent objects is well-known \cite{Batanin-EckmannHilton, Bourke-Thesis}. If however one attempts to tell the same story for the 2-monad $\tnb{S}$ instead of $\tnb{M}$, one runs into the problem identified in Example \ref{exam:bar-of-Sm-not-a-catead}, that the resulting codescent object is not a catead. Thus Proposition \ref{prop:Bourke-codesc-catead} does not apply. In this section we identify the structure of \emph{crossed internal category}, which is enjoyed by the simplicial object of Example \ref{exam:bar-of-Sm-not-a-catead}, and which enables its codescent object to be computed in Section \ref{ssec:functions}.

The structure of a crossed internal category involves the notion of an opfibration within a 2-category \cite{Street-FibrationIn2cats}. We now briefly recall this background, in a manner compatible notationally and terminologically with the fuller recollection of this theory given in section(4.2) of \cite{Weber-PolynomialFunctors}.

For a functor $f : A \to B$, an arrow $\alpha : a_1 \to a_2$ in $A$ is \emph{$f$-opcartesian} when for all $\gamma : a_1 \to a_3$ and $\beta:fa_2 \to fa_3$ such that $\beta f(\alpha) = f(\gamma)$, there is a unique $\overline{\beta} : a_2 \to a_3$ such that $f\overline{\beta} = \beta$ and $\overline{\beta}\alpha = \gamma$. To give $f$ the structure of an \emph{opfibration} is to give, for each $a \in A$ and $\beta : fa \to b$, an $f$-opcartesian arrow $\overline{\beta} : a \to a'$ such that $f\overline{\beta} = \beta$. The arrows of $A$ that arise as $\overline{\beta}$ for some $(a,\beta)$ are said to be \emph{chosen} $f$-opcartesian. When the chosen $f$-opcartesian arrows include all the identity arrows, and are closed under composition, $f$ said to be a \emph{split} opfibration.

For an arrow $f : A \to B$ in arbitrary 2-category $\ca K$, given arrows $a_1 : X \to A$ and $a_2 : X \to A$, an $f$-opcartesian 2-cell $\alpha : a_1 \to a_2$ is one which is $\ca K(X,f)$-opcartesian as an arrow of $\ca K(X,A)$, where $\ca K(X,f) : \ca K(X,A) \to \ca K(X,B)$ is the functor given by composition with $f$. To give $f$ the structure of a split opfibration is to give $\ca K(X,f)$ the structure of a split opfibration in the sense of the previous paragraph for all $X \in \ca K$, such that for all $g : Y \to X$, the functor $\ca K(g,A) : \ca K(X,A) \to \ca K(Y,A)$ given by precomposing with $g$ preserves chosen opcartesian arrows. In this context we call the chosen opcartesian arrows of $\ca K(X,f)$ for some $X$, the \emph{chosen $f$-opcartesian 2-cells}. By definition, all identity 2-cells between arrows with codomain $A$ are chosen $f$-opcartesian, and chosen $f$-opcartesian 2-cells are closed under vertical composition.

At this generality split opfibrations can be composed, when $\ca K$ has pullbacks split opfibrations can be pulled back along arbitrary maps in $\ca K$, and when $\ca K$ has comma objects one has an easy to describe 2-monad $\Psi_{\ca K}$ on the 2-category $\ca K^{[1]}$ of arrows of $\ca K$, whose strict algebras are exactly the split opfibrations in $\ca K$. A strict morphism $f \to g$ of $\Psi_{\ca K}$-algebras is a pair $(u,v)$ fitting into a commutative square
\[ \xygraph{{\xybox{\xygraph{!{0;(1.5,0):(0,1)::} {A}="p0" [r] {C}="p1" [d] {D}="p2" [l] {B}="p3" "p0":"p1"^-{u}:"p2"^-{g}:@{<-}"p3"^-{v}:@{<-}"p0"^-{f}}}}
[r(4)]
{\xybox{\xygraph{!{0;(1.5,0):(0,1)::} {A}="p0" [r] {C}="p1" [d] {D}="p2" [l] {B}="p3" "p0":@/^{1pc}/"p1"^(.4){u_1}|(.4){}="pdtop":"p2"^-{g}:@/_{1pc}/@{<-}"p3"_(.6){v_1}|(.6){}="pdbot":@{<-}"p0"^-{f}
"p0":@/_{1pc}/"p1"_(.6){u_2}|(.6){}="pctop" "p3":@/_{1pc}/"p2"_(.6){v_2}|(.6){}="pcbot"
"pdtop":@{}"pctop"|(.25){}="dtop"|(.75){}="ctop" "dtop":@{=>}"ctop"^{\alpha}
"pdbot":@{}"pcbot"|(.25){}="dbot"|(.75){}="cbot" "dbot":@{=>}"cbot"^{\beta}}}}} \]
as on the left, such that post-composition with $u$ sends chosen $f$-opcartesian 2-cells to chosen $g$-opcartesian 2-cells. A 2-cell $(u_1,v_1) \to (u_2,v_2)$ of $\Algs {\Psi_{\ca K}}$ consists of a pair $(\alpha,\beta)$ fitting into a commutative cylinder as on the right.

With these necessary notions recalled, we are now in a position to give the central definition of this article.
\begin{defn}\label{def:crossed-internal-cat}
A \emph{crossed internal category} in a 2-category $\ca K$ is a category object $X:\Delta^{\op} \to \ca K$, together with the structure of a split opfibration on $d_0:X_1 \to X_0$ such that 
\[ \xygraph{{X_0}="p1" [r] {X_1}="p2" [r] {X_2}="p3" [dl] {X_0}="p4" "p1":"p2"^-{s_0}:@{<-}"p3"^-{d_1}:"p4"^-{d_0^2}:@{<-}"p1"^-{1} "p2":"p4"^{d_0}} \]
are morphisms of split opfibrations over $X_0$. When $\ca K = \Cat$ we shall say that $X$ is a \emph{crossed double category}.
\end{defn}
To make sense of the definition, as recalled above, given
\[ \xygraph{{A}="tl" [r] {B}="tr" [d] {D}="br" [l] {C}="bl" "tl":"tr"^-{f'}:"br"^-{g}:@{<-}"bl"^-{f}:@{<-}"tl"^-{h}:@{}"br"|-{\tn{pb}}} \]
in $\ca K$, a split opfibration structure on $f$ gives rise to a split opfibration structure on $f'$ such that the above square becomes a morphism $f' \to f$ of split opfibrations. Moreover split opfibrations compose. Thus for a category object $X$ as in the definition, the structure of split opfibration on $d_0:X_1 \to X_0$, gives split opfibration structures to the $d_0^n:X_n \to X_0$.
\begin{defn}\label{def:crossed-internal-functor}
Let $X$ and $Y$ be crossed internal categories in a 2-category $\ca K$. A \emph{crossed internal functor} $f:X \to Y$ is a simplicial morphism such that the square
\[ \xygraph{{X_1}="tl" [r] {X_0}="tr" [d] {Y_0}="br" [l] {Y_1}="bl" "tl":"tr"^-{d_0}:"br"^-{f_0}:@{<-}"bl"^-{d_0}:@{<-}"tl"^-{f_1}} \]
is a morphism from $d_0:X_1 \to X_0$ to $d_0:Y_1 \to Y_0$ of split opfibrations. When $\ca K = \Cat$ we shall say that $f$ is a \emph{crossed double functor}.
\end{defn}
When $\ca K = \Cat$ a crossed internal category is a double category with extra structure, and now we unpack what this amounts to in elementary terms. A square $\alpha$ as on the left
\[ \xygraph{{\xybox{\xygraph{{w}="tl" [r] {x}="tr" [d] {z}="br" [l] {y}="bl" "tl":"tr"^-{f}:"br"^-{g}:@{<-}"bl"^-{h}:@{<-}"tl"^-{k}:@{}"br"|*{\alpha}}}} [r(2.5)]
{\xybox{\xygraph{{w}="tl" [r] {x}="tr" [d] {z}="br" [d] {b}="bbr" [l] {a}="bbl" "tl":"tr"^-{f}:"br"^-{g}:"bbr"^-{j}:@{<-}"bbl"^-{i}:@{<-}"tl"^-{l}:@{}"bbr"|*{\beta}}}} [r(1.25)] {=} [r(1.25)]
{\xybox{\xygraph{{w}="tl" [r] {x}="tr" [d] {z}="br" [d] {b}="bbr" [l] {a}="bbl" [u] {y}="bl" "tl":"tr"^-{f}:"br"^-{g}:@{<-}"bl"|-{h}:@{<-}"tl"^-{k}:@{}"br"|*{\alpha}
"br":"bbr"^-{j}:@{<-}"bbl"^-{i}:@{<-}"bl"^-{m}:@{}"bbr"|*{\gamma}
}}}} \]
in $X$ is $d_0$-opcartesian as an arrow of $X_1$ iff given $(i,j,l,\beta)$ as above, there exists unique $(m,\gamma)$ satisfying the equation on the right in the previous display. Henceforth we shall say that such an $\alpha$ is an \emph{opcartesian square}, meaning that it is $d_0$-opcartesian as an arrow of $X_1$.

To provide the chosen opcartesian arrows making $d_0:X_1 \to X_0$ an opfibration, is to give a choice of opcartesian square
\[ \xygraph{{w}="tl" [r] {x}="tr" [d] {y}="br" [l] {z}="bl" "tl":"tr"^-{f}:"br"^-{g}:@{<-}"bl"^-{\rho_{f,g}}:@{<-}"tl"^-{\lambda_{f,g}}:@{}"br"|-*{\kappa_{f,g}}} \]
for each pair $(f,g)$ as shown. This is a split opfibration when the $\kappa_{f,g}$ are compatible with vertical composition in $X$, that is when $\kappa_{f,1}=1_f$ and $\kappa_{\rho_{f,g_1},g_2} \comp_{\tn{v}} \kappa_{f,g_1} = \kappa_{f,g_2g_1}$. To say that $s_0:X_0 \to X_1$ and $d_1:X_2 \to X_1$ underlie morphisms of split opfibrations as in Definition \ref{def:crossed-internal-cat}, is to say that the $\kappa_{f,g}$ are compatible with horizontal composition in $X$. That is, $\kappa_{1,g}=1_g$ and $\kappa_{f_2,g} \comp_{\tn{h}} \kappa_{f_1,\lambda_{f_2,g}} = \kappa_{f_2f_1,g}$. Thus the extra structure on a double category which makes it into a crossed internal category in $\Cat$ is a choice of opcartesian square $\kappa_{f,g}$ for each pair $(f,g)$ as above, which is compatible with vertical and horizontal composition in $X$. A crossed double functor is a double functor which preserves these chosen opcartesian squares. The names chosen in Definitions \ref{def:crossed-internal-cat} and \ref{def:crossed-internal-functor} come from
\begin{exam}\label{ex:Loday-Fieodorowicz}
Crossed simplicial groups of (1.1) of \cite{FieLoday-CrossedSimplicial} may be identified, essentially by Proposition 2.8 of \cite{FieLoday-CrossedSimplicial}, with crossed double categories $X$ such that the $X_n$ are groupoids, the category of objects and horizontal arrows is $\Delta$, and the functor $d_0:X_1 \to X_0$ is a discrete opfibration. The condition on a double category that it be a crossed double category such that $d_0$ is a discrete opfibration was called the ``star condition'' (2.3) of \cite{FieLoday-CrossedSimplicial}. The discreteness of the opfibration $d_0$ amounts to the uniqueness of the squares $\kappa_{f,g}$. In the motivating example of \cite{FieLoday-CrossedSimplicial} relevant for cyclic homology, the category of objects and vertical arrows is the (categorical) coproduct of the cyclic groups.
\end{exam}
\begin{rem}\label{rem:catead-vs-crossed-int-cat}
A componentwise discrete category object $X$ in a 2-category $\ca K$ is a crossed internal category. Moreover, for any catead $X$, the morphism $d_0 : X_1 \to X_0$ has a canonical structure of split opfibration by \cite{Weber-Fam2fun} Theorem 3.5. In fact, with respect to this structure, $d_1 : X_2 \to X_1$ is a morphism of split opfibrations. To see this in the case $\ca K = \Cat$, use Remark 2.86 of \cite{Bourke-Thesis} and the above explicit description of the structure of crossed double categories in elementary terms, and then the general statement follows by a representable argument. However, $s_0 : X_0 \to X_1$ will not be a morphism of split opfibrations unless $X$ is componentwise discrete. So for example, the higher kernel (see \cite{Bourke-Thesis}) of a functor $f : A \to B$ such that $A$ is not discrete, is a catead which is not a crossed internal category. Thus in view of Example \ref{exam:bar-of-Sm-not-a-catead}, crossed internal categories and cateads are different generalisations of componentwise discrete category objects.
\end{rem}

\subsection{Crossed internal categories from adjunctions of 2-monads.}
\label{ssec:cic-from-adj2mnd}
We now identify monad theoretic situations which give examples of crossed internal categories. To this end we recall the notions of opfamilial 2-functor and opfamilial 2-natural transformation from \cite{Weber-Fam2fun, Weber-PolynomialFunctors}.

For a 2-functor $T:\ca K \to \ca L$ and an object $X \in \ca K$, we denote by $T_X : \ca K/X \to \ca L/TX$ the 2-functor given on objects by applying $T$ to morphisms into $X$. A \emph{local right adjoint} $\ca K \to \ca L$ is a 2-functor $T:\ca K \to \ca L$ equipped with a left adjoint to $T_X$ for all $X \in \ca K$. When $\ca K$ has a terminal object $1$, to exhibit $T$ as a local right adjoint, it suffices to give a left adjoint to $T_1$. When $\ca K$ and $\ca L$ have comma objects and $\ca K$ has a terminal object, an \emph{opfamilial 2-functor} $T : \ca K \to \ca L$ is a local right adjoint equipped with $\overline{T}_1 : \ca K \to \Algs {\Psi_{T1}}$ such that $U^{\Psi_{T1}}\overline{T}_1 = T_1$, where $\Psi_{T1}$ is the 2-monad on $\ca L/T1$ whose algebras are opfibrations in $\ca L$ into $T1$. As expressed by Proposition 4.3.3 of \cite{Weber-PolynomialFunctors}, opfamilial 2-functors are those 2-functors which are compatible with the theory of opfibrations. In particular, they preserve split opfibrations.

A 2-natural transformation $\phi : S \to T$ between opfamilial 2-functors is opfamilial when its naturality squares are pullbacks, and when for all $X \in \ca K$, $\alpha$'s naturality square with respect to the unique map $t_X : X \to 1$ is a morphism of split fibrations $(\alpha_X,\alpha_1) : St_X \to Tt_X$. By Lemma 4.3.5 of \cite{Weber-PolynomialFunctors} this implies that the naturality squares of $\phi$ with respect to any split opfibration are morphisms of split opfibrations. An \emph{opfamilial 2-monad} is a 2-monad whose underlying endo-2-functor, and unit and multiplication are opfamilial.
\begin{prop}\label{prop:crossed-int-cats-from-classifiers}
Let $F:(\ca L,S) \to (\ca K, T)$ be an adjunction of 2-monads, $T$ have rank, $\ca K$ have all codescent objects, $\ca L$ have a terminal object $1$, and $T$ and $T^2$ preserve the codescent object of $U^T\ca R_F1$. If the 2-monad $T$ and the 2-functor $F_!$ are opfamilial, and the naturality squares of $F^c$ are pullbacks, then the morphism of simplicial objects
\[ \sigma_T\ca R_F1 : TU^T\ca R_F1 \longrightarrow U^T\ca R_F1 \]
in $\ca K$ of Corollary \ref{cor:alg-class-explicit} is a crossed internal functor between crossed internal categories.
\end{prop}
\begin{proof}
By Proposition \ref{prop:RF-cat-object} $\ca R_F1$ is an internal category, and since $T$ and $U^T$ preserve pullbacks, so are $TU^T\ca R_F1$ and $U^T\ca R_F1$, and so $\sigma_T\ca R_F1$ is an internal functor. In a unique way the diagram on the left
\[ \xygraph{{\xybox{\xygraph{{S^21}="p0" [r] {S1}="p1" [r] {1}="p2" [dl] {1}="p3" "p0":"p1"^-{\mu^S_1}:@{<-}"p2"^-{\eta^S_1}:"p3"^-{1_1}:@{<-}"p0"^-{t_{S^21}}
"p1":"p3"|-{t_{S1}}}}}
[r(4.5)]
{\xybox{\xygraph{!{0;(1.5,0):(0,.6667)::} {TF_!S^21}="p0" [r] {TF_!S1}="p1" [r] {TF_!1}="p2" [dl] {TF_!1}="p3" "p0":"p1"^-{TF_!\mu^S_1}:@{<-}"p2"^-{TF_!\eta^S_1}:"p3"^-{1_{TF_!1}}:@{<-}"p0"^-{TF_!t_{S^21}}
"p1":"p3"|-{TF_!t_{S1}}}}}} \]
depicts morphisms of split opfibrations over $1$. Since $TF_!$ is opfamilial it preserves split opfibrations and morphisms thereof, and so the diagram on the right is a diagram of split opfibrations over $TF_!1$, and so $U^T\ca R_F1$ is a crossed internal category in $\ca K$. Since $T$ preserves split opfibrations and morphisms thereof, $TU^T\ca R_F1$ is also a crossed internal category. Since $T$'s multiplication $\mu^T$ is opfamilial, its naturality square at the split opfibration $F_!t_{S1}$ is a morphism $(\mu^T_{F_!S1},\mu^T_{F_!1})$ of split opfibrations $T^2F_!(t_{S1}) \to TF_!(t_{S1})$, and so $\sigma_T\ca R_F1$ is a crossed internal functor.
\end{proof}
Thus in the context of Proposition \ref{prop:crossed-int-cats-from-classifiers} the underlying object of the corresponding internal algebra classifier $T^S$ is obtained as the codescent object of a crossed internal category in $\ca K$,
and its $T$-algebra structure $a^S : T(T^S) \to T^S$ is similarly obtained from a crossed internal functor.
\begin{exams}\label{exams:int-examples-opfamilial->cic}
For adjunctions of 2-monads arising as in Examples \ref{exams:easy-actions-internalised-examples}, by applying $\tn{Cat}(-)$ to an adjunction of monads in $\Cart$, satisfy the conditions of Proposition \ref{prop:crossed-int-cats-from-classifiers}, by \cite{Weber-Fam2fun} and \cite{Weber-PolynomialFunctors} Remark 4.3.7. However, in this case the conclusion of Proposition \ref{prop:crossed-int-cats-from-classifiers} follows easily and directly, since $TU^T\ca R_F1$ and $U^T\ca R_F1$ are componentwise discrete category objects, and any simplicial morphism between componentwise discrete category objects is a crossed internal functor.
\end{exams}
\begin{exams}\label{exams:opfam->cic-polycases}
In the situation of Examples \ref{exams:adj-monad-op-morphism} of an adjunction of 2-monads $F : (\Cat/I,S) \longrightarrow (\Cat/J,T)$ arising from an operad morphism $F : S \to T$,  $F^c$, $F_!$ and $T$ are opfamilial by \cite{Weber-PolynomialFunctors} Theorem 4.4.5, and $T$ preserves all codescent objects as explained in Example \ref{exams:easy-actions-operadic-examples}. Similarly for adjunctions of 2-monads coming from morphisms of non-symmetric or braided operads as in Remark \ref{rem:nonsym-braided-internal-algebras}.
\end{exams}
\begin{exam}\label{exam:bar-of-Sm-as-cics}
Applying Examples \ref{exams:opfam->cic-polycases} to the case of the identity on $\tnb{S}$ exhibits $U^{\tnb{S}}\ca R_{\tnb{S}}1$ of Examples \ref{exam:bar-of-Sm-not-a-catead} and \ref{exam:dblcat-for-cmon-classifier} as a crossed double category, and $\sigma_{\tnb{S}}\ca R_{\tnb{S}}1$ as a crossed double functor. A chosen opcartesian square of $U^{\tnb{S}}\ca R_{\tnb{S}}1$ is by definition a chosen opcartesian arrow of $\tnb{S}(\mathbb{P})$ with respect to the split opfibration $\tnb{S}(t_{\mathbb{P}})$. For any category $A$, the identity arrows of $A$ are the chosen opcartesian arrows which exhibit the functor $t_A : A \to 1$ as a split opfibration. Recall that an arrow of $\tnb{S}(A)$ is of the form
\[ (\rho,(\alpha_k)_{1{\leq}k{\leq}n}) : (a_k)_{1{\leq}k{\leq}n} \longrightarrow (b_k)_{1{\leq}k{\leq}n} \]
where $\rho \in \Sigma_n$ and $\alpha_k : a_k \to b_{\rho k}$ is in $A$ for $1 \leq k \leq n$. By the way that $\tnb{S}$ preserves split opfibrations as explained in Lemma 6.3 of \cite{Weber-Fam2fun}, $(\rho,(\alpha_k)_k)$ is a chosen opcartesian arrow of $\tnb{S}(t_A)$ iff the $\alpha_k$ are all identities. Applying this in the case $A = \mathbb{P}$ and using the explicit description of the squares of $U^{\tnb{S}}\ca R_{\tnb{S}}1$ found in Example \ref{exam:dblcat-for-cmon-classifier}, a square
\[ \xygraph{{\underline{m}}="p0" [r] {\underline{n}}="p1" [d] {\underline{n}}="p2" [l] {\underline{m}}="p3" "p0":"p1"^-{f}:"p2"^-{\rho}:@{<-}"p3"^-{g}:@{<-}"p0"^-{\rho'}:@{}"p2"|-{=}} \]
in $U^{\tnb{S}}\ca R_{\tnb{S}}1$ is chosen opcartesian iff $\rho'$ is order preserving on the fibres of $g$. Since in this case a square is uniquely determined by its boundary, for any $f$ and $\rho$ as above there exist unique $g$ and $\rho'$ such that $\rho f = g\rho'$ and $\rho'$ is order preserving on the fibres of $g$.
\end{exam}
\begin{rem}\label{rem:bij-mon-factorisation}
We denote by $\S$ the category whose objects are natural numbers and morphisms $m \to n$ are functions $\underline{m} \to \underline{n}$, and regard $\Delta_+$ and $\P$ as subcategories of $\mathbb{S}$.  As explained in the proof of Proposition 3.1 of \cite{Batanin-EckmannHilton}, every $h : m \to n$ in $\mathbb{S}$ factors uniquely as $h = g\rho$, where $\rho \in \P$, $g \in \Delta_+$, and $\rho$ is order preserving on the fibres of $g$. We shall refer to this as the \emph{bijective-monotone factorisation} of $h$. Thus in the context of Example \ref{exam:bar-of-Sm-as-cics} the chosen opcartesian square $\kappa_{f,\rho}$ is obtained by taking the bijective-monotone factorisation of $\rho f$.
\end{rem}

\subsection{The 2-category of corners.}
\label{ssec:Corners}

The first step in our general computation of codescent objects of crossed internal categories is to produce an internal 2-category, called the \emph{internal 2-category of corners}, from a crossed internal category. For the purposes of our first definition, recall that an object $D$ of a 2-category $\ca K$ is \emph{discrete} when for all $X \in \ca K$, the hom category $\ca K(X,D)$ is discrete. In other words, any 2-cell between arrows into $D$ must be an identity.
\begin{defn}\label{def:internal-2-category}
Let $\ca K$ be a 2-category. Then a \emph{2-category in $\ca K$} is a category object in $\ca K$ whose object of objects is discrete.
\end{defn}
\begin{rem}\label{rem:2-cat-as-cic}
For any 2-category $\ca K$ and discrete object $D$ therein, every morphism $X \to D$ has a canonical structure of a split (op)fibration. This extends to morphisms and 2-cells of split opfibrations over $D$, and so any 2-category in $\ca K$ has the structure of a crossed internal category.
\end{rem}
When $\ca K = \Cat$ a 2-category in $\ca K$ is just a (small) 2-category, viewed as a double category whose vertical arrows are all identities. Recall that for a 2-category $\ca K$ with pullbacks, a category object in $\ca K$ can also be regarded as a monad in $\Span {\ca K}$. An advantage of this viewpoint is that one can use the basic notions of monad theory, such as distributive laws, when talking about category objects.

Let $\ca K$ be a 2-category with pullbacks and comma objects. Since cotensoring with a fixed object $A \in \ca K$ is a pullback preserving functor $\Cat^{\op} \to \ca K$, its restriction to $\Delta^{\op}$ gives a category object whose underlying span is
\[ A \xleftarrow{A^{\delta_1}} A^{[1]} \xrightarrow{A^{\delta_0}} A. \]
When $\ca K = \Cat(\ca E)$ for some category $\ca E$ with pullbacks, writing $i_A:A_0 \to A$ for the inclusion of objects, the composite span ${i_A}_{\bullet} \comp A^{[1]} \comp i_A^{\bullet}$ underlies another monad in $\Span {\ca K}$ whose corresponding simplicial object is the just the category object $A$ in $\ca E$ viewed as a componentwise-discrete category object in $\ca K$ as in Example \ref{ex:internal-nerve}.
\begin{const}\label{const:span-distlaw-from-crintcat}
Let $X$ be a crossed internal category in a 2-category $\ca K$ with comma objects and pullbacks. The underlying endospans
\[ \xygraph{{\xybox{\xygraph{{X_0}="p0" [r] {X_1}="p1" [r] {X_0}="p2" "p0":@{<-}"p1"^-{d_1}:"p2"^-{d_0}}}}
[r(4)] {\xybox{\xygraph{{X_0}="p0" [r] *!(0,-.02){X_0^{[1]}}="p1" [r] {X_0}="p2" "p0":@{<-}"p1"^-{X^{\delta_1}}:"p2"^-{X^{\delta_0}}}}}} \]
underlie monads in $\Span {\ca K}$, and we shall now construct the underlying 2-cell datum
\[ \delta_X : X_0^{[1]} \comp X_1 \longrightarrow X_1 \comp X_0^{[1]} \]
of a distributive law between them. The composite spans $X_0^{[1]} \comp X_1$ and $X_1 \comp X_0^{[1]}$ are explicitly
\[ \xygraph{{\xybox{\xygraph{!{0;(.75,0):(0,1)::} {d_0 \downarrow 1_{X_0}}="top" (:[dl] {X_1} (:[dl] {X_0},:[dr] {X_0}="pm"),:[dr] {X_0^{[1]}} (:"pm",:[dr] {X_0})) "pm" [u] {\scriptstyle{\tn{pb}}}}}} [r(2.5)] {\tn{and}} [r(2.5)]
{\xybox{\xygraph{!{0;(.75,0):(0,1)::} {1_{X_0} \downarrow d_1}="top" (:[dl] {X_0^{[1]}} (:[dl] {X_0},:[dr] {X_0}="pm"),:[dr] {X_1} (:"pm",:[dr] {X_0}))
"pm" [u] {\scriptstyle{\tn{pb}}}}}}} \]
respectively. We write
\[ \xygraph{{\xybox{\xygraph{!{0;(1.5,0):(0,.6667)::} {d_0 \downarrow 1_{X_0}}="p0" [r] {X_0}="p1" [d] {X_0}="p2" [l] {X_1}="p3" "p0":"p1"^-{q_1}:"p2"^-{1_{X_0}}:@{<-}"p3"^-{d_0}:@{<-}"p0"^-{p_1} "p0" [d(.55)r(.35)] :@{=>}[r(.3)]^-{\lambda_1}}}}
[r(4)]
{\xybox{\xygraph{!{0;(1.5,0):(0,.6667)::} {1_{X_0} \downarrow d_1}="p0" [r] {X_1}="p1" [d] {X_0}="p2" [l] {X_0}="p3" "p0":"p1"^-{q_2}:"p2"^-{d_1}:@{<-}"p3"^-{1_{X_0}}:@{<-}"p0"^-{p_2} "p0" [d(.55)r(.35)] :@{=>}[r(.3)]^-{\lambda_2}}}}} \]
for the defining comma squares for $d_0 \downarrow 1_{X_0}$ and $1_{X_0} \downarrow d_1$. Since $d_0$ is a split opfibration there is an arrow $r:d_0 \downarrow 1_{X_0} \to X_1$ and $\lambda_3:p_1 \to r$ unique such that $\lambda_3$ is chosen $d_0$-opcartesian and $d_0\lambda_3 = \lambda_1$. Using the universal property of $\lambda_2$, we define $\delta_X:d_0 \downarrow 1_{X_0} \to 1_{X_0} \downarrow d_1$ as the unique arrow in $\ca K$ such that
\[ \begin{array}{lcccr} {p_2\delta_x = p_1d_1} && {q_2\delta_X = r} && {\lambda_2\delta_x = d_1\lambda_3} \end{array} \]
from which it follows easily that the arrow $\delta_X$ in $\ca K$ underlies a morphism of spans as required.
\end{const}
In the case where $\ca K = \Cat$ we unpack an explicit description of $\delta_X$ in double categorical terms. An object of $d_0 \downarrow 1_{X_0}$ consists of $(f,g)$ in $X$ as on the left in
\[ \xygraph{{\xybox{\xygraph{{a}="p0" [r] {b}="p1" [d] {c}="p2" "p0":"p1"^-{f}:"p2"^-{g}}}}
[r(3)d(.1)]
{\xybox{\xygraph{{a}="p0" [r] {b}="p1" [d] {b'}="p2" [l] {a'}="p3" "p0":"p1"^-{f}:"p2"^-{v}:@{<-}"p3"^-{f'}:@{<-}"p0"^-{u} "p0" [d(.5)r(.5)] {w}}}}
[r(2)]
{\xybox{\xygraph{{c}="p0" [d] {c'}="p1" "p0":"p1"^-{x}}}}} \]
\[  \]
and a morphism $(f,g) \to (f',g')$ in $d_0 \downarrow 1_{X_0}$ consists of $(u,v,w,x)$ as on the right in the previous display such that $xg = g'v$. On the other hand an object of the category $1_{X_0} \downarrow d_1$ consists of $(f,g)$ in $X$ as in
\[ \xygraph{{\xybox{\xygraph{{a}="p0" [d] {b}="p1" [r] {c}="p2" "p0":"p1"_-{f}:"p2"_-{g}}}}
[r(3)]
{\xybox{\xygraph{{a}="p0" [d] {a'}="p1" "p0":"p1"^-{u}}}}
[r(2)]
{\xybox{\xygraph{{b}="p0" [r] {c}="p1" [d] {c'}="p2" [l] {b'}="p3" "p0":"p1"^-{g}:"p2"^-{w}:@{<-}"p3"^-{g'}:@{<-}"p0"^-{v} "p0" [d(.5)r(.5)] {x}}}}} \]
\[  \]
and a morphism $(f,g) \to (f',g')$ consists of  $(u,v,w,x)$ as on the right in the previous display such that $vf = f'u$. Recall that given an object $(f,g)$ of $d_0 \downarrow 1_{X_0}$, the crossed internal structure of $X$ gives us a chosen opcartesian square as on the left in
\[ \xygraph{{\xybox{\xygraph{{a}="p0" [r] {b}="p1" [d] {c}="p2" [l] {d}="p3" "p0":"p1"^-{f}:"p2"^-{g}:@{<-}"p3"^-{\rho_{f,g}}:@{<-}"p0"^-{\lambda_{f,g}} "p0" [d(.5)r(.5)] {\kappa_{f,g}}}}}
[r(4.5)]
{\xybox{\xygraph{{\xybox{\xygraph{{a}="p0" [r] {b}="p1" [d] {c}="p2" [d] {c'}="p3" [l] {d'}="p4" [u] {d}="p5" "p0":"p1"^-{f}:"p2"^-{g}:"p3"^-{x}:@{<-}"p4"^-{\rho_{f',g'}}:@{<-}"p5"^-{\delta}:@{<-}"p0"^-{\lambda_{f,g}} "p5":"p2"_-{\rho_{f,g}} "p0":@{}"p2"|-*{\kappa_{f,g}} "p5":@{}"p3"|-*{\varepsilon}}}}
[r(1.5)] {=} [r(1.25)]
{\xybox{\xygraph{{a}="p0" [r] {b}="p1" [d] {b'}="p2" [d] {c'}="p3" [l] {d'}="p4" [u] {a'}="p5" "p0":"p1"^-{f}:"p2"^-{v}:"p3"^-{g'}:@{<-}"p4"^-{\rho_{f',g'}}:@{<-}"p5"^-{\lambda_{f',g'}}:@{<-}"p0"^-{u} "p5":"p2"^-{f'} "p0":@{}"p2"|-*{w} "p5":@{}"p3"|-*{\kappa_{f',g'}}}}}}}}} \]
The effect on objects of $\delta_X$ is $(f,g) \mapsto (\lambda_{f,g},\rho_{f,g})$, and $\delta_X$ sends a morphism $(u,v,w,x) : (f,g) \to (f',g')$ of $d_0 \downarrow 1_{X_0}$ to $(u,\delta,x,\varepsilon)$ where $(\delta,\varepsilon)$ are unique satisfying the equation on the right in the previous display.
\begin{prop}\label{prop:span-dist-law-from-crintcat}
Let $\ca K$ be a 2-category with comma objects and pullbacks and $X$ be a crossed internal category in $\ca K$. Then the morphism of spans $\delta_X$ of Construction \ref{const:span-distlaw-from-crintcat} satisfies the axioms of a distributive law in $\Span {\ca K}$.
\end{prop}
\begin{proof}
Since everything is defined in terms of limits in $\ca K$, the result for general $\ca K$ follows from the case $\ca K = \Cat$ by a representable argument. There are 4 axioms to check, two which involve the monad units, and two which involve the monad multiplications. One of the unit axioms is encoded by the commutativity of the triangle on the left
\[ \xygraph{{\xybox{\xygraph{!{0;(1.75,0):(0,.57143)::} {X_0^{[1]}}="p0" [r] {X_0^{[1]} \comp X_1}="p1" [d] {X_1 \comp X_0^{[1]}}="p2" "p0":"p1"^-{X_0^{[1]} \comp s_0}:"p2"^-{\delta_X}:@{<-}"p0"^-{s_0 \comp X_0^{[1]}}}}}
[r(3)d(.1)]
{\xybox{\xygraph{{b}="p0" [r] {b}="p1" [d] {c}="p2" [l] {c}="p3" "p0":"p1"^-{1_b}:"p2"^-{g}:@{<-}"p3"^-{1_c}:@{<-}"p0"^-{g} "p0" [d(.5)r(.5)] {\id_g}}}}
[r(3.5)]
{\xybox{\xygraph{{\xybox{\xygraph{{b}="p0" [r] {b}="p1" [d] {c}="p2" [d] {c'}="p3" [l] {c'}="p4" [u] {c}="p5" "p0":"p1"^-{1_b}:"p2"^-{g}:"p3"^-{x}:@{<-}"p4"^-{1_{c'}}:@{<-}"p5"^-{x}:@{<-}"p0"^-{g} "p5":"p2"_-{1_c} "p0":@{}"p2"|-*{\id_g} "p5":@{}"p3"|-*{\id_x}}}}
[r(1.1)] {=} [r(1.1)]
{\xybox{\xygraph{{b}="p0" [r] {b}="p1" [d] {b'}="p2" [d] {c'}="p3" [l] {c'}="p4" [u] {b'}="p5" "p0":"p1"^-{1_b}:"p2"^-{v}:"p3"^-{g'}:@{<-}"p4"^-{1_{c'}}:@{<-}"p5"^-{g'}:@{<-}"p0"^-{v} "p5":"p2"^-{1_{b'}} "p0":@{}"p2"|-*{\id_{v}} "p5":@{}"p3"|-*{\id_{g'}}}}}}}}} \]
For any vertical arrow $g$, $\kappa_{1_b,g}$ is the identity square indicated in the middle of the previous display, and from this follows the commutativity of the left triangle on objects. Given a morphism $(v,x) : g \to g'$ in $X_0^{[1]}$, the equation on the right in the previous display ensures that $\delta_X(v,v,\id_v,x) = (v,x,x,\id_x)$, and so the left triangle commutes on arrows. For all horizontal arrows $f$, $\kappa_{f,1_b} = \id_f$. Using this fact and the universal property of opcartesian cells, one similarly verifies the other unit law.

One of the axioms involving the monad multiplications is encoded by the commutativity of
\begin{equation}\label{eq:dist-multn-axiom}
\xygraph{!{0;(2.5,0):(0,.4)::} {X_0^{[2]} \comp X_1}="p0" [r] {X_0^{[1]} \comp X_1 \comp X_0^{[1]}}="p1" [r] {X_1 \comp X_0^{[2]}}="p2" [l(.5)d] {X_1 \comp X_0^{[1]}}="p3" [l] {X_0^{[1]} \comp X_1}="p4" "p0":"p1"^-{X_0^{[1]} \comp \delta_X}:"p2"^-{\delta_X \comp X_0^{[1]}}:"p3"^-{X_1 \comp X_0^{s_1}}:@{<-}"p4"^-{\delta_X}:@{<-}"p0"^-{X_0^{s_1} \comp X_1}}
\end{equation}
An object of $X_0^{[2]} \comp X_1$ is a triple $(f,g,h)$ as on the left in
\[ \xygraph{{\xybox{\xygraph{{a}="p0" [r] {b}="p1" [d] {c}="p2" [d] {d}="p3" "p0":"p1"^-{f}:"p2"^-{g}:"p3"^-{h}}}}
[r(3)]
{\xybox{\xygraph{{a}="p0" [r] {b}="p1" [d] {c}="p2" [d] {d}="p3" [l] {e_2}="p4" [u] {e_1}="p5" "p0":"p1"^-{f}:"p2"^-{g}:"p3"^-{h}:@{<.}"p4"^-{}:@{<-}"p5"^-{}:@{<-}"p0"^-{} "p5":"p2"^{} "p0":@{}"p2"|-*{\kappa_{f,g}} "p5":@{}"p3"|-*{\kappa_{\rho_{f,g},h}} "p0":@{.>}@/_{1pc}/"p4"}}}} \]
and the effect of $(X_1 \comp X_0^{s_1})(\delta_X \comp X_0^{[1]})(X_0^{[1]} \comp \delta_X)$ on $(f,g,h)$ is obtained by vertically composing the two chosen opcartesian squares depicted on the right in the previous display, and taking the dotted pair of arrows as indicated. This is the same as first vertically composing $h$ and $g$, and then taking the lower boundary of the corresponding chosen opcartesian square, since chosen opcartesian squares vertically compose, and so (\ref{eq:dist-multn-axiom}) commutes at the level of objects.

A morphism $(f,g,h) \to (f',g',h')$ in $X_0^{[2]} \comp X_1$ consists of $(u,v,w,x,y)$ in $X$ as in
\[ \xygraph{{a}="p0" [r] {b}="p1" [d] {b'}="p2" [l] {a'}="p3" "p0":"p1"^-{f}:"p2"^-{v}:@{<-}"p3"^-{f'}:@{<-}"p0"^-{u}:@{}|-*{w}"p2"
"p1" [r] {c}="p4" [d] {c'}="p5" "p4":"p5"^-{x}
"p4" [r] {d}="p6" [d] {d'}="p7" "p6":"p7"^-{y}} \]
such that $xg = g'v$ and $yh = h'x$. Inducing $(\delta_1,\varepsilon_1)$ as on the left
\[ \xygraph{{\xybox{\xygraph{{\xybox{\xygraph{{a}="p0" [r] {b}="p1" [d] {c}="p2" [d] {c'}="p3" [l] {e_3}="p4" [u] {e_1}="p5" "p0":"p1"^-{f}:"p2"^-{g}:"p3"^-{x}:@{<-}"p4"^-{\rho_{f',g'}}:@{<-}"p5"^-{\delta_1}:@{<-}"p0"^-{} "p5":"p2"^{} "p0":@{}"p2"|-*{\kappa_{f,g}} "p5":@{}"p3"|-*{\varepsilon_1}}}}
[r(1.2)] {=} [r(1.2)]
{\xybox{\xygraph{{a}="p0" [r] {b}="p1" [d] {b'}="p2" [d] {c'}="p3" [l] {e_3}="p4" [u] {a'}="p5" "p0":"p1"^-{f}:"p2"^-{v}:"p3"^-{g'}:@{<-}"p4"^-{\rho_{f',g'}}:@{<-}"p5"^-{}:@{<-}"p0"^-{u} "p5":"p2"^-{f'} "p0":@{}"p2"|-*{w} "p5":@{}"p3"|-*{\kappa_{f',g'}}}}}}}}
[r(5)]
{\xybox{\xygraph{{\xybox{\xygraph{{e_1}="p0" [r] {c}="p1" [d] {d}="p2" [d] {d'}="p3" [l] {e_4}="p4" [u] {e_2}="p5" "p0":"p1"^-{\rho_{f,g}}:"p2"^-{h}:"p3"^-{y}:@{<-}"p4"^-{}:@{<-}"p5"^-{\delta_2}:@{<-}"p0"^-{} "p5":"p2"^-{}
"p0":@{}"p2"|-*{\kappa_{\rho_{f,g},h}} "p5":@{}"p3"|-*{\varepsilon_2}}}}
[r(1.2)] {=} [r(1.2)]
{\xybox{\xygraph{{e_1}="p0" [r] {c}="p1" [d] {c'}="p2" [d] {d'}="p3" [l] {e_4}="p4" [u] {e_3}="p5" "p0":"p1"^-{\rho_{f,g}}:"p2"^-{x}:"p3"^-{h'}:@{<-}"p4"^-{}:@{<-}"p5"^-{}:@{<-}"p0"^-{\delta_1} "p5":"p2"^-{}
"p0":@{}"p2"|-*{\varepsilon_1} "p5":@{}"p3"|-{\kappa_{\rho_{f',g'},h'}}}}}}}}} \]
and then $(\delta_2,\varepsilon_2)$ as on the right, the effect of $(X_1 \comp X_0^{s_1})(\delta_X \comp X_0^{[1]})(X_0^{[1]} \comp \delta_X)$ on $(u,v,w,x,y)$ is $(u,v,x,\varepsilon_2)$. This coincides with the effect of $\delta_X(X_0^{s_1} \comp X_1)$ on $(u,v,w,x,y)$ by
\[ \xygraph{{\xybox{\xygraph{{a}="p0" [r] {b}="p1" [d] {d}="p2" [d] {d'}="p3" [l] {e_4}="p4" [u] {e_5}="p5" "p0":"p1"^-{}:"p2"^-{}:"p3"^-{}:@{<-}"p4"^-{}:@{<-}"p5"^-{}:@{<-}"p0"^-{} "p5":"p2"^-{} "p0":@{}"p2"|-*{\kappa_{f,hg}} "p5":@{}"p3"|-*{\varepsilon_2}}}}
[r] {=} [r]
{\xybox{\xygraph{{a}="p0" [r] {b}="p1" [d] {c}="p2" [d] {d}="p3" [d] {d'}="p4" [l] {e_4}="p5" [u] {e_5}="p6" [u] {e_1}="p7" "p0":"p1"^-{}:"p2"^-{}:"p3"^-{}:"p4"^-{}:@{<-}"p5"^-{}:@{<-}"p6"^-{}:@{<-}"p7"^-{}:@{<-}"p0"^-{} "p7":"p2"^-{} "p6":"p3"^-{} "p0":@{}"p2"|-*{\kappa_{f,g}} "p7":@{}"p3"|-*{\kappa_{\rho_{f,g},h}} "p6":@{}"p4"|-*{\varepsilon_2}}}}
[r] {=} [r]
{\xybox{\xygraph{{a}="p0" [r] {b}="p1" [d] {c}="p2" [d] {c'}="p3" [d] {d'}="p4" [l] {e_4}="p5" [u] {e_3}="p6" [u] {e_1}="p7" "p0":"p1"^-{}:"p2"^-{}:"p3"^-{}:"p4"^-{}:@{<-}"p5"^-{}:@{<-}"p6"^-{}:@{<-}"p7"^-{}:@{<-}"p0"^-{} "p7":"p2"^-{} "p6":"p3"^-{} "p0":@{}"p2"|-*{\kappa_{f,g}} "p7":@{}"p3"|-*{\varepsilon_1} "p6":@{}"p4"|-{\kappa_{\rho_{f',g'},h'}}}}}
[r] {=} [r]
{\xybox{\xygraph{{a}="p0" [r] {b}="p1" [d] {b'}="p2" [d] {c'}="p3" [d] {d'}="p4" [l] {e_4}="p5" [u] {e_3}="p6" [u] {a'}="p7" "p0":"p1"^-{}:"p2"^-{}:"p3"^-{}:"p4"^-{}:@{<-}"p5"^-{}:@{<-}"p6"^-{}:@{<-}"p7"^-{}:@{<-}"p0"^-{} "p7":"p2"^-{} "p6":"p3"^-{} "p0":@{}"p2"|-*{w} "p7":@{}"p3"|-*{\kappa_{f',g'}} "p6":@{}"p4"|-{\kappa_{\rho_{f',g'},h'}}}}}
[r] {=} [r]
{\xybox{\xygraph{{a}="p0" [r] {b}="p1" [d] {b'}="p2" [d] {d'}="p3" [l] {e_4}="p4" [u] {a'}="p5" "p0":"p1"^-{}:"p2"^-{}:"p3"^-{}:@{<-}"p4"^-{}:@{<-}"p5"^-{}:@{<-}"p0"^-{} "p5":"p2"^-{} "p0":@{}"p2"|-*{w} "p5":@{}"p3"|*-{\kappa_{f',h'g'}}}}}} \]
and the uniqueness part of the universal property of $\kappa_{f,hg}$. Using the fact that chosen opcartesian squares horizontally compose, one similarly verifies the remaining distributive law axiom.
\end{proof}
Thanks to the above result, the composite span $X_1 \comp X_0^{[1]}$ has the structure of a monad. In the case where $\ca K = \Cat(\ca E)$ for some category $\ca E$ with pullbacks, we have the inclusion of objects $i_{X_0} : X_{00} \to X_0$ of $X_0$, and in $\Span {\ca K}$ the adjunction $(i_{X_0})^{\bullet} \ladj (i_{X_0})_{\bullet}$. Thus the composite span
\[ (i_{X_0})_{\bullet} \comp X_1 \comp X_0^{[1]}  \comp (i_{X_0})^{\bullet}  \]
is the underlying endomorphism of a monad in $\Span{\ca K}$ on $X_{00}$. Since $X_{00}$ is discrete, we have thus exhibited an internal 2-category from the crossed internal category $X$.
\begin{defn}\label{def:Cnr}
Let $X$ be a crossed internal category in a 2-category $\ca K$ of the form $\tn{Cat}(\ca E)$ for $\ca E$ a category with pullbacks. Then the internal 2-category just described is called the \emph{internal 2-category of corners} of $X$, and is denoted as $\tn{Cnr}(X)$.
\end{defn}
\begin{rem}\label{rem:Cnr-of-2-cat}
By Remark \ref{rem:2-cat-as-cic} an internal 2-category $X$ is a crossed internal category. In this case the monad $X_0^{[1]}$ in $\Span {\ca K}$ and $i_{X_0}$ are identities, and so $\tn{Cnr}(X) = X$.
\end{rem}
In the case where $\ca K = \Cat$ one can unpack the following explicit description of $\tn{Cnr}(X)$ in terms of the double category $X$. Its objects are the objects of $X$, and an arrow $x \to y$ is a pair $(f,g)$ where $f$ is a vertical arrow and $g$ is a horizontal arrow as on the left in
\[ \xygraph{{\xybox{\xygraph{{x}="l" [d] {z}="m" [r] {y}="r" "l":"m"^-{f}:"r"^-{g}}}}
[r(3)]
{\xybox{\xygraph{{x}="ttl" [d] {a}="tl" [r] {y}="tr" [d] {b}="br" [l] {c}="bl" [r(2)] {z}="brr"
"ttl":"tl"_{f}:"tr"^-{g}:"br"^-{h}:@{<-}"bl"^-{\rho_{g,h}}:@{<-}"tl"^-{\lambda_{g,h}}:@{}"br"|-*{\kappa_{g,h}}:"brr"_-{k}}}}
[r(3)]
{\xybox{\xygraph{{x}="ttl" [d] {z_1}="tl" [r] {y}="tr" [d] {y}="br" [l] {z_2}="bl" "ttl":"tl"_-{f}:"tr"^-{g}:"br"^-{1_y}:@{<-}"bl"^-{k}:@{<-}"tl"^-{\alpha} "tl":@{}"br"|-*{\beta}}}}} \]
and is called a \emph{corner} from $x$ to $y$. Such an $(f,g)$ is an identity in $\Cnr(X)$ when $f$ and $g$ are identities. The composite of $(f,g):x \to y$ and $(h,k):y \to z$ in $\Cnr(X)$ is defined to be $(\lambda_{g,h}f,k\rho_{g,h})$ as in the middle of the previous display. Given $(f,g)$ and $(h,k):x \to y$, a 2-cell $(f,g) \to (h,k)$ in $\Cnr(X)$ is a pair $(\alpha,\beta)$ where $\alpha$ is a vertical arrow and $\beta$ is a square as on the right in the previous display, such that $\alpha f=h$. Vertical composition of 2-cells in $\Cnr(X)$ is given in the evident manner by vertical composition in $X$.
\begin{exam}\label{exam:cnr-bar-of-Sm}
Continuing the discussion of Examples \ref{exam:dblcat-for-cmon-classifier} and \ref{exam:bar-of-Sm-as-cics}, the 2-category $\tn{Cnr}(U^{\tnb{S}}\ca R_{\tnb{S}}1)$ has the following explicit description. It has natural numbers as objects, and an arrow $m \to n$ is a pair $(\rho,f)$ where $\rho \in \Sigma_m$ and $f : \underline{m} \to \underline{n}$ is an order preserving function. The composite of
\[ \xygraph{!{0;(1.5,0):(0,1)::} {l}="p0" [r] {m}="p1" [r] {n}="p2" "p0":"p1"^-{(\rho_1,f_1)}:"p2"^-{(\rho_2,f_2)}} \]
is given by $(\rho_3\rho_1,f_2f_3)$, where $\rho_2f_1 = f_3\rho_3$ is the bijective-monotone factorisation (see Remark \ref{rem:bij-mon-factorisation}) of $\rho_2f_1$. Given $(\rho_1,f_1)$ and $(\rho_2,f_2) : m \to n$, a 2-cell $(\rho_1,f_1) \to (\rho_2,f_2)$ is unique if it exists, and exists iff $f_2\rho_2\rho_1^{-1} = f_1$, that is iff $f_2\rho_2 = f_1\rho_1$.
\end{exam}
As we saw in Remark \ref{rem:Cnr-of-2-cat} one has $\tn{Cnr}(X) = X$ for $X$ an internal 2-category. Another special case in which $\tn{Cnr}(X)$ assumes a simple form is given by the following result, which covers the Example \ref{ex:Loday-Fieodorowicz} of crossed simplicial groups.
\begin{lem}\label{lem:discrete-Cnr}
Let $X$ be a crossed internal category in a 2-category $\ca K$ of the form $\Cat(\ca E)$ for $\ca E$ a category with pullbacks, and suppose that $d_0:X_1 \to X_0$ is a discrete opfibration. Then $\Cnr(X)$ is componentwise discrete.
\end{lem}
\begin{proof}
The general statement follows from the case $\ca K = \Cat$ by a representable argument. In that case observe that the $\alpha$ and $\beta$ of a 2-cell $(f,g) \to (h,k)$ must necessarily coincide with
\[ \xygraph{{z_1}="p0" [r] {y}="p1" [d] {y}="p2" [l] {z_1}="p3" "p0":"p1"^-{g}:"p2"^-{1_y}:@{<-}"p3"^-{k=g}:@{<-}"p0"^-{1_{z_1}}:@{}"p2"|-*{1_g}} \]
and thus be identities, whence $\Cnr(X)_1$ is discrete. Since $\tn{Cnr}(X)_0$ is discrete by definition and $\Cnr(X)$ is a category object, the result follows.
\end{proof}
\begin{rem}\label{rem:functoriality-Cnr}
The observation of Remark \ref{rem:2-cat-as-cic} extends to morphisms so that one has a full inclusion $2\tn{-Cat}(\ca K) \hookrightarrow \CrIntCat {\ca K}$ between the category of 2-categories in $\ca K$ to that of crossed internal categories and crossed internal functors in $\ca K$. In general we have a functor
\[ \Cnr_{\ca K} : \CrIntCat {\ca K} \longrightarrow 2\tn{-Cat}(\ca K) \]
which is a retraction of the inclusion. Since $\Cnr_{\ca K}$ is described in terms of limits, it is a limit preserving functor in general. In the case $\ca K = \Cat$, since at the set-wise level the limits used in $\Cnr_{\Cat}$'s explicit description are finite and connected, it follows that $\Cnr_{\Cat}$ also preserves filtered colimits and coproducts, since these commute in $\Set$ with finite connected limits.
\end{rem}

\subsection{Computation of codescent objects.}
\label{ssec:CodescComp}

In this section we prove results involving the computation of codescent objects of crossed internal categories within 2-categories $\ca K$ of the form $\tn{Cat}(\ca E)$. The first, Theorem \ref{thm:codesc-2-cat}, explains how to compute codescent objects of internal 2-categories. Then in Theorem \ref{thm:codesc-cr-int-cat}, we see that the codescent object of a crossed internal category $X$ is that of $\tn{Cnr}(X)$, its associated internal 2-category of corners. Thus we can compute codescent for an arbitrary crossed internal category within a 2-category $\ca K$ of the form $\tn{Cat}(\ca E)$ where $\ca E$ is a category with pullbacks and pullback-stable reflexive coequalisers. This is recorded below in Corollary \ref{cor:codesc-gen-cicat}. On the other hand we can compute codescent for crossed internal categories satisfying the hypotheses of Lemma \ref{lem:discrete-Cnr} in complete generality. This is recorded below in Corollary \ref{cor:codesc-gen-E}.

In an earlier version of this article Lemma \ref{lem:pi0-pres-enough-pbs}, Theorem \ref{thm:codesc-2-cat} and Corollary \ref{cor:codesc-gen-cicat} were stated less generally, for a locally cartesian closed category $\ca E$ with coequalisers. I am indebted to John Bourke for pointing out that the proofs of these results go through in the greater generality now presented here.

Given a category $\ca E$ with pullbacks and reflexive coequalisers, the inclusion $d$ of discrete objects of $\tn{Cat}(\ca E)$ has a left adjoint $\pi_0$, which on objects is given by taking the coequaliser of the source and target maps. Regarding $d$ as an inclusion, we denote the unit of this adjunction as $\eta_X :X \to \pi_0X$. Since $d$ preserves coequalisers, we regard the defining coequaliser of $\pi_0X$ as living in $\tn{Cat}(\ca E)$.

Recall from Example \ref{ex:internal-nerve} that for $f:X \to Y$ in $\tn{Cat}(\ca E)$, a 2-cell $\phi:f \to g$ has underlying data an arrow $\phi_0:X_0 \to Y_1$ in $\ca E$ such that $d_1\phi_0 = f_0$ and $d_0\phi_0 = g_0$. Given such $\phi_0$ one induces $\phi_1$, $\phi_2:X_1 \to Y_2$ in $\ca E$ unique such that
\[ \begin{array}{lcccccr} {d_2\phi_1 = \phi_0d_1} && {d_0\phi_1 = g_1} && {d_2\phi_2 = f_1} && {d_0\phi_2 = \phi_0d_0} \end{array} \]
and then the ``naturality condition'' for $\phi$ demands that $d_1\phi_1 = d_1\phi_2$. Recall also that a morphism $h:A \to B$ in a 2-category $\ca K$ is \emph{cofully faithful} when for all $Z \in \ca K$, the functor $\ca K(h,Z) : \ca K(B,Z) \to \ca K(A,Z)$ given by precomposition with $h$, is fully faithful.
\begin{lem}\label{lem:coff-into-pi0}
Let $\ca E$ be a category with pullbacks and reflexive coequalisers, and let $\ca K = \tn{Cat}(\ca E)$. Then for all $X \in \ca K$, $\eta_X : X \to \pi_0X$ is cofully faithful.
\end{lem}
\begin{proof}
Let us suppose that $Y$, $f$ and $g:\pi_0X \to Y$, and $\phi:f\eta_X \to g\eta_X$ in $\ca K$ are given. We must exhibit a unique $\psi:f \to g$ such that $\psi \eta_X = \phi$. To this end we consider the following diagram in $\ca E$
\[ \xygraph{!{0;(2,0):(0,.5)::}
{X_2}="p0" [r] {X_1}="p1" [r] {X_0}="p2"
"p2":"p1"|-{} "p1":@<1.5ex>"p2"^-{} "p1":@<-1.5ex>"p2"_-{} "p0":@<1.5ex>"p1"^-{} "p0":"p1"|-{} "p0":@<-1.5ex>"p1"_-{}
"p0" [d] {\pi_0X}="q0" [r] {\pi_0X}="q1" [r] {\pi_0X}="q2"
"q2":"q1"|-{} "q1":@<1.5ex>"q2"^-{} "q1":@<-1.5ex>"q2"_-{} "q0":@<1.5ex>"q1"^-{} "q0":"q1"|-{} "q0":@<-1.5ex>"q1"_-{}
"q0" [d] {Y_2}="r0" [r] {Y_1}="r1" [r] {Y_0}="r2"
"r2":"r1"|-{} "r1":@<1.5ex>"r2"^-{} "r1":@<-1.5ex>"r2"_-{} "r0":@<1.5ex>"r1"^-{} "r0":"r1"|-{} "r0":@<-1.5ex>"r1"_-{}
"p0":"q0"_-{\eta_{X,2}}(:@<1ex>"r0"^(.35){g_2},:@<-1ex>"r0"_-{f_2}) "p1":"q1"^-{\eta_{X,1}}(:@<1ex>"r1"^(.4){g_1},:@<-1ex>"r1"_-{f_1})
"p2":"q2"^-{\eta_{X,0}}(:@<1ex>"r2"^-{g_0},:@<-1ex>"r2"_-{f_0})
"p2":@{.>}"r1"|(.25){\phi_0}
"p1"(:@{}@<1ex>"r0"|(.1){}="d1"|(.8){}="c1",:@{}@<-1ex>"r0"|(.2){}="d2"|(.85){}="c2")
"d1":@{.>}"c1"|(.75){\phi_1} "d2":@{.>}"c2"|(.2){\phi_2}} \]
in which the arrows of the middle row are all identities. By definition $\phi_1$ is the unique morphism such that $d_2\phi_1 = \phi_0d_1$ and $d_0\phi_1 = s_0g_0\eta_{X,1}$. Note that
\[ d_0s_1\phi_0d_1 = s_0d_0\phi_0d_1 = s_0g_0\eta_{X,0}d_1 = s_0g_0\eta_{X,1} \]
and $d_2s_1\phi_0d_1 = \phi_0d_1$, and so $\phi_1 = s_1\phi_0d_1$. Similarly $\phi_2 = s_0\phi_0d_0$. Thus
\[ \phi_0d_1 = d_1s_1\phi_0d_1 = d_1\phi_1 = d_1\phi_2 = d_1s_0\phi_0d_0 = \phi_0d_0 \]
and so there is a unique $\psi_0:\pi_0X \to Y_1$ such that $\psi_0\eta_{X,0} = \phi_0$. Since $d_1\psi_0\eta_{X,0} = f_0\eta_{X,0}$, $d_0\psi_0\eta_{X,0} = g_0\eta_{X,0}$ and $\eta_{X,0}$ is epi, we have $d_1\psi_0 = f_0$ and $d_0\psi_0 = g_0$, and the naturality condition for $\psi$ is automatically satisfied since $\pi_0X$ is a discrete category object. The required uniqueness follows from the uniqueness aspect of $\eta_{X,0}$'s universal property as a coequaliser.
\end{proof}
The next lemma ensures, that for nice enough $\ca E$, $\pi_0$ preserves enough pullbacks so that when it is applied componentwise to an internal 2-category, the result will still be a category object. Recall that a category $\ca E$ as in Lemma \ref{lem:coff-into-pi0} has \emph{pullback stable reflexive coequalisers} when for all $f : A \to B \in \ca E$, the functor $\Delta_f : \ca E/B \to \ca E/A$ preserves reflexive coequalisers. For $D \in \ca E$ the functor $\Sigma_D : \ca E/D \to \ca E$, which sends an arrow into $D$ to its domain, creates pullbacks and colimits, and so any slice of $\ca E$ will have pullback stable reflexive coequalisers whenever $\ca E$ does. In particular for such $\ca E$ and $A \in \ca E$, the functor $(-) \times A : \ca E \to \ca E$ preserves reflexive coequalisers. 
\begin{lem}\label{lem:pi0-pres-enough-pbs}
Let $\ca E$ be a category with pullbacks and pullback-stable reflexive coequalisers, and let $\ca K = \tn{Cat}(\ca E)$. Then $\pi_0$ preserves pullbacks
\[ \xygraph{{P}="p0" [r] {A}="p1" [d] {D}="p2" [l] {B}="p3" "p0":"p1"^-{}:"p2"^-{}:@{<-}"p3"^-{}:@{<-}"p0"^-{}:@{}"p2"|-{\tn{pb}}} \]
in $\ca K$ such that $D$ is discrete.
\end{lem}
\begin{proof}
Under the given assumptions on $\ca E$, $\pi_0$ preserves products by a standard application of the $(3 \times 3)$-lemma (\cite{Johnstone-ToposTheory} Lemma 0.17), since the coequalisers in question are all reflexive coequalisers. We can apply this in the case $\ca K = \tn{Cat}(\ca E/D) \iso \tn{Cat}(\ca E)/D$ to obtain the result.
\end{proof}
In the context of this last result, given a 2-category $X$ internal to $\ca K$, we denote by $\pi_{0*}X$ the componentwise discrete simplicial object
\[ \Delta^{\op} \xrightarrow{X} \ca K \xrightarrow{\pi_0} \ca E \xrightarrow{d} \ca K.  \]
A mild reformulation of the condition recalled in Example \ref{ex:internal-nerve} for $X$ to be a category object is that 
\[ \xygraph{!{0;(1.5,0):(0,.6667)::} {X_{n+2}}="p0" [r] {X_{n+1}}="p1" [d] {X_0}="p2" [l] {X_1}="p3" "p0":"p1"^-{d_{n+2}}:"p2"^-{d_0^{\comp n}}:@{<-}"p3"^-{d_1}:@{<-}"p0"^-{d_0^{\comp n}}} \]
is a pullback for all $n \in \N$, and so since $X_0$ is discrete, $\pi_{0*}X$ is also a category object by Lemma \ref{lem:pi0-pres-enough-pbs}. Since $\pi_{0*}X$ is a componentwise-discrete category object, its codescent object exists, and is just $\pi_{0*}X$ regarded as an object of $\ca K$. In the case $\ca K = \Cat$, $X$ is a 2-category and $\pi_{0*}X$ is the category whose objects are those of $X$ and hom sets are the sets of connected components of the corresponding hom categories of $X$.
\begin{thm}\label{thm:codesc-2-cat}
Let $\ca E$ be a category with pullbacks and pullback-stable reflexive coequalisers, $\ca K = \tn{Cat}(\ca E)$ and let $X$ be an internal 2-category. Then the codescent object of $X$ exists and $\tn{CoDesc}(X) = \pi_{0*}X$.
\end{thm}
\begin{proof}
We continue to regard $d$ as an inclusion. In
\[ \xygraph{!{0;(2,0):(0,.5)::}
{X_2}="p0" [r] {X_1}="p1" [r] {X_0}="p2"
"p2":"p1"|-{} "p1":@<1.5ex>"p2"^-{} "p1":@<-1.5ex>"p2"_-{} "p0":@<1.5ex>"p1"^-{} "p0":"p1"|-{} "p0":@<-1.5ex>"p1"_-{}
"p0" [d] {\pi_{0*}X_2}="q0" [r] {\pi_{0*}X_1}="q1" [r] {\pi_{0*}X_0}="q2"
"q2":"q1"|-{} "q1":@<1.5ex>"q2"^-{} "q1":@<-1.5ex>"q2"_-{} "q0":@<1.5ex>"q1"^-{} "q0":"q1"|-{} "q0":@<-1.5ex>"q1"_-{}
"p0":"q0"_-{\eta_{X_2}} "p1":"q1"^-{\eta_{X_1}} "p2":"q2"^-{\eta_{X_0}}
"q2" :[r] {C}="q3"^-{q_0} "p2" :@{.>}[r] {Y}="p3"^-{r_0}} \]
$\eta_{X_0} = 1_{X_0}$ since $X_0$ is discrete, and we have $q_1:q_0\pi_0(d_1) \to q_0\pi_0(d_0)$ such that $(q_0,q_1)$ exhibit $C$ as the codescent object of $\pi_{0*}X$. It suffices to show that $(q_0,q_1\eta_{X_1})$ exhibits $C$ as the codescent object of $X$. The calculations
\[ \begin{array}{l} {q_1\eta_{X_1}s_0 = q_1\pi_0(s_0)\eta_{X_0} = \id} \\
{(q_1\eta_{X_1}d_0)(q_1\eta_{X_1}d_2) = (q_1\pi_0(d_0)\eta_{X_2})(q_1\pi_0(d_2)\eta_{X_2}) = q_1\pi_0(d_1)\eta_{X_2} = q_1\eta_{X_1}d_1} \end{array} \]
ensure that $(q_0,q_1\eta_{X_1})$ is a cocone. Since $\ca K$ admits cotensors with $[1]$, it suffices to verify the 1-dimensional universal property of $(q_0,q_1\eta_{X_1})$. So we suppose $Y$, $r_0$ and $r_1:d_1r_0 \to d_0r_0$ are given making $(r_0,r_1)$ the data of a cocone for $X$ with vertex $Y$. Since $\eta_{X_0}$ is an identity we regard $r_0$ as a morphism $r_0:\pi_{0*}X_0 \to Y$. Since $\eta_{X,1}$ is cofully faithful by Lemma \ref{lem:coff-into-pi0}, there is a unique $r'_1:r_0\pi_0(d_1) \to r_0\pi_0(d_0)$ such that $r'_1\eta_{X_1} = r_1$. To see that $(r_0,r'_1)$ is a cocone for $\pi_{0*}X$ with vertex $Y$, note that
\[ r'_1\pi_0(s_0) = r'_1\pi_0(s_0)\eta_{X_0} = r'_1\eta_{X_1}s_0 = r_1s_0 = \id  \]
and that
\[ \begin{array}{rcl} {(r'_1\pi_0(d_0))(r'_1\pi_0(d_2))\eta_{X_2}}
& = & {(r'_1\eta_{X_1}d_0)(r'_1\eta_{X_1}d_2) = (r_1d_0)(r_1d_2) = r_1d_1} \\
& = & {r'_1\eta_{X_1}d_1 = r'_1\pi_0(d_1)\eta_{X_2},} \end{array} \]
and so since $\eta_{X_2}$ is cofully faithful, $(r'_1\pi_0(d_0))(r'_1\pi_0(d_2)) = r'_1\pi_0(d_1)$. Since $(q_0,q_1)$ is a codescent cocone, there is a unique $s:C \to Y$ such that $sq_0 = r_0$ and $sq_1 = r'_1$. Since $\eta_{X_1}$ is cofully faithful this last equation is equivalent to $sq_1\eta_{X_1} = r_1$.
\end{proof}
In the case $\ca E = \Set$, this last result says that the codescent object of a 2-category $X$, regarded as a double category whose vertical arrows are all identities, is the category whose objects are those of $X$ and hom sets are the sets of connected components of the corresponding hom-categories of $X$.
\begin{thm}\label{thm:codesc-cr-int-cat}
Let $\ca E$ be a category with pullbacks, $\ca K = \tn{Cat}(\ca E)$, and let $X$ be a crossed internal category in $\ca K$. Then the codescent object of $X$ exists iff the codescent object of $\Cnr(X)$ does, in which case $\tn{CoDesc}(X) = \tn{CoDesc}(\tn{Cnr}(X))$.
\end{thm}
\begin{proof}
For $Y \in \ca K$, we shall exhibit a bijection between cocones for $\tn{Cnr}(X)$ with vertex $Y$, and cocones for $X$ with vertex $Y$, naturally in $Y$. Since $\ca K$ admits cotensors with $[1]$, the result follows from this. The construction of $\tn{Cnr}(X)$ from $X$ was done using limits in $\ca K$, so one has $\ca K(Z,\tn{Cnr}(X)) \iso \tn{Cnr}(\ca K(Z,X))$ naturally in $Z$. Thus by a representable argument it suffices to exhibit the required bijections in the case $\ca K = \Cat$.

Let $q_0:X_{00} \to Y$ and $q_1:q_0d_1 \to q_0d_0$ be the data of a cocone for $\tn{Cnr}(X)$ with vertex $Y$. We define a cocone $(q'_0,q'_1)$ for $X$ with vertex $Y$ as follows. For an object $a \in X_0$, we define $q'_0a = q_0a$. For an arrow $f:a \to b$ in $X_0$, which recall is a vertical arrow of the double category $X$, we define $q'_0f = (q_1)_{(f,1_b)}$. From
\[ q'_01_a = (q_1)_{(1_a,1_a)} = (q_1s_0)_a = 1_{q_0a} \]
where $a \in X_0$, and
\[ \begin{array}{rclcl} {q'_0(g)q'_0(f)} & = & (q_1)_{(g,1_c)}(q_1)_{(f,1_b)} & = & ((q_1d_0)(q_1d_2))_{(g,1_c),(f,1_b)} \\
& = & (q_1d_1)_{(g,1_c),(f,1_b)} & = & (q_1)_{(gf,1_c)} = q'_0(gf)  \end{array} \]
where $f:a \to b$ and $g:b \to c$ are in $X_0$, $q'_0$ is a functor. For a horizontal arrow $h:a \to b$ of the double category $X$, which recall is an object of $X_1$ such that $d_1h = a$ and $d_0h = b$, we define $(q'_1)_h = (q_1)_{(1_a,h)}$.

We now verify the naturality of $q'_1$ with respect to an arbitrary morphism $(u,v,w) : h \to k$ which recall is a square in the double category $X$ as depicted on the left in
\[ \xygraph{!{0;(3,0):(0,1)::} 
{\xybox{\xygraph{{a}="p0" [r] {b}="p1" [d] {d}="p2" [l] {c}="p3" "p0":"p1"^-{h}:"p2"^-{v}:@{<-}"p3"^-{k}:@{<-}"p0"^-{u}:@{}"p2"|-*{w}}}}
[r]
{\xybox{\xygraph{!{0;(1.5,0):(0,.6667)::} {q_0a}="p0" [r] {q_0b}="p1" [d] {q_0d}="p2" [l] {q_0c}="p3" "p0":"p1"^-{(q_1)_{(1_a,h)}}:"p2"^-{(q_1)_{(v,1_d)}}:@{<-}"p3"^-{(q_1)_{(1_c,k)}}:@{<-}"p0"^-{(q_1)_{(u,1_c)}}}}}
[r]
{\xybox{\xygraph{{a}="p0" [r] {b}="p1" [d] {d}="p2" [d] {d}="p3" [l] {c}="p4" [u] {e}="p5" "p0":"p1"^-{h}:"p2"^-{v}:"p3"^-{1_d}:@{<-}"p4"^-{k}:@{<-}"p5"^-{x}:@{<-}"p0"^-{\lambda} "p5":"p2"|-{\rho} "p0":@{}"p2"|-*{\kappa}:@{}"p4"|-*{y}}}}} \]
and the corresponding naturality square in $Y$, which we must exhibit as commutative, is given in the middle of this display. We factor the original square $w$ as on the right, where $\kappa$ is chosen opcartesian. Now
\[ (q_1)_{(1_c,k)}(q_1)_{(u,1_c)} = ((q_1d_0)(q_1d_2))_{((1_c,k),(u,1_c))} = (q_1d_1)_{((1_c,k),(u,1_c))} = (q_1)_{(u,k)} \]
and, from the explicit description of composition in $\tn{Cnr}(X)$, one also has
\[ (q_1)_{(v,1_d)}(q_1)_{(1_a,h)} = ((q_1d_0)(q_1d_2))_{((v,1_d),(1_a,h))} = (q_1d_1)_{((v,1_d),(1_a,h))} = (q_1)_{(\lambda,\rho)}. \]
Note that the data $(x,y)$ above is that of a 2-cell $(\lambda,\rho) \to (u,k)$ in the 2-category $\tn{Cnr}(X)$, which is the same thing as an arrow in the category $\tn{Cnr}(X)_1$. The naturality of $q_1$ with respect to this arrow says exactly that $(q_1)_{(u,k)} = (q_1)_{(\lambda,\rho)}$, whence the commutativity of the given naturality square for $q'_1$. The calculations
\[ (q'_1s_0)_a = (q'_1)_{1_a} = (q_1)_{(1_a,1_a)} = (q_1s_0)_a = 1_{q_0a}  \]
and
\[ \begin{array}{rcl} {(q'_1d_0)(q'_1d_2)_{(k,h)}} & = & {((q_1d_0)(q_1d_2))_{((1_b,k),(1_a,h))} = (q_1d_1)_{((1_b,k),(1_a,h))}} \\
& = & {(q_1)_{kh} = (q'_1d_1)_{(k,h)}} \end{array} \]
exhibit $(q'_0,q'_1)$ as a cocone for $X$ with vertex $Y$. It is straight forward to verify that the assignation $(q_0,q_1) \mapsto (q'_0,q'_1)$ just spelled out is natural in $Y$.

Conversely given $r_0:X_0 \to Y$ and $r_1:r_0d_1 \to r_0d_0$ forming a cocone for $X$ with vertex $Y$ we define a cocone $(r'_0,r'_1)$ for $\tn{Cnr}(X)$ as follows. First we take $r'_0 = r_0i_{X_0}$ so that $r'_0$ is just the effect of $r_0$ on objects. Given $(f,g) : a \to b$ in the 2-category $\tn{Cnr}(X)$ we must define the component $(r'_1)_{(f,g)}:r_0a \to r_0b$ in $Y$, and this we take to be the composite $(r_1)_gr_0(f)$. For a 2-cell $(u,v) : (f,g) \to (h,k)$ of $\tn{Cnr}(X)$, which is a square in $X$ as on the left in
\[ \xygraph{{\xybox{\xygraph{{x}="p0" [r] {b}="p1" [d] {b}="p2" [l] {y}="p3" "p0":"p1"^-{g}:"p2"^-{1_b}:@{<-}"p3"^-{k}:@{<-}"p0"^-{u}:@{}"p2"|-*{v}}}}
[r(3.5)]
{\xybox{\xygraph{!{0;(1.5,0):(0,.6667)::} {r_0a}="p0" [r] {r_0x}="p1" [r] {r_0b}="p2" [d] {r_0b}="p3" [l] {r_0y}="p4" [l] {r_0a}="p5" "p0":"p1"^-{r_0f}:"p2"^-{(r_1)_g}:"p3"^-{1_{r_0b}}:@{<-}"p4"^-{(r_1)_k}:@{<-}"p5"^-{r_0h}:@{<-}"p0"^-{1_{r_0a}} "p1":"p4"^{r_0u}}}}} \]
such that $uf = h$, the corresponding naturality for $r'_1$ is the commutativity of the outside of the diagram on the right, in which the left inner square commutes by the functoriality of $r_0$ and the right inner square commutes by the naturality of $r_1$ with respect to $(u,1_b,v)$. The calculation
\[ (r'_1s_0)_a = (r'_1)_{(1_a,1_a)} = (r_1s_0)_ar_0(1_a) = 1_{r_0a} \]
witnesses the cocone unit axiom for $(r'_0,r'_1)$. Given a composable pair of arrows $((h,k),(f,g))$ in $\tn{Cnr}(X)$ as indicated on the left in
\[ \xygraph{{\xybox{\xygraph{{a}="p0" [d] {x}="p1" [r] {b}="p2" [d] {y}="p3" ([l] {z}="p4",[r] {c}="p5") "p0":"p1"^-{f}:"p2"^-{g}:"p3"^-{h}:@{<-}"p4"^-{\rho}:@{<-}"p1"^-{\lambda}:@{}"p3"|-*{\kappa}:"p5"_-{k}}}}
[r(4)]
{\xybox{\xygraph{!{0;(1.5,0):(0,.6667)::} {r_0a}="p0" [r] {r_0x}="p1" [r] {r_0b}="p2" [d] {r_0y}="p3" [d] {r_0c}="p4" [ul] {r_0z}="p5" "p0":"p1"^-{r_0f}:"p2"^-{(r_1)_g}:"p3"^-{r_0h}:"p4"^-{(r_1)_k}:@{<-}"p5"^-{(r_1)_{k\rho}}:@{<-}"p0"^-{r_0(\lambda f)} "p1":"p5"^-{r_0\lambda}:"p3"^-{(r_1)_{\rho}}}}}} \]
in which $\kappa$ is a chosen opcartesian square. The commutativity of the outside of the diagram on the right, in which the square commutes by the naturality of $r_1$, the upper triangle commutes by the functoriality of $r_0$ and the lower triangle commutes since $(r_1d_0)(r_1d_2) = r_1d_1$, witnesses the equation $(r'_1d_0)(r'_1d_2) = r'_1d_1$ for the component $((h,k),(f,g))$.

It remains to show that the two processes just described are mutually inverse. Supposing $(r_0,r_1) = (q'_0,q'_1)$, $r'_0$ is clearly $q_0$, and for any morphism $(f,g)$ of $\tn{Cnr}(X)$
\[ (r'_1)_{(f,g)} = (q_1)_{(1,g)}(q_1)_{(f,1)} = (q_1)_{(f,g)}   \]
this last equality since $(q_1d_0)(q_1d_2) = q_1d_1$, and so $r'_1 = q_1$. Conversely supposing $(q_0,q_1) = (r'_0,r'_1)$, the functors $q'_0$ and $r_0$ clearly agree on objects, and
\[ \begin{array}{lccr} {q'_0f = (r'_1)_{(f,1)} = r_0f}
&&& {(q'_1)_h = (r'_1)_{(1,h)} = (r_1)_h} \end{array} \]
establishes that $q'_0 = r_0$ and $q'_1 = r_1$.
\end{proof}
This result as two important corollaries. First, putting it together with Theorem \ref{thm:codesc-2-cat} gives the following calculation of codescent for general crossed internal categories, though with some assumptions on $\ca E$.
\begin{cor}\label{cor:codesc-gen-cicat}
Let $\ca E$ be a category with pullbacks and pullback-stable reflexive coequalisers, $\ca K = \tn{Cat}(\ca E)$, and let $X$ be a crossed internal category in $\ca K$. Then the codescent object of $X$ exists and $\tn{CoDesc}(X) = \pi_{0*}(\tn{Cnr}(X))$.
\end{cor}
\begin{rem}\label{rem:codesc-cdc-explicit-cocone}
In the case where $\ca K = \Cat$ so that $X$ is a crossed double category, tracing through the constructions of Theorems \ref{thm:codesc-2-cat} and \ref{thm:codesc-cr-int-cat}, one has an explicit codescent cocone $(q_0,q_1)$
\[ \xygraph{!{0;(2,0):(0,1)::} {X_2}="p0" [r] {X_1}="p1" [r] {X_0}="p2" [r] {\pi_{0*}(\tn{Cnr}(X))}="p3"
"p2":"p1"|-{s_0} "p1":@<1.5ex>"p2"^-{d_1} "p1":@<-1.5ex>"p2"_-{d_0} "p0":@<1.5ex>"p1"^-{d_2} "p0":"p1"|-{d_1} "p0":@<-1.5ex>"p1"_-{d_0}
"p2":"p3"^-{q_0}} \]
in which $q_0$ is the identity on objects and sends a vertical arrow $f : x \to y$ of $X$ to the connected component of the corner $(f,1_y)$. The component of $q_1$ at a horizontal arrow $g : x \to y$ of $X$, is the connected component of the corner $(1_x,g)$. Since for a general morphism $h$ of the codescent object, there exists a corner $(f,g)$ of $X$ such that $q_0(f,g) = h$ and $(1,g) \comp (f,1) = (f,g)$ in $\Cnr(X)$, every such arrow $h$ factors as $h = h_1h_2$ where $h_2$ is in the image of $q_0$ and $h_1$ is a component of the natural transformation $q_1$.
\end{rem}
On the other hand, given a condition on the crossed internal category under consideration one obtains a codescent calculation valid for any $\ca E$ with pullbacks. Recall from Lemma \ref{lem:discrete-Cnr} that if a crossed internal category satisfies the condition that $d_0:X_1 \to X_0$ is a discrete opfibration, then $\Cnr(X)$ is componentwise discrete, and so its codescent object exists and is just $\Cnr(X)$ viewed as an object of $\ca K = \tn{Cat}(\ca E)$.
\begin{cor}\label{cor:codesc-gen-E}
Let $\ca E$ be a category with pullbacks, $\ca K = \tn{Cat}(\ca E)$, and let $X$ be a crossed internal category in $\ca K$ such that $d_0:X_1 \to X_0$ is a discrete opfibration. Then the codescent object of $X$ exists and $\tn{CoDesc}(X) = \tn{Cnr}(X)$.
\end{cor}
\begin{exams}\label{exams:connes-cyclic}
Corollary \ref{cor:codesc-gen-E} applies in particular to the crossed simplicial groups of Loday-Fiedorowicz \cite{FieLoday-CrossedSimplicial}. In such cases the arrows of $\tn{CoDesc}(X)$ are literally formal composites of arrows of two special types, these coming from the vertical and horizontal arrows of the double category $X$. In the most famous example, $X$ is the double category whose category of objects and horizontal arrows is $\Delta$, the vertical arrows are given by cyclic permutations and $\CoDesc(X)$ is Connes' cyclic category.
\end{exams}

\section{Examples}
\label{sec:examples}

This final section is devoted to illustrative examples. In Section \ref{ssec:functions} we exhibit $\S$ as the free symmetric strict monoidal category containing a commutative monoid. Then in Section \ref{ssec:SigmaDelta} we explain how the categories $\Sigma\Delta$ and $\tn{B}\Delta$ of \cite{DayStreet-Substitudes} fit into our framework. The free braided strict monoidal category containing a commutative monoid is explained in Section \ref{ssec:Vines} using the work of Lavers \cite{Lavers-Vines}. In Section \ref{ssec:props} we calculate internal algebra classifiers for adjunctions of 2-monads arising from operad morphisms, and in Section \ref{ssec:braided-props} we look at the non-symmetric and braided analogues of this. For an operad morphism $F : S \to T$ there is in fact more than one naturally associated adjunction of 2-monads, and thus more than one internal algebra classifier one could consider. In Section \ref{ssec:Sigma-free-operads} we understand how these are related, and in Examples \ref{exams:Batanin-Berger-classifiers}, reconcile the internal algebra classifier calculations of \cite{BataninBerger-HtyThyOfAlgOfPolyMnd} with those of Section \ref{ssec:props}.

\subsection{Finite sets and functions.}
\label{ssec:functions}
As in Remark \ref{rem:bij-mon-factorisation} we regard $\Delta_+$ and $\tnb{S}(1) = \P$ as subcategories of the category $\S$, which itself is a skeleton of the category of finite sets. Ordinal sum is the tensor product for a symmetric strict monoidal structure on $\S$, and since $1$ is terminal in $\S$, it underlies a unique commutative monoid. We denote by $u : 1 \to \S$ the corresponding symmetric lax monoidal functor.
\begin{thm}\label{thm:free-smc-containing-a-commutative-monoid}
\cite{Davydov-QComAlg}
The commutative monoid $u : 1 \to \S$ exhibits $\S$ as the internal $\tnb{S}$-algebra classifier.
\end{thm}
\begin{proof}
The underlying category of an internal $\tnb{S}$-algebra classifier is the codescent object of $U^{\tnb{S}}\ca R_{\tnb{S}}1$ by Proposition \ref{prop:dagger-F}, Corollary \ref{cor:alg-class-explicit} and Examples \ref{exams:easy-actions-operadic-examples}. By Corollary \ref{cor:codesc-gen-cicat} this codescent object is computed as $\pi_{0*}\Cnr(U^{\tnb{S}}\ca R_{\tnb{S}}1)$. An explicit description of $\Cnr(U^{\tnb{S}}\ca R_{\tnb{S}}1)$ was given in Example \ref{exam:cnr-bar-of-Sm}. From that explicit description, a pair of arrows $(\rho_1,f_1)$ and $(\rho_2,f_2) : m \to n$ are in the same connected component of the hom category $\Cnr(U^{\tnb{S}}\ca R_{\tnb{S}}1)(m,n)$ iff $f_1\rho_1 = f_2\rho_2$. So an object of the codescent object may be identified as a natural number, and the morphism corresponding to the equivalence class containing the arrow $(\rho,f)$ of $\Cnr(U^{\tnb{S}}\ca R_{\tnb{S}}1)$, can be identified with the composite function $f\rho$. Conversely for any function $h : \underline{m} \to \underline{n}$, take its bijective-monotone factorisation $h = f\rho$, and then the connected component of $(\rho,f)$ is identified with $h$. Thus one has
\[ \CoDesc(U^{\tnb{S}}\ca R_{\tnb{S}}1) = \mathbb{S}. \]
In view of Remark \ref{rem:codesc-cdc-explicit-cocone} the codescent cocone is $(i,\overline{i})$, where $i : \P \to \S$ is the inclusion, and the component of $\overline{i}$ at an order preserving function $f : \underline{m} \to \underline{n}$ is just $f$ viewed as an arrow of $\S$. To see that the given codescent cocone lives in $\Algs {\tnb{S}}$ it suffices by Corollary \ref{cor:alg-class-explicit} to show that
\[ \xygraph{!{0;(2,0):(0,.5)::} {\tnb{S}(\P)}="p0" [r] {\tnb{S}(\S)}="p1" [d] {\S}="p2" [l] {\P}="p3" "p0":"p1"^-{\tnb{S}(i)}:"p2"^-{+}:@{<-}"p3"^-{i}:@{<-}"p0"^-{\mu_1}} \]
commutes, and this is straight forward. Since $i\eta_1$ is the functor which picks out $1$, it is equal to $u$ as a functor, and so the result follows by Remark \ref{rem:universal-int-alg-explicit}.
\end{proof}
\begin{rem}\label{rem:Cnr->S-biequiv}
In terms of the notation of Remark \ref{rem:T-to-the-S}, Theorem \ref{thm:free-smc-containing-a-commutative-monoid} says that $\tnb{S}^{\tnb{S}} = \S$. Observe that the homs of the 2-category $\Cnr(U^{\tnb{S}}\ca R_{\tnb{S}}1)$ are equivalent to discrete categories, and that this is witnessed by the fact that the 2-functor
\[ \Cnr(U^{\tnb{S}}\ca R_{\tnb{S}}1) \longrightarrow \S \]
which sends any corner to its connected component (viewed as a function), is a biequivalence.
\end{rem}
\begin{rem}\label{rem:alternative-proof-via-Br}
We denote by $\pi : \tnb{B} \to \tnb{S}$ the morphism of 2-monads which comes from the process of taking the underlying permutation of a braid ($\pi_1$ is literally this). By Examples \ref{exams:M-Sm-Br-intalg} a $\tnb{B}$-algebra internal to a symmetric monoidal category $\ca V$ is also just a commutative monoid in $\ca V$. Thus one also has $\tnb{S}^{\tnb{B}} = \S$. One can see this directly in terms of codescent objects via a mild variation of the proof of Theorem \ref{thm:free-smc-containing-a-commutative-monoid}, in which the double category $U^{\tnb{S}}\ca R_{\tnb{S}}1$ is replaced by
\[ \xygraph{!{0;(3,0):(0,1)::}
{\tnb{S}(\tnb{B}(\B))}="p0" [r] {\tnb{S}(\B)}="p1" [r] {\P.}="p2"
"p2":"p1"|-{\tnb{S}(\eta^{\tnb{B}}_1)} "p1":@<1.5ex>"p2"^-{\mu^{\tnb{S}}_1\tnb{S}(\pi_1)} "p1":@<-1.5ex>"p2"_-{\tnb{S}(t_{\B})} "p0":@<1.5ex>"p1"^-{\mu^{\tnb{S}}_{\B}\tnb{S}(\pi_{\B})} "p0":"p1"|-{\tnb{S}(\mu^{\tnb{B}}_1)} "p0":@<-1.5ex>"p1"_-{\tnb{S}(\tnb{B}(t_{\B}))}} \]
As a result the corresponding 2-category of corners has more 2-cells, and so is no longer biequivalent to a category. However, these extra 2-cells are then quotiented away by the process of computing the codescent object.
\end{rem}

\subsection{$\Sigma\Delta$ and $\tn{B}\Delta$.}
\label{ssec:SigmaDelta}
Section 4 of \cite{DayStreet-Substitudes} an description of $\Sigma\Delta$, the free symmetric strict monoidal category containing a monoid, and of $\tn{B}\Delta$, the free braided strict monoidal category containing a monoid, was given. We recover $\Sigma\Delta$ by applying our machinery to the monad morphism $\iota : \tnb{M} \to \tnb{S}$ mentioned in Examples \ref{exams:M-Sm-Br-intalg}. We denote by $\ca S$ the double category
\[ \xygraph{!{0;(3,0):(0,1)::}
{\tnb{S}(\tnb{M}(\N))}="p0" [r] {\tnb{S}(\N)}="p1" [r] {\P.}="p2"
"p2":"p1"|-{\tnb{S}(\eta^{\tnb{M}}_1)} "p1":@<1.5ex>"p2"^-{\mu^{\tnb{S}}_1\tnb{M}(\iota_1)} "p1":@<-1.5ex>"p2"_-{\tnb{S}(t_{\N})} "p0":@<1.5ex>"p1"^-{\mu^{\tnb{S}}_{\N}\tnb{S}(\iota_{\N})} "p0":"p1"|-{\tnb{S}(\mu^{\tnb{M}}_1)} "p0":@<-1.5ex>"p1"_-{\tnb{S}(\tnb{M}(t_{\N}))}} \]
underlying the computation of $\tnb{S}^{\tnb{M}}$. In an evident way one can regard $\ca S$ as the sub-double category of $U^{\tnb{S}}\ca R_{\tnb{S}}1$ consisting of all the objects, vertical arrows and horizontal arrows, and just the chosen opcartesian squares. From this explicit description, $\Cnr(\ca S)$ is the underlying category of $\tn{Cnr}(U^{\tnb{S}}\ca R_{\tnb{S}}1)$.

As a functor between discrete categories $t_{\N}$ is a discrete opfibration. Since $\tnb{S}$ is an opfamilial 2-functor by Theorem 4.3.12 and Remark 4.3.13 of \cite{Weber-OpPoly2Mnd}, it preserves discrete opfibrations by Theorem 6.2 of \cite{Weber-Fam2fun}. Formally, the fact that $\Cnr(\ca S)$ turns out to be a category comes from the fact that $\tnb{S}(t_{\N})$ is a discrete opfibration and Lemma \ref{lem:discrete-Cnr}. By the explicit description of $\tn{Cnr}(U^{\tnb{S}}\ca R_{\tnb{S}}1)$, $\tn{Cnr}(\ca S)$ is exactly $\Sigma\Delta$ as described in \cite{DayStreet-Substitudes}. The ``distributive law for $\Delta$ over $\P$'' participating in \cite{DayStreet-Substitudes}'s description of $\Sigma\Delta$ comes from Proposition \ref{prop:span-dist-law-from-crintcat} applied to this case. To summarise, one has $\Sigma\Delta =  \tnb{S}^{\tnb{M}} = \Cnr(\ca S)$.

Similarly, taking instead the monad morphism $\iota' : \tnb{M} \to \tnb{B}$, the double category $\ca B$ underlying the computation of $\tnb{B}^{\tnb{M}}$ is
\[ \xygraph{!{0;(3,0):(0,1)::}
{\tnb{B}(\tnb{M}(\N))}="p0" [r] {\tnb{B}(\N)}="p1" [r] {\B.}="p2"
"p2":"p1"|-{\tnb{B}(\eta^{\tnb{M}}_1)} "p1":@<1.5ex>"p2"^-{\mu^{\tnb{B}}_1\tnb{M}(\iota'_1)} "p1":@<-1.5ex>"p2"_-{\tnb{B}(t_{\N})} "p0":@<1.5ex>"p1"^-{\mu^{\tnb{B}}_{\N}\tnb{B}(\iota'_{\N})} "p0":"p1"|-{\tnb{B}(\mu^{\tnb{M}}_1)} "p0":@<-1.5ex>"p1"_-{\tnb{B}(\tnb{M}(t_{\N}))}} \]
As with case of $\iota$, $\tnb{B}(t_{\N})$ is a discrete opfibration, making $\Cnr(\ca B)$ a category. This category is exactly $\tn{B}\Delta$, and so we have $\tn{B}\Delta = \tnb{B}^{\tnb{M}} = \Cnr(\ca B)$.

\subsection{Vines.}
\label{ssec:Vines}
We turn now to a discussion of the free braided strict monoidal category containing a commutative monoid. A computation of this within our setting proceeds as with Theorem \ref{thm:free-smc-containing-a-commutative-monoid}, except with $\tnb{B}$ replacing $\tnb{S}$ throughout. Thus $\tnb{B}^{\tnb{B}}$ is a category with natural numbers as objects, and a morphism is some braided analogue of a function. Such morphisms have been studied in detail by Lavers \cite{Lavers-Vines} and are called \emph{vines}.

Vines generalise braids in that the strings are allowed to merge.
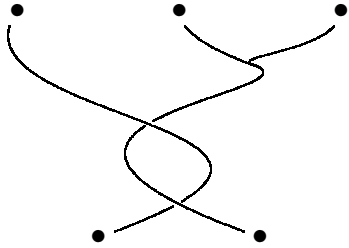
\[ \xygraph{!{0;(.85,0):(0,.7)::}
{\bullet}="t1" [r(2)] {\bullet}="t2" [r(2)] {\bullet}="t3" "t2" [dr] *{}="m"
"t1" [d(4)] [r] {\bullet}="b1" [r(2)] {\bullet}="b2"
"t1"-@`{"t1"+(-1,-2),"m"+(2,-1)}"b1"|(.89)*=<3pt>{}
"t2"-@`{"t2"+(0,-.5)}"m"
"t3"-@`{"t3"+(0,-.5),"m"+(-.5,.1),"m"+(-.05,.05)}"m"
"m"-@`{"m"+(.5,-.5),"m"+(-4,-1)}"b2"|(.52)*=<3pt>{}} \]
Formally, given natural numbers $m$ and $n$ one fixes $m$ distinct points in $\R^3$ along the line $l_I = \{(x,y,z) \, \, : \, \, y = 0, \, z=1\}$ and $n$ distinct points along $l_O = \{(x,y,z) \, \, : \, \, y = 0, \, z=0\}$. The $m$ distinguished points on $l_I$ are linearly ordered according to their $x$-coordinate, and similarly for the $n$ distinguished points of $l_O$. Writing $\underline{m} \cdot \mathbb{I}$ for the coproduct of $m$ copies of the unit interval $\mathbb{I}$, a vine $m \to n$ is an equivalence class of piecewise linear continuous functions $(v_1,...,v_m) : \underline{m} \cdot \mathbb{I} \to \R^3$ such that
\begin{enumerate}
\item For all $i$, $v_i(0)$ is the $i$-th distinguished point on $l_I$.
\item For all $i$, $v_i(1)$ is one of distinguished points on $l_O$.
\item The paths $v_i$ descend from height $z=1$ to height $z=0$ at constant speed.
\item For $i < j$, $v_i(t) = v_j(t)$ implies $v_i(s) = v_j(s)$ for all $s \geq t$.
\end{enumerate}
The equivalence relation is generated by the appropriate notion of homotopy, that is, deforming the path $v_i$ to $v'_i$ by a homotopy whose image is disjoint from those of the $v_j$ for $j \neq i$, exhibits $(v_1,...,v_i,....,v_n)$ and $(v_1,...,v'_i,....,v_n)$ as equivalent vines. The precise definition in the case $m = n$ is given as Definition 1 of \cite{Lavers-Vines}, and it is trivial to adapt this to give a definition of a vine $m \to n$ in general. We follow the aesthetic choice made in \cite{Lavers-Vines} and draw our vines as though they are smooth, which is justified by the fact that any smooth path can be approximated arbitrarily well by piecewise-linear ones.

Vines can be composed vertically giving a category $\mathbb{V}$ whose objects are natural numbers and morphisms are vines. Every vine $v : m \to n$ has an underlying function sending $i \in \underline{m}$ to the distinguished point $v_i(1)$ of $l_O$, this prescription being independent of the piecewise-linear continuous function $(v_1,...,v_m) : \underline{m} \cdot \mathbb{I} \to \R^3$ chosen to represent $v$. Vines whose underlying functions are bijective are exactly braids, and thus the maximal subgroupoid of $\mathbb{V}$ is exactly $\B$. Ordinal sum and horizontal disjoint union of vines is the tensor product for a braided strict monoidal structure on $\mathbb{V}$, the unit of which is $0$.

Given $(v_1,...,v_m) : \underline{m} \cdot \mathbb{I} \to \R^3$ which respresents a vine $v : m \to n$, we call the constituent paths $v_i$ \emph{strings}, and say for $i < j$ that the strings $v_i$ and $v_j$ \emph{merge at time} $t$ when $t \in \mathbb{I}$ is the least such that $v_i(t) = v_j(t)$. The equivalence relation which participates in the definition of vines enables one to move the $t$ at which two strings merge, to be as early or as late as possible, without changing the vine. If for instance consecutive strings $v_i,v_{i+1},...,v_j$ cross over each other before all merging, then by merging as early as possible these crossings can be ignored to produce a simpler representative of the same vine. For example
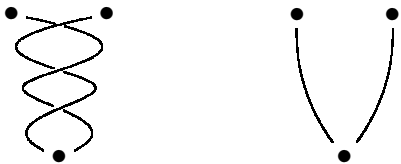
\[ \xygraph{{\xybox{\xygraph{!{0;(1,0):(0,1.5)::}
{\bullet}="t1" [r] {\bullet}="t2" [l(.5)d] {\bullet}="b"
"t1"-@`{"t2"+(1,-.25),"t1"+(-1.5,-.5),"t2"+(.5,-.75)}"b"|(.052)*=<4pt>{}|(.7)*=<4pt>{}
"t2"-@`{"t1"+(-1,-.25),"t2"+(1.5,-.5),"t1"+(-.5,-.75)}"b"|(.3575)*=<4pt>{}}}}
[r(3)]
{\xybox{\xygraph{!{0;(1,0):(0,1.5)::}
{\bullet}="t1" [r] {\bullet}="t2" [l(.5)d] {\bullet}="b"
"t1"-@/_{.5pc}/"b"
"t2"-@/^{.5pc}/"b"}}}} \]
represent the same vine. Thus in particular, $1$ is a terminal object in $\mathbb{V}$, and so underlies a commutative monoid. We denote by $u : 1 \to \mathbb{V}$ the corresponding braided lax monoidal functor.

An order preserving function can be regarded as a vine whose underlying function is order preserving, and which has no crossings. In this way $\Delta_+$ is a subcategory of $\mathbb{V}$. Given a representative $(v_1,...,v_m)$ of an arbitrary vine $v : m \to n$, one can leave merging as late as possible
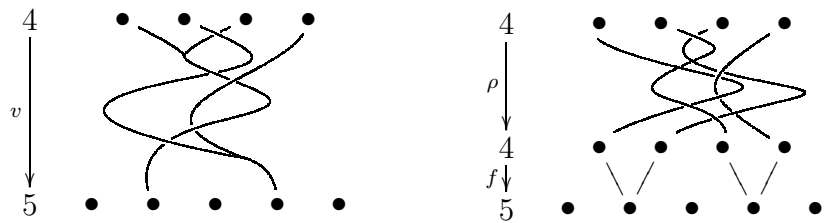
\[ \xygraph{*{\xybox{\xygraph{!{0;(.65,0):(0,3)::} 
{\bullet}="t1" [r] {\bullet}="t2" [r] {\bullet}="t3" [r] {\bullet}="t4" "t1" [l(.5)d]
{\bullet}="b1" [r] {\bullet}="b2" [r] {\bullet}="b3" [r] {\bullet}="b4" [r] {\bullet}="b5"
"t2" [d(.2)] *{}="m1" "b3" [u(.25)r(.5)] *{}="m2"
"t1"-@`{"m1"+(0,.04)}"m1"
"t3"-@`{"m1"+(0,.05)}"m1"|(.28)*=<3pt>{}
"m1"-@`{"m1"+(0,-.04),"t3"+(.5,-.4),"t3"+(.5,-.5),"t2"+(-1,-.65)}"b2"|(.25)*=<3pt>{}|(.84)*=<3pt>{}
"t2"-@`{"t3"+(.4,-.2),"t3"+(.4,-.25),"m1"+(-3,-.3)}"m2"|(.4725)*=<4pt>{}
"t4"-@`{"t4"+(0,-.3),"t2"+(-1.5,-.45)}"m2"|(.8)*=<3pt>{}
"m2"-@`{"m2"+(.5,-.05)}"b4"
"t1" [l(1.5)] {4} :[d] {5}_-{v}}}}
[r(5)]
*!(0,.6){\xybox{\xygraph{!{0;(.65,0):(0,3)::}
{\bullet}="t1" [r] {\bullet}="t2" [r] {\bullet}="t3" [r] {\bullet}="t4" "t1" [d(.6667)]
{\bullet}="m1" [r] {\bullet}="m2" [r] {\bullet}="m3" [r] {\bullet}="m4" "m1" [l(.5)d(.3333)]
{\bullet}="b1" [r] {\bullet}="b2" [r] {\bullet}="b3" [r] {\bullet}="b4" [r] {\bullet}="b5"
"t1"-@`{"t1"+(-.5,-.1),"t3"+(0.5,-.325),"t3"+(0.5,-.35),"m1"+(0,.1)}"m1"|(.36)*=<3pt>{}|(.735)*=<3pt>{}
"t2"-@`{"t3"+(0,-.12),"t3"+(0,-.14),"t1"+(0,-.4),"m3"+(.3,.2)}"m3"|(.38)*=<2pt>{}|(.44)*=<2pt>{}
"t3"-@`{"t2"+(-.5,-.175),"t4"+(0.5,-.36),"t4"+(0.5,-.38),"m2"+(0,.1)}"m2"|(.045)*=<3pt>{}|(.232)*=<3pt>{}|(.81)*=<3pt>{}
"t4"-@`{"t2"+(.5,-.33),"t2"+(.5,-.35)}"m4"|(.61)*=<2pt>{}|(.76)*=<2pt>{}
"b2"(-"m1",-"m2") "b4"(-"m3",-"m4")
"t1" [l(1.5)] {4} :[d(.6667)] {4}_-{\rho} :[d(.3333)] {5}_-{f}}}}} \]
and then decompose $v$ as $v = f\rho$ where $f$ is an order preserving function and $\rho$ is a braid. This factorisation is unique with the extra condition that the restriction of the braid $\rho$ to each fibre of $f$ is trivial. Lavers made this heuristic argument precise in the proof of Proposition 8 of \cite{Lavers-Vines}, and so we have
\begin{lem}\label{lem:bij-mon-V}
\cite{Lavers-Vines}
Every vine $v : m \to n$ factors uniquely as $v = f\rho$ where $f : m \to n$ is an order preserving function, $\rho \in \tn{Br}_m$, and $\rho$'s restriction to each fibre of $f$ is an identity braid.
\end{lem}
This result is the analogue of the bijective-monotone factorisation of Remark \ref{rem:bij-mon-factorisation}, for the category $\mathbb{V}$. We are thus in a position to give the main result of this section, that the category $\mathbb{V}$ is the free braided strict monoidal category containing a commutative monoid. While this result is known to experts, see \cite{Davydov-QComAlg} for instance, I know of no place in the literature where a proof is given.
\begin{thm}\label{thm:free-bmc-containing-a-commutative-monoid}
The commutative monoid $u : 1 \to \mathbb{V}$ exhibits $\mathbb{V} = \tnb{B}^{\tnb{B}}$.
\end{thm}
\begin{proof}
We adapt the proof of Theorem \ref{thm:free-smc-containing-a-commutative-monoid} replacing permutations by braids throughout. The double category $U^{\tnb{B}}\ca R_{\tnb{B}}1$ has the following explicit description. The category of objects and horizontal maps is $\Delta_+$ and that of objects and vertical maps is $\B$. A square is determined uniquely by its boundary, and will exist with boundary
\[ \xygraph{{m}="p0" [r] {n}="p1" [d] {n}="p2" [l] {m}="p3" "p0":"p1"^-{f}:"p2"^-{\rho}:@{<-}"p3"^-{g}:@{<-}"p0"^-{\rho'}} \]
iff $\rho'=\rho(\rho'_k)_{1{\leq}k{\leq}n}$, where $\rho'_k$ is the restriction of the braid $\rho'$ to the fibre $g^{-1}\{k\}$, that is, $\rho'_k$ is the braid obtained from $\rho'$ by ignoring all but the strings which finish in $g^{-1}\{k\}$. Equivalently, a square of the above form exists in $U^{\tnb{B}}\ca R_{\tnb{B}}1$ iff $\rho f = g \rho'$ in $\mathbb{V}$. It will chosen opcartesian iff the $\rho'_k$ are all identity braids. From this explicit description of $U^{\tnb{B}}\ca R_{\tnb{B}}1$, a 2-cell $(\rho_1,f_1) \to (\rho_2,f_2)$ of $\tn{Cnr}(U^{\tnb{B}}\ca R_{\tnb{B}}1)$ is invertible and unique if it exists, and will exist iff $f_1\rho_1 = f_2\rho_2$ in $\mathbb{V}$. Thus a morphism $m \to n$ of $\CoDesc(U^{\tnb{B}}\ca R_{\tnb{B}}1)$ can be identified with the vine $f\rho$ obtained from any element $(f,\rho)$ of the corresponding connected component in $\tn{Cnr}(U^{\tnb{B}}\ca R_{\tnb{B}}1)$. By Lemma \ref{lem:bij-mon-V} every vine $m \to n$ arises in this way, giving the equation
\begin{equation}\label{eq:vine-codesc}
\CoDesc(U^{\tnb{B}}\ca R_{\tnb{B}}1) = \mathbb{V}
\end{equation}
of categories. To reconcile the braided strict monoidal category structures and the internal commutative monoids, one proceeds as in the proof of Theorem \ref{thm:free-smc-containing-a-commutative-monoid}.
\end{proof}
\begin{rem}\label{rem:Lavers-strucure-Theorem}
Restricting to endomorphism monoids the equation (\ref{eq:vine-codesc}) is exactly Theorem 2 of \cite{Lavers-Vines}.
\end{rem}
\begin{rem}\label{rem:Cnr-B-Bar-biequiv-cat}
As in the situation of Theorem \ref{thm:free-smc-containing-a-commutative-monoid}, the identity on objects 2-functor $\tn{Cnr}(U^{\tnb{B}}\ca R_{\tnb{B}}) \to \mathbb{V}$ given on morphisms by $(\rho,f) \mapsto f\rho$ is a biequivalence.
\end{rem}

\subsection{PROP's from operads.}
\label{ssec:props}
Recall that in Section \ref{ssec:explicit-adj2mnd-op-morphism} we described how a morphism $F : S \to T$ of operads with underlying object function $f : I \to J$, gives rise to an adjunction of 2-monads, which we denote as $F : (\Cat/I,S) \to (\Cat/J,T)$, in which the adjunction $F_! \ladj F^*$ is $\Sigma_f \ladj \Delta_f$. We now give an explicit description of the associated internal algebra classifier $T^S$ relative to $F$.

In the general context of a morphism of operads $F : S \to T$ with underlying object function $f : I \to J$, we begin by describing explicitly the strict $T$-algebra $F \downarrow (-)$ and lax $S$-morphism $u_F : 1 \to \overline{F}(F \downarrow (-))$ in Definitions \ref{defn:F-over-j}, \ref{defn:F-over-blank} and \ref{defn:lax-nat-u-F} below. This uses the explicit description of the algebras and algebra morphisms of the polynomial 2-monads arising from operads given in \cite{Weber-OpPoly2Mnd} section 4. In Theorem \ref{thm:intalg-classifier-from-operad-morphism} we then prove that $u_F$ exhibits $F \downarrow (-) = T^S$ by computing the codescent object of $\ca R_F$ using our techniques. This computation generalises that of Theorem \ref{thm:free-smc-containing-a-commutative-monoid}, which is the case where $F$ is $1_{\Com}$.
\begin{defn}\label{defn:F-over-j}
Given a morphism of operads $F : S \to T$ with underlying object function $f : I \to J$ and $j \in J$, we shall now describe the category $F \downarrow j$. An object of $F \downarrow j$ is a pair $((i_k)_{1{\leq}k{\leq}n},\alpha)$, where $i_k \in I$ for $1 \leq k \leq n$ and $\alpha : (fi_k)_k \to j$ is in $T$. A general morphism of $F \downarrow j$ is of the form
\[ (h,(\beta_{k_2})_{k_2}) : ((i_{1,k_1})_{1{\leq}k_1{\leq}n_1},\alpha_1) \longrightarrow ((i_{2,k_2})_{1{\leq}k_2{\leq}n_2},\alpha_2) \]
where $h : \underline{n}_1 \to \underline{n}_2$ is a function, and for $1 \leq k_2 \leq n_2$, $\beta_{k_2} : (i_{1,k_1})_{hk_1=k_2} \to i_{k_2}$ is in $S$. We call $h$ the \emph{indexing function} and the $\beta_{k_2}$ the \emph{decorating operations} of the given morphism. This data is required to satisfy the \emph{commutativity condition} that $\alpha_2(F\beta_{k_2})_{k_2}\rho_h = \alpha_1$, where $h = \phi_h\rho_h$ is the bijective-monotone factorisation of $h$. When $h$ and the $\beta_{k_2}$ are all identities, $(h,(\beta_{k_2})_{k_2})$ is an identity in $F \downarrow j$. The composite of 
\[ \xygraph{!{0;(4,0):(0,1)::}
{((i_{1,k_1})_{1{\leq}k_1{\leq}n_1},\alpha_1)}="p0" [r] {((i_{2,k_2})_{1{\leq}k_2{\leq}n_2},\alpha_2)}="p1" [r] {((i_{3,k_3})_{1{\leq}k_3{\leq}n_3},\alpha_3)}="p2"
"p0":"p1"^-{(h_1,(\beta_{1,k_2})_{k_2})}:"p2"^-{(h_2,(\beta_{2,k_3})_{k_3})}} \]
in $F \downarrow j$ is $(h_2h_1, (\beta_{2,k_3}(\beta_{1,k_2})_{h_2k_2=k_3})_{k_3})$.
\end{defn}
The well-definedness of composition and its unit and associative laws are easily verified. To begin to obtain an intuition for the morphisms of $F \downarrow j$, consider first the extreme cases listed in
\begin{exams}\label{exams:F-over-j}
\begin{enumerate}
\item When $S$ and $T$ are categories so that $F$ is just a functor, $F \downarrow j$ is the usual comma category, hence the notation ``$F \downarrow j$''. In this case the data of an arrow $(i_1,\alpha_2) \to (i_2,\alpha_2)$ the indexing function $h$ is uniquely determined, and so such an arrow is just an arrow $\beta : i_1 \to i_2$ of $S$. The commutativity condition says that the triangle
\[ \xygraph{{fi_1}="p0" [r(2)] {fi_2}="p1" [dl] {j}="p2" "p0":"p1"^-{F\beta}:"p2"^-{\alpha_2}:@{<-}"p0"^-{\alpha_1}} \]
commutes.
\item When $S = T = \Com$ so that $F$ and $j$ are uniquely determined, $F \downarrow j = \mathbb{S}$. In this case the $\alpha$'s and $\beta$'s participating in the definition of an arrow of $F \downarrow j$ in Definition \ref{defn:F-over-j} are uniquely determined, and the commutativity condition is vacuous.
\end{enumerate}
\end{exams}
Using trees to denote operadic operations, an illustration of the data of a more general sort of arrow of $F \downarrow j$ is given in
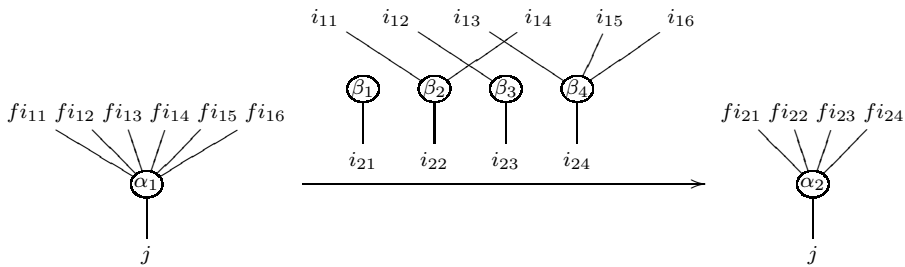
\[ \xygraph{
{\xybox{\xygraph{!{0;(.5,0):(0,1.5)::}
{\scriptstyle{\alpha_1}} *\xycircle<6pt,5pt>{-} (-[d] {\scriptstyle{j}},-[l(2.5)u] {\scriptstyle{fi_{11}}},-[l(1.5)u] {\scriptstyle{fi_{12}}},-[l(.5)u] {\scriptstyle{fi_{13}}},-[r(.5)u] {\scriptstyle{fi_{14}}},-[r(1.5)u] {\scriptstyle{fi_{15}}},-[r(2.5)u] {\scriptstyle{fi_{16}}})}}}
:[r(7)]
{\xybox{\xygraph{!{0;(.5,0):(0,1.5)::}
{\scriptstyle{\alpha_2}} *\xycircle<6pt,5pt>{-} (-[d] {\scriptstyle{j}},-[l(1.5)u] {\scriptstyle{fi_{21}}},-[l(.5)u] {\scriptstyle{fi_{22}}},-[r(.5)u] {\scriptstyle{fi_{23}}},-[r(1.5)u] {\scriptstyle{fi_{24}}})}}}
[l(3.25)u(1)]
{\xybox{\xygraph{!{0;(.5,0):(0,1.5)::}
{\scriptstyle{\beta_1}} *\xycircle<6pt,5pt>{-}="p0" [r(1.5)]
{\scriptstyle{\beta_2}} *\xycircle<6pt,5pt>{-}="p1" [r(1.5)]
{\scriptstyle{\beta_3}} *\xycircle<6pt,5pt>{-}="p2" [r(1.5)]
{\scriptstyle{\beta_4}} *\xycircle<6pt,5pt>{-}="p3"
"p0" [l(.8)u] {\scriptstyle{i_{11}}}="q0" [r(1.5)] {\scriptstyle{i_{12}}}="q1" [r(1.5)] {\scriptstyle{i_{13}}}="q2" [r(1.5)] {\scriptstyle{i_{14}}}="q3" [r(1.5)] {\scriptstyle{i_{15}}}="q4" [r(1.5)] {\scriptstyle{i_{16}}}="q5"
"p0" -[d] {\scriptstyle{i_{21}}}
"p1" (-[d] {\scriptstyle{i_{22}}},-"q0",-"q3")
"p2" (-[d] {\scriptstyle{i_{23}}},-"q1")
"p3" (-[d] {\scriptstyle{i_{24}}},-"q2",-"q4",-"q5")}}}} \]
in which the underlying function $h : \underline{6} \to \underline{4}$ is given by $1,4 \mapsto 2$, $2 \mapsto 3$ and $3,5,6 \mapsto 4$; and the constituent morphisms of $S$ are $\beta_1 : () \to i_{21}$, $\beta_2 : (i_{11},i_{14}) \to i_{22}$, $\beta_3 : i_{12} \to i_{23}$ and $\beta_4 : (i_{13},i_{15},i_{16}) \to i_{24}$. In this way a morphism of $F \downarrow j$ is a function between finite sets labelled by the operations of $S$. The commutativity condition in our example says that
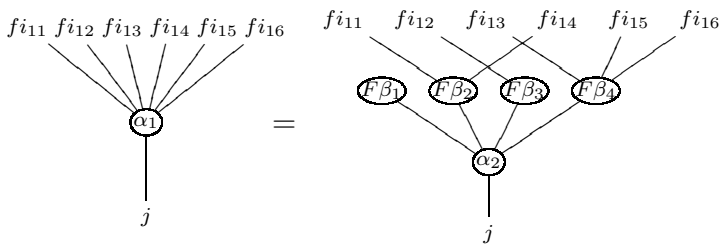
\[ \xygraph{
*!(0,1){\xybox{\xybox{\xygraph{!{0;(.5,0):(0,2)::}
{\scriptstyle{\alpha_1}} *\xycircle<6pt,5pt>{-} (-[d] {\scriptstyle{j}},-[l(2.5)u] {\scriptstyle{fi_{11}}},-[l(1.5)u] {\scriptstyle{fi_{12}}},-[l(.5)u] {\scriptstyle{fi_{13}}},-[r(.5)u] {\scriptstyle{fi_{14}}},-[r(1.5)u] {\scriptstyle{fi_{15}}},-[r(2.5)u] {\scriptstyle{fi_{16}}})}}}}
[r(3)] {=} [r(2.5)]
{\xybox{\xygraph{!{0;(.5,0):(0,1.5)::}
{\scriptstyle{F\beta_1}} *\xycircle<9pt,5pt>{-}="p0" [r(1.5)]
{\scriptstyle{F\beta_2}} *\xycircle<9pt,5pt>{-}="p1" [r(1.5)]
{\scriptstyle{F\beta_3}} *\xycircle<9pt,5pt>{-}="p2" [r(1.5)]
{\scriptstyle{F\beta_4}} *\xycircle<9pt,5pt>{-}="p3" [l(2.25)d]
{\scriptstyle{\alpha_2}} *\xycircle<6pt,5pt>{-}="p4"
"p0" [l(.8)u] {\scriptstyle{fi_{11}}}="q0" [r(1.5)] {\scriptstyle{fi_{12}}}="q1" [r(1.5)] {\scriptstyle{fi_{13}}}="q2" [r(1.5)] {\scriptstyle{fi_{14}}}="q3" [r(1.5)] {\scriptstyle{fi_{15}}}="q4" [r(1.5)] {\scriptstyle{fi_{16}}}="q5"
"p0" -"p4"
"p1" (-"q0",-"q3",-"p4")
"p2" (-"q1",-"p4")
"p3" (-"q2",-"q4",-"q5",-"p4")
"p4" -[d] {\scriptstyle{j}}}}}} \]
in $T$.

Recall from \cite{Weber-OpPoly2Mnd} that when regarding $T$ as a 2-monad on $\Cat/J$, a strict $T$-algebra is a strict morphism of operads $T \to \CatAsOp$ in the sense of \cite{Weber-OpPoly2Mnd} Definitions 4.1 and 4.2. We will now exhibit the extra structure which exhibits $j \mapsto F \downarrow j$ as a strict $T$-algebra.
\begin{defn}\label{defn:F-over-blank}
The strict morphism of operads $F \downarrow (-) : T \to \CatAsOp$ is given on objects by $j \mapsto F \downarrow j$. For $\gamma : (j_k)_{1{\leq}k{\leq}p} \to j$ in $T$ the associated product $F \downarrow \gamma$ has object map as on the right
\[ \begin{array}{lccr} {F \downarrow \gamma : \prod_{k=1}^p F \downarrow j_k \longrightarrow F \downarrow j} &&& {((i_{kl})_{1{\leq}l{\leq}n_k},\alpha_k)_k \mapsto ((i_{kl})_{kl},\gamma(\alpha_k)_k)} \end{array} \]
where for each $k$, $\alpha_k : (fi_{kl})_{l} \to j_k$ is in $T$. Given morphisms
\[ (h_k,(\beta_{k,l_2})_{l_2}) : ((i_{k,1,l_1})_{1{\leq}l_1{\leq}n_{k,1}},\alpha_{k,1}) \longrightarrow ((i_{k,2,l_2})_{1{\leq}l_2{\leq}n_{k,2}},\alpha_{k,2}) \]
of $F \downarrow j_k$ for $1 \leq k \leq p$, $F \downarrow \gamma$ sends these to
\[ (h_1 + ... + h_p, (\beta_{k,l_2})_{k,l_2}) : ((i_{k,1,l_1})_{k,l_1},\gamma(\alpha_{k,1})_k) \longrightarrow ((i_{k,2,l_2})_{k,l_2},\gamma(\alpha_{k,2})_k).  \]
Given $\gamma$ as above and a permutation $\rho \in \Sigma_p$ we must describe the corresponding symmetry
\[ \xygraph{!{0;(2,0):(0,.5)::} {\prod_{k=1}^p F \downarrow j_{\rho k}}="p0" [r(2)] {\prod_{k=1}^p F \downarrow j_k}="p1" [dl] {F \downarrow j}="p2" "p0":"p1"^-{\sigma_{\rho}}:"p2"^-{F \downarrow \gamma}:@{<-}"p0"^-{F \downarrow (\gamma\rho)} "p0" [d(.5)r(.9)] :@{=>}[r(.2)]^-{\xi_{\gamma,\rho}}} \]
and so we define the component at $\sigma_{\rho}^{-1}((i_{kl})_{l},\alpha_k)_k = ((i_{\rho k,l})_{kl},\alpha_{\rho k})_k$ by
\[ (\rho(1_{n_k})_k,(1_{i_{kl}})_{kl}) : ((i_{\rho k,l})_{l},\alpha_{\rho k})_k \longrightarrow ((i_{kl})_l,\alpha_k)_k. \]
\end{defn}
\begin{rem}\label{rem:F-over-blank-well-defined}
To see that $F \downarrow \gamma$ is well-defined note first that the bijective-monotone factorisation is compatible with ordinal sum in the sense that $\phi_{h_1 + ... + h_p} = \phi_{h_1} + ... + \phi_{h_p}$ and $\rho_{h_1 + ... + h_p} = \rho_{h_1} + ... + \rho_{h_p}$, and so the commutativity condition of $(F \downarrow \gamma)(h_k,(\beta_{k,l_2})_{l_2})_k$ is witnessed by the calculation
\[ (\gamma(\alpha_{k,1})_k(F\beta_{k,l_2})_{k,l_2}(\rho_{h_1 + ... + h_k}) = \gamma(\alpha_{k,1}(F\beta_{k,l_2})_{l_2}\rho_{h_k})_k = \gamma(\alpha_{k,2})_k. \]
The functoriality of $F \downarrow \gamma$ is easily verified. To see that the symmetry $\xi_{\gamma,\rho}$ is well-defined, we note that the commutativity condition of $(\rho(1_{n_k})_k,(1_{i_{kl}})_{kl})$ is witnessed by the equation
\[ (\gamma\rho)(\alpha_{\rho k})_k = (\gamma(\alpha_k)_k)(\rho(1_{n_k})_k). \]
The naturality of $\xi_{\gamma,\rho}$ is easily verified. Since the $F \downarrow \gamma$ are given by postcomposing with $\gamma$, the unit and substitution data for $F \downarrow (-)$ are identities by the unit and associativity laws for composition in $T$, and the axiom for the symmetries follows from the equivariance of composition in $T$.
\end{rem}
Since the 2-functor $\overline{F} : \Algs T \to \Algs S$ is the process of precomposing with $F$, we denote the strict $S$-algebra $\overline{F}(F \downarrow (-))$ as $F \downarrow F(-)$. From \cite{Weber-OpPoly2Mnd} Theorem 4.13, to give a lax morphism of $S$-algebras is to give a lax-natural transformation between the corresponding operad morphisms $S \to \CatAsOp$ in the sense of \cite{Weber-OpPoly2Mnd} Definition 4.10.
\begin{defn}\label{defn:lax-nat-u-F}
The lax morphism of $S$-algebras
\[ u_F : 1 \longrightarrow F \downarrow F(-) \]
has components $u_{F,i} : 1 \to F \downarrow fi$ which pick out $(i,1_{fi})$. Given $\alpha : (i_k)_{1{\leq}k{\leq}n} \to i$ in $S$, the unique component of $\overline{u}_{F,\alpha}$ is $(t_n,\alpha) : ((i_k)_k,F\alpha) \to (i,1_{fi})$ where $t_n$ is the unique function $n \to 1$.
\end{defn}
\begin{thm}\label{thm:intalg-classifier-from-operad-morphism}
Let $F : S \to T$ be a morphism of operads. Then the lax $S$-morphism $u_F$ exhibits $F \downarrow (-)$ as the internal $S$-algebra classifier relative to $F$.
\end{thm}
\begin{proof}
This proof is an elaborate version of the proof of Theorem \ref{thm:free-smc-containing-a-commutative-monoid}. As such, we shall proceed via the following steps.
\begin{enumerate}
\item For each $j \in J$, describe the crossed double category $(U^T\ca R_F1)_j$ explicitly.\label{proof-step-cdc}
\item For each $j \in J$, give an explicit description of $\Cnr((U^T\ca R_F1)_j)$. \label{proof-step-corners}
\item Identify $\pi_{0*}\Cnr((U^T\ca R_F1)_j)$ with the category $F \downarrow j$ of Definition \ref{defn:F-over-j}. Thus by Corollary \ref{cor:codesc-gen-cicat} $\CoDesc(U^T\ca R_F1)$ in $\Cat/J$ underlies $F \downarrow (-)$. \label{proof-step-underlying-codesc}
\item Describe the codescent cocone for (\ref{proof-step-underlying-codesc}) explicitly using Remark \ref{rem:codesc-cdc-explicit-cocone}. \label{proof-step-codescent-cocone}
\item Since $T$ preserves all sifted colimits by Examples \ref{exams:easy-actions-operadic-examples}, we reconcile the strict $T$-algebra structure on $F \downarrow (-)$ induced from (\ref{proof-step-codescent-cocone}) and Corollary \ref{cor:alg-class-explicit}, with that of Definition \ref{defn:F-over-blank}. In other words, this step exhibits the codescent cocone of (\ref{proof-step-codescent-cocone}) as living in $\Algs T$. \label{proof-step-codesc-in-TAlgs}
\item Identify $u_F$ of Definition \ref{defn:lax-nat-u-F} with the lax $S$-morphism obtained by applying Lemma \ref{lem:unit-in-terms-of-codescent} to the codescent cocone in $\Algs T$ established in (\ref{proof-step-codesc-in-TAlgs}).\label{proof-step-universal-internal-data}
\end{enumerate}
The result then follows by Proposition \ref{prop:dagger-F}.

\underline{Step \ref{proof-step-cdc}}: An object of $(U^T\ca R_F1)_j$ is an object of $(TF_!1)_j$. By Lemma 3.10 of \cite{Weber-OpPoly2Mnd} and since $F_!1$ is just $f : I \to J$ regarded as an object of $\Cat/J$, such an object is a pair $((i_k)_{1{\leq}k{\leq}n},\alpha)$ where $i_k \in I$ and $\alpha : (fi_k)_k \to j$ is in $T$. Thus an object of $(U^T\ca R_F1)_j$ is an object of $F \downarrow j$. A vertical morphism of $(U^T\ca R_F1)_j$ is a morphism of $(TF_!1)_j$. Since $I$ is discrete vertical morphisms are thus of the form $\rho : ((i_{\rho k})_k,\alpha\rho) \longrightarrow ((i_k)_k,\alpha)$ where $(i_k)_k$ and $\alpha$ are as above, and $\rho \in \Sigma_n$. Thus one can identify the vertical arrows of $(U^T\ca R_F1)_j$ as morphisms of $F \downarrow j$ whose indexing functions are bijective and decorating operations are identities. Composition of vertical arrows is as in $F \downarrow j$.

The horizontal morphisms of $(U^T\ca R_F1)_j$ are the objects of $(TF_!S1)_j$. As a functor into $J$, $F_!S1$ is the composite $ft_S$, and so an object of $(TF_!S1)_j$ is a pair $(\alpha,(\beta_k)_{1{\leq}k{\leq}n})$, where for each $k$, $\beta_k : (i_{kl})_{1{\leq}l{\leq}m_k} \to i_k$ is an arrow of $S$, and $\alpha : (fi_k)_k \to j$ is an arrow of $T$. One computes the sources and targets as
\[ \begin{array}{rcl}
{\mu^T_{F_!1}T(F^c_1)(\alpha,(\beta_k)_{1{\leq}k{\leq}n})} & = & {((i_{kl})_{kl},\alpha(F\beta_k)_k)} \\
{TF_!t_{S1}(\alpha,(\beta_k)_{1{\leq}k{\leq}n})} & = & {((i_k)_k,\alpha).}
\end{array} \]
Putting $n_1 = m_1 + ... + m_k$, $n_2 = n$, $\alpha_1 = \alpha(F\beta_k)_k$ and $\alpha_2 = \alpha$, a typical horizontal morphism of $(U^T\ca R_F1)_j$ can be written as
\[ (\phi,(\beta_{k_2})_{1{\leq}k_2{\leq}n_2}) : ((i_{1,k_1})_{1{\leq}k_1{\leq}n_1},\alpha_1) \longrightarrow ((i_{2,k_2})_{1{\leq}k_2{\leq}n_2},\alpha_2) \]
where $\phi : \underline{n}_1 \to \underline{n}_2$ is in $\Delta_+$, and $\beta_{k_2} : (i_{1,k_1})_{\phi k_1=k_2} \to i_{k_2}$ is in $S$ for $1 \leq k_2 \leq n_2$; such that $\alpha_1 = \alpha_2(F\beta_{k_2})_{k_2}$. In other words, a horizontal arrow of $(U^T\ca R_F1)_j$ can be identified as a morphism of $F \downarrow j$ whose indexing function is order preserving. It is straight forward to verify that the composition of horizontal arrows is as in $F \downarrow j$.

The squares of $(U^T\ca R_F1)_j$ are the morphisms of $(TF_!S1)_j$, and so by definition are of the form
\[ (\rho,(\rho_k)_k) : (\alpha\rho,(\beta_{\rho k}\rho_k)_k) \longrightarrow (\alpha,(\beta_k)_k) \]
where $\alpha : (i_k)_{1{\leq}k{\leq}n} \to fj$ is in $T$, $\beta_k : (i_{kl})_{1{\leq}l{\leq}m_k} \to i_k$ is in $S$ for $1 \leq k \leq n$, $\rho \in \Sigma_n$ and $\rho_k \in \Sigma_{m_{\rho k}}$ for $1 \leq k \leq n$. The sources and targets of such a square are as in
\[ \xygraph{!{0;(4,0):(0,.25)::} {((i_{\rho k,\rho_k l})_{kl},(\alpha\rho)(F\beta_{\rho k}\rho_k)_k)}="p0" [r] {((i_{\rho k})_k,\alpha\rho)}="p1" [d] {((i_k)_k,\alpha)}="p2" [l] {((i_{kl})_{kl},\alpha(F\beta_k)_k)}="p3" "p0":"p1"^-{(\psi,(\beta_{\rho k})\rho_k)_k}:"p2"^-{\rho}:@{<-}"p3"^-{(\phi,(\beta_k)_k)}:@{<-}"p0"^-{\rho(\rho_k)_k}} \]
where $\phi : \underline{m} \to \underline{n}$ is the monotone map whose fibres have cardinalities $(m_1,...,m_n)$ and $\psi : \underline{m} \to \underline{n}$ is the monotone map whose fibres have cardinalities $(m_{\rho 1},...,m_{\rho n})$. Thus any such square is determined uniquely by its boundary, and renotating as above, the vertical and horizontal arrows as in
\begin{equation}\label{eq:square-for-general-classifier}
\xygraph{!{0;(4.5,0):(0,.2222)::}
{((i_{1,\rho_1k_1})_{1{\leq}k_1{\leq}n_1},\alpha_1\rho_1)}="p0" [r] {((i_{2,\rho_2k_2})_{1{\leq}k_2{\leq}n_2},\alpha_2\rho_2)}="p1" [d] {((i_{2,k_2})_{1{\leq}k_2{\leq}n_2},\alpha_2)}="p2" [l] {((i_{1,k_1})_{1{\leq}k_1{\leq}n_1},\alpha_1)}="p3" "p0":"p1"^-{(\phi_1,(\beta_{1,k_2})_{k_2})}:"p2"^-{(\rho_2,(1_{i_{2,\rho_2k_2}})_{k_2})}:@{<-}"p3"^-{(\phi_2,(\beta_{2,k_2})_{k_2})}:@{<-}"p0"^-{(\rho_1,(1_{i_{1,\rho_1k_1}})_{k_1})}}
\end{equation}
bound a square of $(U^T\ca R_F1)_j$ iff the square commutes in $F \downarrow j$. It is straight forward to verify that the horizontal and vertical composition of squares is just horizontal and vertical pasting of commutative squares in $F \downarrow j$.

By the way that $T$ preserves split opfibrations as explained in Lemma 6.3 of \cite{Weber-Fam2fun}, $(\rho,(\rho_k)_k)$ is chosen $TF_!t_{S1}$-opcartesian iff the $\rho_k$ are identities, and so a square (\ref{eq:square-for-general-classifier}) of $(U^T\ca R_F1)_j$ is chosen opcartesian iff $\rho_1$ is order preserving on the fibres of $\phi_2$. Given $(\phi,(\beta_{k_2})_{k_2})$ and $\rho$ in $(U^T\ca R_F1)_j$ as in the solid parts of
\[ \xygraph{!{0;(4.5,0):(0,.2222)::}
{((i_{1,k_1})_{1{\leq}k_1{\leq}n_1},\alpha_1)}="p0" [r] {((i_{2,\rho k_2})_{1{\leq}k_2{\leq}n_2},\alpha_2\rho)}="p1" [d] {((i_{2,k_2})_{1{\leq}k_2{\leq}n_2},\alpha_2)}="p2" [l] {((i_{1,\rho_2^{-1}k_1})_{1{\leq}k_1{\leq}n_1},\alpha_1\rho_2^{-1})}="p3" "p0":"p1"^-{(\phi,(\beta_{k_2})_{k_2})}:"p2"^-{(\rho,(1_{i_{2,\rho k_2}})_{k_2})}:@{<.}"p3"^-{(\phi_2,(\beta_{\rho^{-1}k_2})_{k_2})}:@{<.}"p0"^-{(\rho_2,(1_{i_{1,k_1}})_{k_1})}} \]
the corresponding chosen opcartesian square has boundary as indicated, in which  $\rho\phi = \phi_2\rho_2$ is the bijective-monotone factorisation of $\rho\phi$.

\underline{Step \ref{proof-step-corners}}: The 2-category $\tn{Cnr}(U^T\ca R_F1)_j$ has objects those of $F \downarrow j$, and an arrow $a \to b$ in $\tn{Cnr}(U^T\ca R_F1)_j$ is a composable pair
\[ \xygraph{{a}="p0" [r] {c}="p1" [r] {b}="p2" "p0":"p1"^-{g}:"p2"^-{h}} \]
of arrows of $F \downarrow j$, in which the indexing function of $g$ is bijective, the decorating operations of $g$ are identities, and the indexing function of $h$ is order preserving. Given $(g_1,h_1)$ and $(g_2,h_2) : a \to b$ in $\tn{Cnr}(U^T\ca R_F1)_j$, a 2-cell between them will be unique and invertible if it exists, and will exist iff $h_1g_1 = h_2g_2$ in $F \downarrow j$.

\underline{Step \ref{proof-step-underlying-codesc}}: Thus $(g_1,h_1)$ and $(g_2,h_2) : a \to b$ will be in the same connected component of $\tn{Cnr}(U^T\ca R_F1)_j(a,b)$ iff $h_1g_1 = h_2g_2$ in $F \downarrow j$. Thus an arrow of $\pi_{0*}\tn{Cnr}(U^T\ca R_F1)_j$ may be identified with a morphism of $F \downarrow j$. To see that an arbitrary morphism
\[ (\phi,(\beta_{k_2})_{1{\leq}k_2{\leq}n_2}) : ((i_{1,k_1})_{1{\leq}k_1{\leq}n_1},\alpha_1) \longrightarrow ((i_{2,k_2})_{1{\leq}k_2{\leq}n_2},\alpha_2) \]
of $F \downarrow j$ may be so regarded, take the bijective-monotone factorisation $\phi = \phi_2\rho$, and then $(\phi,(\beta_{k_2})_{k_2})$ factors as
\[ \xygraph{!{0;(4.5,0):(0,1)::} 
{((i_{1,k_1})_{1{\leq}k_1{\leq}n_1},\alpha_1)}="p0" [r] {((i_{1,\rho^{-1}k_1})_{1{\leq}k_1{\leq}n_1},\alpha_1\rho^{-1})}="p1" [r] {((i_{2,k_2})_{1{\leq}k_2{\leq}n_2},\alpha_2).}="p2"
"p0":"p1"^-{(\rho,(1_{i_{1,k_1}})_{k_1})}:"p2"^-{(\phi_2,(\beta_{k_2})_{k_2})}} \]

\underline{Step \ref{proof-step-codescent-cocone}}: The 1-cell part $q_{0,j} : (TF_!S1)_j \longrightarrow F \downarrow j$ of the codescent cocone underlying step \ref{proof-step-underlying-codesc}, is the inclusion of all the objects, and those morphisms whose indexing functions are bijective and decorating operations are identities. The component of the 2-cell part $q_{1,j}$ of this codescent cocone at a horizontal arrow is that arrow regarded as an arrow of $F \downarrow j$.

\underline{Step \ref{proof-step-codesc-in-TAlgs}}: Denoting by $a_F : T(F \downarrow (-)) \to F \downarrow (-)$ the action coming from Definiton \ref{defn:F-over-blank}, in order to reconcile this action with that induced by step \ref{proof-step-codescent-cocone} and Corollary \ref{cor:alg-class-explicit}, we must show that
\begin{equation}\label{eq:square-for-alg-structure-reconciliation}
\xygraph{!{0;(2.5,0):(0,.4)::} {(T^2F_!1)_j}="p0" [r] {T(F \downarrow (-))_j}="p1" [d] {F \downarrow j}="p2" [l] {(TF_!1)_j}="p3" "p0":"p1"^-{(Tq_0)_j}:"p2"^-{a_{F,j}}:@{<-}"p3"^-{q_{0,j}}:@{<-}"p0"^-{\mu^T_{F_!S1,j}}}
\end{equation}
commutes for all $j \in J$. An object of $(T^2F_!1)_j$ is a triple $(\alpha,(\alpha_k)_k,(i_{kl})_{kl})$ where $\alpha : (j_k)_k \to j$ and $\alpha_k : (j_{kl})_l \to j_k$ are in $T$, and $i_{kl} \in f^{-1}\{j_{kl}\}$. Since on objects the algebra structure of $F \downarrow (-)$ is given by substitution in $T$, one has
\[ q_{0,j}\mu^T_{F_!S1,j}(\alpha,(\alpha_k)_k,(i_{kl})_{kl}) = ((i_{kl})_{kl},\alpha(\alpha_k)_k) = a_{F,j}(Tq_0)_j(\alpha,(\alpha_k)_k,(i_{kl})_{kl}) \]
and so (\ref{eq:square-for-alg-structure-reconciliation}) commutes on objects. A morphism of $(T^2F_!1)_j$ is of the form
\begin{equation}\label{eq:morphism-T2F1j}
(\rho,(\rho_k)_k,(1_{i_{\rho k,\rho_{\rho k}l}})_{kl}) : (\alpha\rho,(\alpha_k\rho_k)_k,(i_{\rho k,\rho_{\rho k}l})_{kl}) \longrightarrow (\alpha,(\alpha_k)_k,(i_{kl})_{kl}).
\end{equation}
Applying $q_{0,j}\mu^T_{F_!S1,j}$ to (\ref{eq:morphism-T2F1j}) gives the unique morphism of $F \downarrow j$ whose codomain is $((i_{kl})_{kl},\alpha(\alpha_k)_k)$, indexing function is $\rho(\rho_k)_k$ and decorating operations are identities. Applying $(Tq_0)_j$ to (\ref{eq:morphism-T2F1j}) gives
\[ (\rho,(h_k)_k) : (\alpha\rho,((i_{\rho k,\rho_kl})_l)_k) \longrightarrow (\alpha,((i_{kl})_l,\alpha_k)_k) \]
in which $h_k$ has indexing function $\rho_k$ and decorating operations identities. To apply $a_{F,j}$ to this, is to apply the arrow map of the appropriate product functor given in Definition \ref{defn:F-over-blank}. This produces the unique morphism of $F \downarrow j$ whose codomain is $((i_{kl})_{kl},\alpha(\alpha_k)_k)$, indexing function is $\rho(\rho_k)_k$ and decorating operations are identities, and so (\ref{eq:square-for-alg-structure-reconciliation}) commutes on arrows.

\underline{Step \ref{proof-step-universal-internal-data}}: Our task is to verify
\begin{equation}\label{eq:internal-structure-reconciliation}
\begin{array}{lccr} {u_F = F^*(q_0)F^*(\eta^T_{F_!1})\eta^F_1} &&& {\overline{u}_F = F^*(q_1)F^*(\eta^T_{F_!S1})\eta^F_{S1}} \end{array}
\end{equation}
The first equation of (\ref{eq:internal-structure-reconciliation}) says that for all $i \in I$, the composite
\[ \xygraph{{1}="p0" [r(2)] {(F_!1)_{fi}=f^{-1}\{fi\}}="p1" [r(3)] {(TF_!1)_{fi}}="p2" [r(2)] {F \downarrow fi}="p3" "p0":"p1"^-{i}:"p2"^-{\eta^T_{F_!1,fi}}:"p3"^-{q_{0,Fi}}} \]
picks out $(i,1_{fi})$, and this is clear. Recall that an object of $S1$ over $i$ is an operation $(i_k)_k \to i$ in $S$. The second equation of (\ref{eq:internal-structure-reconciliation}) says that the component at $\alpha$ of the natural transformation
\[ \xygraph{!{0;(2.5,0):(0,.4)::} {(S1)_i}="p0" [r] {(F_!S1)_{fi}}="p1" [r] {(TF_!S1)_{fi}}="p2" [r] {F \downarrow fi}="p3"
"p0":"p1"^-{(\eta^F_{S1})_i}:"p2"^-{\eta^T_{F_!S1,fi}}:@/^{1pc}/"p3"|-{}="t"
"p2":@/_{1pc}/"p3"|-{}="b"
"t":@{}"b"|(.25){}="d"|(.75){}="c" "d":@{=>}"c"^-{q_1}} \]
is $(t_n,\alpha) : ((i_k)_k,F\alpha) \to (i,1_{fi})$, which is clear because of the explicit descriptions of $TF_!S1$ and $q_1$ given in steps \ref{proof-step-cdc} and \ref{proof-step-codescent-cocone} respectively.
\end{proof}
We consider now the case where $T$ is an operad with set of colours $I$, and $F$ is the unique operad morphism $T \to \Com$. In that case the symmetric strict monoidal category $\Com^T$ has the following explicit description from Definitions \ref{defn:F-over-j} and \ref{defn:F-over-blank}, and Theorem \ref{thm:intalg-classifier-from-operad-morphism}. An object is a finite sequence $(i_1,...,i_n)$ of elements of $I$. A morphism is of the form
\[ (f,(\alpha_k)_{1{\leq}k{\leq}n}) : (i_1,...,i_m) \longrightarrow (j_1,...,j_n) \]
in which $f : \underline{m} \to \underline{n}$ is the indexing function, and for each $k$, $\alpha_k : (i_l)_{fl=k} \to j_k$ is the corresponding decorating operation from $T$. Composition comes from operadic substitution for $T$ and the composition of functions. The tensor product on objects is given by concatenation of sequences. When $T$ has a single colour, $\Com^T$ is exactly the PROP associated to the operad $T$.

By Definition \ref{defn:lax-nat-u-F} and Theorem \ref{thm:intalg-classifier-from-operad-morphism} the universal $T$-algebra in $\Com^T$ has underlying $I$-indexed family of objects $(O_i \, : \, i \in I)$ given by $O_i = i$. For each operation $\alpha : (i_1,...,i_n) \to i$ of $T$, the corresponding operad structure map
\[ O_{i_1} \tensor ... \tensor O_{i_n} \longrightarrow O_i \]
is given by the morphism $(t_n,\alpha) : (i_1,...,i_n) \to i$ of $\Com^T$.

A symmetric monoidal category $\ca V$ can be regarded as an operad $\ca U\ca V$, with colours the objects of $\ca V$, and operations $(A_1,...,A_n) \to B$ given by morphisms $A_1 \tensor ... \tensor A_n \to B$. The endomorphism operad of $X \in \ca V$ is then just the full suboperad of $\ca U\ca V$ on the single object $X$, and an algebra of $T$ in $\ca V$ is by definition an operad morphism $T \to \ca U\ca V$. Restricting attention to strict monoidal categories $\ca V$, $\ca V \mapsto \ca U\ca V$ is the effect on objects of a 2-functor as on the left
\[ \begin{array}{lccr} {\ca U_{\tn{s}} : \Algs {\tnb{S}} \longrightarrow \tnb{Opd}} &&& {\ca U_{\tn{ps}} : \PsAlg {\tnb{S}} \longrightarrow \tnb{Opd}} \end{array} \]
and the general assignment $\ca V \mapsto \ca U\ca V$ is the effect on objects of a 2-functor as on the right in the previous display.
\begin{cor}\label{cor:smoncat-opd-adjunctions}
The assignment $T \mapsto \Com^T$ is the effect on objects of a left 2-adjoint to $\ca U_{\tn{s}}$, and a left biadjoint to  $\ca U_{\tn{ps}}$.
\end{cor}
\begin{proof}
The strict universal property $\Com^T$ has by Definition \ref{def:int-alg-classifier}, says that $T \mapsto \Com^T$ is the effect on objects of a left 2-adjoint to $\ca U_{\tn{s}}$. The bicategorical universal property $\Com^T$ enjoys by virtue of Proposition \ref{prop:uprop-intalg-pseudo}, says exactly that $T \mapsto \Com^T$ is the effect on objects of a left biadjoint to $\ca U_{\tn{ps}}$.
\end{proof}

\subsection{Non-symmetric and braided variants.}
\label{ssec:braided-props}
As explained in Section \ref{ssec:polynomials} there is a non-symmetric and a braided variant of the viewpoint of operads as polynomial 2-monads over $\tnb{S}$. The developments of the previous section also have non-symmetric and braided variants.

In the non-symmetric version, one views a morphism $F : S \to T$ of non-symmetric operads as a morphism of polynomial 2-monads over $\tnb{M}$. This can then be regarded as an adjunction
\[ F : (\Cat/I,S) \longrightarrow (\Cat/J,T) \]
of 2-monads, where $I$ and $J$ are the sets of colours of $S$ and $T$ respectively, and the adjunction $F_! \ladj F^*$ is $\Sigma_f \ladj \Delta_f$, where $f$ is the object function of $F$. An explicit description of this adjunctions of 2-monads is along the lines of that given in Section \ref{ssec:explicit-adj2mnd-op-morphism} in the symmetric case, except that one omits mention of any permutations.

The non-symmetric analogue of the definition of $F \downarrow (-)$ is the same as that given in Section \ref{ssec:props}, except that in the morphisms of $F \downarrow j$ described in Definition \ref{defn:F-over-j}, the indexing functions are only allowed to be order preserving. With no permutations to worry about, composition in $F \downarrow j$ is more easily described, and there is no need to use the bijective-monotone factorisation of indexing functions as one does in Definition \ref{defn:F-over-j}. One then has the analogue of Theorem \ref{thm:intalg-classifier-from-operad-morphism} that $F \downarrow (-)$ is the free strict $T$-algebra containing an internal $S$-algebra. The proof of this result is simpler than that of Theorem \ref{thm:intalg-classifier-from-operad-morphism} because the codecent objects involved are componentwise discrete category objects.

For a non-symmetric operad $T$ with set of colours $I$, applying the foregoing to the unique morphism $T \to \Ass$ into the terminal non-symmetric operad, one has an explicit description of the strict monoidal category $\Ass^T$. This is as with $\Com^T$ at the end of the previous section, but again with order preserving indexing functions instead of general ones. When $\ca V$ is a monoidal category, not necessarily symmetric, $\ca U\ca V$ is a non-symmetric operad, and $\ca V \mapsto \ca U\ca V$ is the effect on objects of
\[ \begin{array}{lccr} {\ca U_{\tn{s}}^{\tnb{M}} : \Algs {\tnb M} \longrightarrow \tnb{Non}{\tn{-}}\Sigma{\tn{-}}\tnb{Opd}} &&& {\ca U_{\tn{ps}}^{\tnb{M}} : \PsAlg {\tnb M} \longrightarrow \tnb{Non}{\tn{-}}\Sigma{\tn{-}}\tnb{Opd}.} \end{array} \]
The non-symmetric analogue of Corollary \ref{cor:smoncat-opd-adjunctions} is
\begin{cor}\label{cor:moncat-nonsigma-opd-adjunctions}
The assignment $T \mapsto \Ass^T$ is the effect on objects of a left 2-adjoint to $\ca U_{\tn{s}}^{\tnb{M}}$, and a left biadjoint to  $\ca U_{\tn{ps}}^{\tnb{M}}$.
\end{cor}
\noindent The adjunction $\Ass^{(-)} \ladj \ca U_{\tn{s}}^{\tnb{M}}$ was studied in detail by Hermida in \cite{Hermida-RepresentableMulticategories}.

The braided variant of Section \ref{ssec:props} works in much the same way. A morphism $F : S \to T$ of braided operads with underlying object function $f : I \to J$, is regarded as an adjunction $F : (\Cat/I,S) \to (\Cat/J,T)$ of 2-monads, whose explicit description is as in Section \ref{ssec:explicit-adj2mnd-op-morphism}, but with braids instead of permutations. To define the analogue of $F \downarrow (-)$, the indexing functions of Definition \ref{defn:F-over-j} are replaced by \emph{indexing vines}, and the analogue of the bijective-monotone factorisation described in Lemma \ref{lem:bij-mon-V}, is used to give composition in $F \downarrow j$. The analogue of Theorem \ref{thm:intalg-classifier-from-operad-morphism} follows by the same proof, modulo the replacement of braids with permutations, indexing functions with indexing vines and the bijective-monotone factorisation by its vine analogue.

The algebras of the terminal braided operad $\BrCom$ are also commutative monoids. For a braided operad $T$ with set of colours $I$, the braided analogue of Theorem \ref{thm:intalg-classifier-from-operad-morphism} applied to $T \to \BrCom$ gives an explicit description of $\BrCom^T$, the free braided strict monoidal category containing a $T$-algebra. This is as for $\Com^T$ in Section \ref{ssec:props} except with indexing vines instead of indexing functions. For a braided monoidal category $\ca V$, $\ca U\ca V$ is a braided operad, and so $\ca V \mapsto \ca U\ca V$ provides the object maps of the 2-functors
\[ \begin{array}{lccr} {\ca U_{\tn{s}}^{\tnb{B}} : \Algs {\tnb{B}} \longrightarrow \tnb{BrOpd}} &&& {\ca U_{\tn{ps}}^{\tnb{B}} : \PsAlg {\tnb{B}} \longrightarrow \tnb{BrOpd}.} \end{array} \]
The braided analogue of Corollary \ref{cor:smoncat-opd-adjunctions} is then
\begin{cor}\label{cor:brmoncat-br-opd-adjunctions}
The assignment $T \mapsto \BrCom^T$ is the effect on objects of a left 2-adjoint to $\ca U_{\tn{s}}^{\tnb{B}}$, and a left biadjoint to  $\ca U_{\tn{ps}}^{\tnb{B}}$.
\end{cor}

\subsection{$\Sigma$-free operads.}
\label{ssec:Sigma-free-operads}
An operad $T$ with set of objects $J$ has two associated 2-monads on $\Cat/J$. There is the polynomial 2-monad which we also denote as $T$ because of how operads are identified as polynomial monads in \cite{Weber-OpPoly2Mnd}; and there is the 2-monad which we denote as $T/\Sigma$, whose strict $T$-algebras are the $\Cat$-valued algebras of the operad $T$ in the conventional sense. In Examples \ref{exams:adj-monad-op-morphism} we saw that a morphism of operads $F : S \to T$ with object function $f : I \to J$, gives rise to an adjunction of 2-monads
\[ \begin{array}{lccr} {F : (\Cat/I,S) \longrightarrow (\Cat/J,T)} &&& {F/\Sigma : (\Cat/I,S/\Sigma) \longrightarrow (\Cat/J,T/\Sigma)} \end{array} \]
as on the left, and in this section we shall establish another adjunction of 2-monads as on the right in the previous display. This gives a few different internal algebra classifers which play the role of a ``universal $S$-algebra internal to a $T$-algebra''. In this section we understand how they are related. In particular this enables us to reconcile our $T^S$ as computed in Section \ref{ssec:props}, with the internal algebra classifiers computed by Batanin and Berger in \cite{BataninBerger-HtyThyOfAlgOfPolyMnd}, when $S$ and $T$ are $\Sigma$-free.

The polynomial 2-monad associated to an operad $T$ alluded to above was recalled in Examples \ref{exams:algebras-of-operad-in-V}. In Section 5 of \cite{Weber-OpPoly2Mnd}, the different types of algebras (lax, pseudo or strict) were characterised as certain types of weak operad morphisms $T \to \CatAsOp$, where $\CatAsOp$ is the operad whose objects are small categories. Even the strict algebras of $T$ are weaker than honest operad morphisms $T \to \CatAsOp$, because of the presence certain coherence isomorphisms called \emph{symmetries}. When the symmetries of an algebra of $T$ are identities, then the algebra is said to be \emph{commutative}. So in particular the commutative strict $T$-algebras correspond exactly to the $\Cat$-valued algebras of the operad $T$ in the conventional sense.

Also in \cite{Weber-OpPoly2Mnd}, the 2-monad $T/\Sigma$ on $\Cat/J$ whose strict $T$-algebras are commutative $T$-algebras were constructed from $T$ as a reflexive coidentifier of 2-monads as on the right
\[ \xygraph{*{\xybox{\xygraph{!{0;(1.5,0):(0,.6)::} {J}="p0" [ur] {E^{[1]}_T}="p1" [r] {B^{[1]}_T}="p2" [dr] {J}="p3" [dl] {B_T}="p4" [l] {E_T}="p5" "p0":@{<-}"p1"^-{s^{[1]}_T}:"p2"^-{p^{[1]}_T}:"p3"^-{t^{[1]}_T}:@{<-}"p4"^-{t_T}:@{<-}"p5"^-{p_T}:"p0"^-{s_T}
"p1":@/_{1.25pc}/"p5"_(.425){d_{E_T}}|(.425){}="dal" "p2":@/_{1.25pc}/"p4"_(.425){d_{B_T}}|(.425){}="dar"
"p1":@/^{1.25pc}/"p5"^(.575){c_{E_T}}|(.575){}="cal" "p2":@/^{1.25pc}/"p4"^(.575){c_{E_T}}|(.575){}="car"
"dal":@{}"cal"|(.35){}="dl"|(.65){}="cl" "dl":@{=>}"cl"^-{\alpha_{E_T}}
"dar":@{}"car"|(.35){}="dr"|(.65){}="cr" "dr":@{=>}"cr"^-{\alpha_{B_T}}}}}
[r(5.5)]
*!(0,.58){\xybox{\xygraph{!{0;(1.5,0):(0,1)::} {T^{[1]}_{\Sigma}}="p0" [r] {T}="p1" [r] {T/\Sigma}="p2" "p0":@<1.5ex>"p1"|(.45){}="t"^-{d_T} "p0":@<-1.5ex>"p1"|(.45){}="b"_-{c_T}:"p2"^-{q_T} "t":@{}"b"|(.15){}="d"|(.85){}="c" "d":@{=>}"c"^-{\alpha_T}}}}} \]
of $\alpha_T$, which itself is the result of applying $\PFun{\Cat}$ to the 2-cell in $\Polyc{\Cat}(I,I)$ on the left, in which the 2-cell data comes from taking the cotensor of $E_T$ and $B_T$ with the category $[1]$. The induced 2-functor $\overline{q}_T : \Algs {T/\Sigma} \to \Algs T$ is exactly the inclusion of the commutative $T$-algebras amongst the general ones. Its left adjoint is denoted as $C_T$, and by Proposition 7.6 of \cite{Weber-OpPoly2Mnd}, the component of the unit of $C_T \ladj \overline{q}_T$ at $X \in \Algs T$ is given by the reflexive coidentifier
\begin{equation}\label{eq:C_TX-as-coidentifier}
\xygraph{!{0;(2,0):(0,1)::} {T^{[1]}_{\Sigma}X}="p0" [r] {X}="p1" [r] {C_TX.}="p2" "p0":@<1.5ex>"p1"|(.45){}="t"^-{xd_{T,X}}:"p2"^-{r_X} "p0":@<-1.5ex>"p1"|(.45){}="b"_-{xc_{T,X}} "t":@{}"b"|(.15){}="d"|(.85){}="c" "d":@{=>}"c"^-{x\alpha_{T,X}}}
\end{equation}
In Theorem 7.7 of \cite{Weber-OpPoly2Mnd}, when $T$ is $\Sigma$-free, $C_T \ladj \overline{q}_T$ was exhibited as a Quillen equivalence with respect to the Lack model structures \cite{Lack-HomotopyAspects2Monads} on $\Algs T$ and $\Algs {T/\Sigma}$.
\begin{const}\label{const:F/Sigma}
Given a morphism of operads $F: S \to T$ with object function $f : I \to J$, we now construct the adjunction of 2-monads
\[ F/\Sigma : (\Cat/I,S/\Sigma) \longrightarrow (\Cat/J,T/\Sigma) \]
whose underlying adjunction is, as with $F$, $\Sigma_f \ladj \Delta_f$.

By Lemma 5.3 of \cite{Weber-OpPoly2Mnd} the 2-functor $(-)^{[1]}$ preserves enough distributivity pullbacks so that applying it componentwise sends the morphism of polynomial monads in $\Cat$ as on the left
\[ \xygraph{{\xybox{\xygraph{!{0;(1.5,0):(0,.6667)::} {I}="p0" [r] {E_S}="p1" [r] {B_S}="p2" [r] {I}="p3" [d] {J}="p4" [l] {B_T}="p5" [l] {E_T}="p6" [l] {J}="p7" "p0":@{<-}"p1"^-{s_S}:"p2"^-{p_S}:"p3"^-{t_S}:"p4"^-{f}:@{<-}"p5"^-{t_T}:@{<-}"p6"^-{p_T}:"p7"^-{s_T}:@{<-}"p0"^-{f} "p1":"p6"_-{f_2} "p2":"p5"^-{f_1} "p0":@{}"p6"|-{=} "p1":@{}"p5"|-{\tn{pb}} "p2":@{}"p4"|-{=}}}}
[r(6.5)]
{\xybox{\xygraph{!{0;(1.5,0):(0,.6667)::} {I}="p0" [r] {E^{[1]}_S}="p1" [r] {B^{[1]}_S}="p2" [r] {I}="p3" [d] {J}="p4" [l] {B^{[1]}_T}="p5" [l] {E^{[1]}_T}="p6" [l] {J}="p7" "p0":@{<-}"p1"^-{s^{[1]}_S}:"p2"^-{p^{[1]}_S}:"p3"^-{t^{[1]}_S}:"p4"^-{f}:@{<-}"p5"^-{t^{[1]}_T}:@{<-}"p6"^-{p^{[1]}_T}:"p7"^-{s^{[1]}_T}:@{<-}"p0"^-{f} "p1":"p6"_-{f^{[1]}_2} "p2":"p5"^-{f^{[1]}_1} "p0":@{}"p6"|-{=} "p1":@{}"p5"|-{\tn{pb}} "p2":@{}"p4"|-{=}}}}} \]
to a morphism of polynomial monads as on the right. Applying $\PFun{\Cat}$ to this gives an adjunction of 2-monads
\[ F^{[1]} : (\Cat/I,S^{[1]}_{\Sigma}) \longrightarrow (\Cat/J,T^{[1]}_{\Sigma}) \]
whose underlying adjunction is $\Sigma_f \ladj \Delta_f$. The diagrams
\[ \xygraph{{\xybox{\xygraph{!{0;(1.5,0):(0,.6667)::}
{I}="p0" [r] {E_S}="p1" [r] {B_S}="p2" [r] {I}="p3" [d] {J}="p4" [l] {B_T}="p5" [l] {E_T}="p6" [l] {J}="p7" "p0":@{<-}"p1"^-{s_S}:"p2"^-{p_S}:"p3"^-{t_S}:"p4"^-{f}:@{<-}"p5"^-{t_T}:@{<-}"p6"^-{p_T}:"p7"^-{s_T}:@{<-}"p0"^-{f} "p1":"p6"_-{f_2} "p2":"p5"^-{f_1} "p0":@{}"p6"|-{=} "p1":@{}"p5"|-{\tn{pb}} "p2":@{}"p4"|-{=}
"p1" [u(1.25)] {E^{[1]}_S}="q1" [r] {B^{[1]}_S}="q2"
"p0":@{<-}"q1"^-{s^{[1]}_S}:"q2"^-{p^{[1]}_S}:"p3"^-{t^{[1]}_S}
"q1":@<-1.75ex>"p1"|(.6){}="l1" "q1":@<1.75ex>"p1"|(.6){}="r1" "l1":@{}"r1"|(.25){}="d1"|(.75){}="c1" "d1":@{=>}"c1"^-{\alpha_{E_S}}
"q2":@<-1.75ex>"p2"|(.6){}="l2" "q2":@<1.75ex>"p2"|(.6){}="r2" "l2":@{}"r2"|(.25){}="d2"|(.75){}="c2" "d2":@{=>}"c2"^-{\alpha_{B_S}}}}}
[r(6.5)]
{\xybox{\xygraph{!{0;(1.5,0):(0,.6667)::}
{I}="p0" [r] {E^{[1]}_S}="p1" [r] {B^{[1]}_S}="p2" [r] {I}="p3" [d] {J}="p4" [l] {B^{[1]}_T}="p5" [l] {E^{[1]}_T}="p6" [l] {J}="p7" "p0":@{<-}"p1"^-{s^{[1]}_S}:"p2"^-{p^{[1]}_S}:"p3"^-{t^{[1]}_S}:"p4"^-{f}:@{<-}"p5"^-{}:@{<-}"p6"^-{p^{[1]}_T}:"p7"^-{}:@{<-}"p0"^-{f} "p1":"p6"_-{f^{[1]}_2} "p2":"p5"^-{f^{[1]}_1} "p0":@{}"p6"|-{=} "p1":@{}"p5"|-{\tn{pb}} "p2":@{}"p4"|-{=}
"p6" [d(1.25)] {E^{[1]}_T}="q1" [r] {B^{[1]}_T}="q2"
"p7":@{<-}"q1"_-{s_T}:"q2"_-{p_T}:"p4"_-{t_T}
"p6":@<-1.75ex>"q1"|(.6){}="l1" "p6":@<1.75ex>"q1"|(.6){}="r1"
"l1":@{}"r1"|(.25){}="d1"|(.75){}="c1" "d1":@{=>}"c1"^-{\alpha_{E_T}}
"p5":@<-1.75ex>"q2"|(.6){}="l2" "p5":@<1.75ex>"q2"|(.6){}="r2" "l2":@{}"r2"|(.25){}="d2"|(.75){}="c2" "d2":@{=>}"c2"^-{\alpha_{B_T}}}}}} \]
are equal when regarded as morphisms of $\Polyc{\Cat}(J,J)$, by the naturality of the data exhibiting cotensors with $[1]$ in $\Cat$. Applying $\PFun{\Cat}$ to this equation gives $F^m(F_!\alpha_SF^*) = \alpha_TF^{[1],m}$ in the diagram on the left
\[ \xygraph{*{\xybox{\xygraph{!{0;(2.5,0):(0,.5)::}
{F_!S^{[1]}_{\Sigma}F^*}="p0" [r] {F_!SF^*}="p1" [r(.8)] {F_!(S/\Sigma)F^*}="p2"
"p0":@<1.5ex>"p1"|(.35){}="t1"^-{}:"p2"^-{F_!q_SF^*} "p0":@<-1.5ex>"p1"|(.35){}="b1"_-{}
"t1":@{}"b1"|(.15){}="d1"|(.85){}="c1" "d1":@{=>}"c1"^-{F_!\alpha_SF^*}
"p0" [d] {T^{[1]}_{\Sigma}}="q0" [r] {T}="q1" [r(.8)] {T/\Sigma}="q2"
"q0":@<1.5ex>"q1"|(.45){}="t2"^-{d_T}:"q2"_-{q_T} "q0":@<-1.5ex>"q1"|(.45){}="b2"_-{c_T} "t2":@{}"b2"|(.15){}="d2"|(.85){}="c2" "d2":@{=>}"c2"^-{\alpha_T}
"p0":"q0"_-{F^{[1],m}} "p1":"q1"^-{F^m} "p2":"q2"^-{(F/\Sigma)^m} "p1":@{}"q2"|-{=}}}}
[r(6.5)]
*!(0,-.1){\xybox{\xygraph{!{0;(2,0):(0,.625)::}
{S^{[1]}_{\Sigma}F^*}="p0" [r] {SF^*}="p1" [r] {(S/\Sigma)F^*}="p2"
"p0":@<1.5ex>"p1"|(.45){}="t1"^-{}:"p2"^-{q_SF^*} "p0":@<-1.5ex>"p1"|(.45){}="b1"_-{}
"t1":@{}"b1"|(.15){}="d1"|(.85){}="c1" "d1":@{=>}"c1"^-{\alpha_SF^*}
"p0" [d] {F^*T^{[1]}_{\Sigma}}="q0" [r] {F^*T}="q1" [r] {F^*(T/\Sigma)}="q2"
"q0":@<1.5ex>"q1"|(.45){}="t2"^-{}:"q2"_-{F^*q_T} "q0":@<-1.5ex>"q1"|(.45){}="b2"_-{} "t2":@{}"b2"|(.15){}="d2"|(.85){}="c2" "d2":@{=>}"c2"^-{F^*\alpha_T}
"p0":"q0"_-{F^{[1],l}} "p1":"q1"^-{F^l} "p2":"q2"^-{(F/\Sigma)^l} "p1":@{}"q2"|-{=}}}}} \]
in which $F^m$ is the mate of the lax and the colax monad morphism coherence 2-cells $F^l$ and $F^c$, and $F^{[1],m}$ is defined similarly. Sifted colimits in the appropriate functor categories are preserved by precomposition with any 2-functor, and by post composition with any of the participating 2-functors $F_!$, $F^*$, $S$, $T$, $S/\Sigma$ and $T/\Sigma$, since these are all sifted colimit preserving. Thus the rows in the above diagrams are reflexive coidentifiers, and the diagrams are mates via $F_! \ladj F^*$. Thus one induces $(F/\Sigma)^m$ and $(F/\Sigma)^l$ uniquely as shown. The axioms exhibiting $(F/\Sigma)^l$ as a lax monad morphism coherence datum, follows from those for $F^l$, the monad morphism axioms for $q_S$ and $q_T$, and since $q_S$ is an epimorphism which is preserved by pre and post composition by all the participating 2-functors.
\end{const}
\begin{rem}\label{rem:F/Sigma-algebra-view}
From Construction \ref{const:F/Sigma} one has a commuting square of adjunctions of 2-monads as on the left
\[ \xygraph{*{\xybox{\xygraph{!{0;(3,0):(0,.3333)::} {(\Cat/I,S)}="p0" [r] {(\Cat/J,T)}="p1" [d] {(\Cat/J,T/\Sigma)}="p2" [l] {(\Cat/I,S/\Sigma)}="p3" "p0":"p1"^-{F}:"p2"^-{q_T}:@{<-}"p3"^-{F/\Sigma}:@{<-}"p0"^-{q_S}}}}
[r(6)]
*!(0,-.05){\xybox{\xygraph{!{0;(3,0):(0,.3333)::} {\Algs {T/\Sigma}}="p0" [r] {\Algs {S/\Sigma}}="p1" [d] {\Algs {S}}="p2" [l] {\Algs {T}}="p3" "p0":"p1"^-{\overline{F/\Sigma}}:"p2"^-{\overline{q}_S}:@{<-}"p3"^-{\overline{F}}:@{<-}"p0"^-{\overline{q}_T}}}}} \]
and thus at the level of strict algebras, a commutative square of 2-functors as on the right. Since $\overline{q}_S$ and $\overline{q}_T$ are the inclusions of commutative $S$ and $T$-algebras respectively, $\overline{F/\Sigma}$ is simply the restriction of $\overline{F}$ to the commutative algebras.
\end{rem}
\begin{prop}\label{prop:F/Sigma-polynomial}
If in the context of Construction \ref{const:F/Sigma}, $S$ and $T$ are $\Sigma$-free, then
\[ \xygraph{!{0;(2,0):(0,.75)::} {I}="p0" [r] {\pi_0E_S}="p1" [r] {\pi_0B_S}="p2" [r] {I}="p3" [d] {J}="p4" [l] {\pi_0B_T}="p5" [l] {\pi_0E_T}="p6" [l] {J}="p7" "p0":@{<-}"p1"^-{\pi_0s_S}:"p2"^-{\pi_0p_S}:"p3"^-{\pi_0t_S}:"p4"^-{f}:@{<-}"p5"^-{\pi_0t_T}:@{<-}"p6"^-{\pi_0p_T}:"p7"^-{\pi_0s_T}:@{<-}"p0"^-{f} "p1":"p6"_-{\pi_0f_2} "p2":"p5"^-{\pi_0f_1}} \]
is a morphism in $\PolyMnd{\Cat}$, which $\PFun{\Cat}$ sends to $F/\Sigma$.
\end{prop}
\begin{proof}
In the diagram
\begin{equation}\label{diag:F/Sigma-as-polynomial-monad-morphism} 
\xygraph{!{0;(2,0):(0,.75)::}
{I}="p0" [r] {E_S}="p1" [r] {B_S}="p2" [r] {I}="p3" [d] {J}="p4" [l] {B_T}="p5" [l] {E_T}="p6" [l] {J}="p7" "p0":@{<-}"p1"^-{s_S}:"p2"^-{p_S}:"p3"^-{t_S}:"p4"^-{f}:@{<-}"p5"^-{}|(.4)*=<4pt>{}:@{<-}"p6"|(.4)*=<4pt>{}|(.55){p_T}:"p7"^-{s_T}:@{<-}"p0"^-{f} "p1":"p6"^(.45){f_2}|(.61)*=<4pt>{} "p2":"p5"^(.45){f_1}|(.7)*=<4pt>{}
"p1" [d(.7)r(.6)] {\pi_0E_S}="q1" "p2" [d(.7)r(.6)] {\pi_0B_S}="q2" "p6" [d(.7)r(.6)] {\pi_0E_T}="q6" "p5" [d(.7)r(.6)] {\pi_0B_T}="q5"
"p0":@{<-}@/_{.75pc}/"q1"_-{\pi_0s_S}:"q2"^-{}:@/_{.375pc}/"p3"|(.45){\pi_0t_S}:"p4"^-{f}:@{<-}@/^{.375pc}/"q5"^-{\pi_0t_T}:@{<-}"q6"^-{\pi_0p_T}:@/^{.75pc}/"p7"^-{\pi_0s_T}:@{<-}"p0"^-{f} "q1":"q6"^-{\pi_0f_2} "q2":"q5"|-{\pi_0f_1}
"p1":"q1"^(.6){q_{E_S}} "p2":"q2"^(.6){q_{B_S}} "p6":"q6"_(.4){q_{E_T}} "p5":"q5"_(.4){q_{B_T}}}
\end{equation}
the back face is as in Examples \ref{exams:adj-monad-op-morphism}, and the front face is the result of applying $\pi_0$ to the back face. Since $S$ and $T$ are $\Sigma$-free the categories $E_S$, $B_S$, $E_T$ and $B_T$ are equivalent to discrete categories, and the top and bottom squares of the inner cube are pullbacks, by \cite{Weber-OpPoly2Mnd} Proposition 6.4. Since applying $\pi_0$ to the composite pullback square with vertices $(E_S,\pi_0E_T,\pi_0B_T,B_S)$ is the front square, that front square, which is the middle square of the diagram in the statement, is also a pullback by Lemma \ref{lem:pi0-pres-enough-pbs}.

Thus one may interpret (\ref{diag:F/Sigma-as-polynomial-monad-morphism}) as a commutative square $\ca S$ in $\Polyc{\Cat}(J,J)$. By the definitions of $q_{E_S}$, $q_{B_S}$, $q_{E_T}$ and $q_{B_T}$, and Lemma 6.5 of \cite{Weber-OpPoly2Mnd}, $\ca S$ is sent by $\PFun{\Cat}$ to the commutative square
\[ \xygraph{!{0;(2.5,0):(0,.4)::} {F_!SF^*}="p0" [r] {F_!(S/\Sigma)F^*}="p1" [d] {T/\Sigma}="p2" [l] {T}="p3" "p0":"p1"^-{F_!q_SF^*}:"p2"^-{(F/\Sigma)^m}:@{<-}"p3"^-{q_T}:@{<-}"p0"^-{F^m}} \]
of 2-natural transformations from Construction \ref{const:F/Sigma}. Since $\PFun{\Cat}$ is a locally faithful homomorphism of 2-bicategories by Proposition 3.2.9 and Theorem 4.1.4 of \cite{Weber-PolynomialFunctors}, $(\pi_0f_2,\pi_0f_1)$ satisfy the axioms of a morphism of $\PolyMnd{\Cat}$ since $(F/\Sigma)^m$ satisfies the axioms making $(F/\Sigma)$ an adjunction of 2-monads in Construction \ref{const:F/Sigma}.
\end{proof}
Let $F : S \to T$ be a morphism of operads. The morphisms $q_S$ and $q_T$ provide us with the full inclusions $\Algs {(S/\Sigma)} \subseteq \Algs S$, $\Algl {(S/\Sigma)} \subseteq \Algl S$, $\Algs {(T/\Sigma)} \subseteq \Algs T$ and $\Algl {(T/\Sigma)} \subseteq \Algl T$; of various of the 2-categories of commutative strict algebras amongst the corresponding 2-categories of strict algebras. Associated to the adjunctions of 2-monads $F$, $(F/\Sigma)$ and $q_TF$, one has the associated internal algebra classifiers $T^S$, $(T/\Sigma)^S$ and $(T/\Sigma)^{(S/\Sigma)}$ respectively, which are all strict $T$-algebras. The following result explains how they are related.
\begin{thm}\label{thm:T^S-vs-T/Sigma^S/Sigma}
Let $F : S \to T$ be a morphism of operads.
\begin{enumerate}
\item $(T/\Sigma)^S = (T/\Sigma)^{(S/\Sigma)} = C_T(T^S)$.
\label{thmcase:formula-T/Sigma^S/Sigma}
\item If $T$ is $\Sigma$-free, then $r_{T^S} : T^S \to (T/\Sigma)^{(S/\Sigma)}$ is an equivalence in $\Alg T$.
\label{thmcase:equivalence-of-T^S-and-T/Sigma^S/Sigma}
\end{enumerate}
\end{thm}
\begin{proof}
(\ref{thmcase:formula-T/Sigma^S/Sigma}): As above we regard $(S/\Sigma)$-algebras (resp. $(T/\Sigma)$-algebras) as commutative $S$-algebras (resp. $T$-algebras). The universal properties of $(T/\Sigma)^S$ and $(T/\Sigma)^{(S/\Sigma)}$ say that for all commutative strict $T$-algebras $Y$, one has
\[ \begin{array}{lccr} {\Algs T((T/\Sigma)^S,Y) \iso \Algl S(1,\overline{F}Y)} &&&
{\Algs T((T/\Sigma)^{(S/\Sigma)},Y) \iso \Algl S(1,\overline{F}Y)} \end{array} \]
naturally in $Y$. Thus $(T/\Sigma)^S = (T/\Sigma)^{(S/\Sigma)}$ by the Yoneda Lemma. The defining universal property of $T^S$ says that for all strict $T$-algebras $Y$,
\[ \Algs T(T^S,Y) \iso \Algl S(1,\overline{F}Y) \]
naturally in $Y$, and this applies in particular for $Y$ commutative. Thus the other equation follows by the Yoneda Lemma and the adjunction $C_T \ladj \overline{q}_T$.

(\ref{thmcase:equivalence-of-T^S-and-T/Sigma^S/Sigma}): By Proposition \ref{prop:uprop-intalg-pseudo} $T^S$ is flexible, and so by Theorem 7.7 of \cite{Weber-OpPoly2Mnd}, $r_{T^S}$ is a weak equivalence in $\Algs T$ for the Lack model structure, and thus an equivalence in $\Alg T$ by Proposition 4.10 of \cite{Lack-HomotopyAspects2Monads}.
\end{proof}
\begin{rem}\label{rem:restrictiveness-of-T/Sigma-as-ambient}
By the equation $(T/\Sigma)^S = (T/\Sigma)^{(S/\Sigma)}$ of Theorem \ref{thm:T^S-vs-T/Sigma^S/Sigma}(\ref{thmcase:formula-T/Sigma^S/Sigma}), an $S$-algebra internal to a $T/\Sigma$-algebra is the same thing as an $S/\Sigma$-algebra internal to a $T/\Sigma$-algebra. This formalises the idea that the ambient structure of a $T/\Sigma$-algebra is too restrictive to notice the difference between internal $S$ and internal $S/\Sigma$-algeba structures. We will encounter a similar phenomenon in Example \ref{exam:all-monoids-idempotent} below.
\end{rem}
\begin{exam}\label{exam:Ass}
Consider the case where $F = 1_{\Ass}$. As a symmetric operad $\Ass$ has one object, an $n$-ary operation for each permutation $\rho \in \Sigma_n$, and its $\Set$-valued algebras are exactly monoids. Recall from Example 4.9  of \cite{Weber-OpPoly2Mnd} that a strict algebra structure of the 2-monad $\Ass$ on $\ca V \in \Cat$, consists of an $n$-ary tensor product functor $\bigotimes_{\rho} : \ca V^n \to \ca V$ for each permutation $\rho \in \Sigma_n$, and for $\rho_1, \rho_2 \in \Sigma_n$, an isomorphism
\[ \xygraph{{\ca V^n}="p0" [r(2)] {\ca V^n}="p1" [dl] {\ca V}="p2" "p0":"p1"^-{c_{\rho_2}}:"p2"^-{\bigotimes_{\rho_1}}:@{<-}"p0"^-{\bigotimes_{\rho_1\rho_2}} "p0" [d(.5)r(.85)] :@{=>}[r(.3)]^-{\xi_{\rho_1,\rho_2}}} \]
in which $c_{\rho_2}$ permutes the factors according to $\rho_2$. This data must satisfy the axioms $\bigotimes_{1_1} = 1_{\ca V}$, $\bigotimes_{\rho}(\bigotimes_{\rho_j})_j = \bigotimes_{\rho(\rho_j)_j}$, $\xi_{\alpha,1} = \id$, $\xi_{\alpha,\rho_1\rho_2} = \xi_{\alpha,\rho_1}(\xi_{\alpha\rho_1,\rho_2}c_{\rho_1})$, and $\xi_{\rho}(\xi_{\rho_j})_j = \xi_{\rho(\rho_j)_j}$. When the $\xi$'s are identities, all the $n$-ary tensor products coincide giving a strict monoidal structure on $\ca V$. Thus $\Ass/\Sigma = \tnb{M}$ and so $(\Ass/\Sigma)^{(\Ass/\Sigma)} = \Delta_+$.

On the other hand applying Theorem \ref{thm:intalg-classifier-from-operad-morphism}, one obtains the following explicit description of $\Ass^{\Ass}$. The underlying category is obtained by applying Definition \ref{defn:F-over-j} in this case. An object is a pair $(n,\rho)$ where $n \in \N$ and $\rho \in \Sigma_n$. An arrow $(n_1,\rho_1) \to (n_2,\rho_2)$ is a pair $(h,(\beta_k)_{1{\leq}k{\leq}n_2})$ in which $h : \underline{n}_1 \to \underline{n}_2$ is a function, $\beta_k \in \Sigma_{|h^{-1}\{k\}|}$ for each $k$. The commutativity condition says that $\rho_2(\beta_k)_k\rho_h = \rho_1$ where $h = \phi_h\rho_h$ is the bijective monotone factorisation of $h$. By step (\ref{proof-step-underlying-codesc}) of the proof of Theorem \ref{thm:intalg-classifier-from-operad-morphism}, any such morphism factors as
\[ \xygraph{!{0;(3,0):(0,1)::}
{(n_1,\rho_1)}="p0" [r] {(n_1,\rho_1\rho_h^{-1})}="p1" [r] {(n_2,\rho_2)}="p2"
"p0":"p1"^-{(\rho_h,(1_1)_k)}:"p2"^-{(\phi_h,(\beta_k)_k)}} \]
one in which the indexing function is a permutation, followed by one in which the indexing function is order preserving. A morphism $(h,(\beta_k)_k) : (n_1,\rho_1) \to (n_2,\rho_2)$ in which $h$ is order preserving, is determined uniquely by $h$, and exists iff $\rho_2h = h\rho_1$. The $\beta_k$ are then recovered by restricting $\rho_1$ to the fibres of $h$. On the other hand a morphism $(n_1,\rho_1) \to (n_2,\rho_2)$ with bijective indexing function exists iff $n_1 = n_2$ and is then forced to have indexing function $\rho_1\rho_2^{-1}$. Thus for each $n$, the full subcategory of $\Ass^{\Ass}$ determined by objects of the form $(n,\rho)$ is the indiscrete category whose objects are the elements of $\Sigma_n$.

The underlying functor of $r_{\Ass^{\Ass}} : \Ass^{\Ass} \longrightarrow \Delta_+$ is the universal functor sending the morphisms with bijective indexing function to identities, by its definition as a reflexive coidentifier which is preserved by $U^{\Ass}$. Its object map is $(n,\rho) \mapsto n$, and by Theorem \ref{thm:T^S-vs-T/Sigma^S/Sigma}(\ref{thmcase:equivalence-of-T^S-and-T/Sigma^S/Sigma}) it is an equivalence. In other words $\Ass^{\Ass}$ is an equivalent fattened version of $\Delta_+$, in which for each $n \in \N$, one has $n!$ copies of $\underline{n} \in \Delta_+$.
\end{exam}
\begin{exams}\label{exams:Batanin-Berger-classifiers}
The examples of Batanin and Berger \cite{BataninBerger-HtyThyOfAlgOfPolyMnd} come from morphisms of polynomial monads in $\Set$ in which the middle maps of the participating polynomials have finite fibres. By \cite{KockJ-PolyFunTrees, SzawielZawadowski-TheoriesOfAnalyticMonads} such polynomial monads can be identified with $\Sigma$-free operads. In the context of Proposition \ref{prop:F/Sigma-polynomial}, Batanin and Berger computed $(T/\Sigma)^{(S/\Sigma)}$, and by Theorem \ref{thm:T^S-vs-T/Sigma^S/Sigma}(\ref{thmcase:equivalence-of-T^S-and-T/Sigma^S/Sigma}) one has $T^S \catequiv (T/\Sigma)^{(S/\Sigma)}$ in $\Alg T$. Note that $(T/\Sigma)^{(S/\Sigma)}$ is easier to compute, as one is then in the situation of Examples \ref{exams:easy-actions-internalised-examples}. The previous example gives a flavour of the sort of redundant information present in $T^S$ in such cases.
\end{exams}
Before our last example, some preliminary remarks are in order. The internal algebra classifier $\Com^{\Ass}$ relative to the adjunction of 2-monads corresponding to the unique operad morphism $\Ass \to \Com$, is the free symmetric strict monoidal category containing a monoid, and so coincides with $\tnb{S}^{\tnb{M}} = \Sigma\Delta$ discussed in Section \ref{ssec:SigmaDelta}. Thus it has natural numbers as objects, a morphism $m \to n$ is a pair $(\rho,f)$ where $\rho \in \Sigma_n$ and $f : m \to n$ is in $\Delta_+$, and the composite
\[ \xygraph{!{0;(1.5,0):(0,1)::} {l}="p0" [r] {m}="p1" [r] {n}="p2" "p0":"p1"^-{(\rho_1,f_1)}:"p2"^-{(\rho_2,f_2)}} \]
is $(\rho_3\rho_1,f_2f_3)$, where $\rho_2f_1 = f_3\rho_3$ is the bijective-monotone factorisation of $\rho_2f_1$ in $\S$. The invertible arrows of $\Sigma\Delta$ are exactly those of the form $(\rho,1)$, and so the maximal subgroupoid of $\Sigma\Delta$ may be identified with $\P$. The tensor product is given on objects by addition of natural numbers, and the underlying object of the universal monoid in $\Sigma\Delta$ is the terminal object $1$. Note moreover that $0$ is a strict initial object in $\Sigma\Delta$. For the purposes of our Example \ref{exam:all-monoids-idempotent}, we need
\begin{lem}\label{lem:SigmaDelta-identify-isos}
If $h : \Sigma\Delta \to X$ is a functor which sends all isomorphisms to identities, then it identifies all parallel pairs of morphisms.
\end{lem}
\begin{proof}
Since $(\rho,f) : m \to n$ in $\Sigma\Delta$ is the composite $(1,f) \comp (\rho,1)$ and $h(\rho,1)$ is an identity, $h(\rho,f)$ does not depend on $\rho$, and so it suffices to show that $h(\rho,f)$ does not depend on $f$. We will use the following morphisms of $\Delta_+$
\[ \begin{array}{lcccccr} {\kappa_{m,n} : m \to n} && {c_b : n \to n+n} && {c_t : n \to n+n} && {d_t : n+n \to n} \end{array} \]
where $\kappa_{m,n}$ sends every element of $\underline{m}$ to the top element $n \in \underline{n}$, $c_b$ is the inclusion of the bottom $n$ elements, $c_t$ is the inclusion of the top $n$ elements, and $d_t$ is unique such that $d_tc_b = 1_n$ and $d_tc_t = \kappa_{n,n}$. Moreover we denote by $\sigma \in \Sigma_{2n}$ the permutation $\tau(1_n,1_n)$ where $1 \neq \tau \in \Sigma_2$, so that $\sigma$ switches the top and bottom $n$ elements of $n+n$. Since $h(\sigma,1) = 1$ and $d_tc_bf = f$, $h$ identifies $(\rho,f)$ with the composite
\[ \xygraph{!{0;(1.5,0):(0,1)::} {m}="p0" [r] {n+n}="p1" [r] {n+n}="p2" [r] {n.}="p3" "p0":"p1"^-{(\rho,c_bf)}:"p2"^-{(\sigma,1)}:"p3"^-{(1,d_t)}} \]
Since $\sigma c_bf = c_tf$ in $\S$, and $d_tc_tf = \kappa_{m,n}$ in $\Delta_+$, this composite is $(1,\kappa_{m,n})$.
\end{proof}
We shall call the algebras of the 2-monad $\Com/\Sigma$ \emph{commutative monoidal categories}, these being exactly the symmetric monoidal categories whose symmetry coherences are identities. In particular a commutative strict monoidal category is exactly a commutative monoid in $\Cat$. We denote by $\C$ the following commutative strict monoidal category. Its objects are natural numbers, its homsets are given by
\[ \C(m,n) = \left\{
\begin{array}{lll} {\emptyset} && {m \neq 0, n = 0} \\ {1} && {\tn{otherwise}} \end{array} \right. \]
and the tensor product is given on objects by addition. The object $1 \in \C$ has a unique commutative monoid structure in $\C$, and we denote by $u : 1 \to \C$ the corresponding symmetric lax monoidal functor. The category $[1] = \{0 < 1\}$ is a commutative strict monoidal category in which the tensor product is given by supremum, and $1 \in [1]$ has a unique commutative monoid structure in $[1]$. The functor $\C \to [1]$, whose object map is given by $n \mapsto 0$ iff $n = 0$, is a symmetric monoidal equivalence, sending the monoid $1 \in \C$ to the monoid $1 \in [1]$.
\begin{exam}\label{exam:all-monoids-idempotent}
We will now show that $u : 1 \to \C$ exhibits $\C = (\Com/\Sigma)^{\Ass}$. Before doing so we consider some of the consequences of this.

Since the multiplication on the monoid $1 \in \C$ is invertible, this shows that any monoid in a commutative monoidal category is idempotent. Thus the existence of a non-idempotent monoid in a symmetric monoidal category $\ca V$, is an obstruction to exhibiting $\ca V$ as symmetric monoidally equivalent to a commutative monoidal category. Since $u :1 \to \C$ is a commutative monoid, it also exhibits $\C = (\Com/\Sigma)^{\Com}$. By the equivalence $\C \catequiv [1]$, for any commutative monoidal category $\ca V$, one has an equivalence
\[ \tn{Mon}(\ca V) \catequiv \PsAlg {\tnb{S}}([1],\ca V)\]
between (commutative) monoids in $\ca V$ and symmetric strong monoidal functors $[1] \to \ca V$.

Let us now establish $\C = (\Com/\Sigma)^{\Ass}$. Since $\Com^{\Ass} = \Sigma\Delta$, $(\Com/\Sigma)^{\Ass}$ is defined as a coidentifier
\[ \xygraph{!{0;(2.5,0):(0,1)::} {\Com^{[1]}_{\Sigma}(\Sigma\Delta)}="p0" [r] {\Sigma\Delta}="p1" [r] {(\Com/\Sigma)^{\Ass}}="p2" "p0":@<1.5ex>"p1"|(.55){}="t"^-{}:"p2"^-{r_{\Sigma\Delta}} "p0":@<-1.5ex>"p1"|(.55){}="b"_-{} "t":@{}"b"|(.15){}="d"|(.85){}="c" "d":@{=>}"c"^-{\beta}} \]
by Theorem \ref{thm:T^S-vs-T/Sigma^S/Sigma}(\ref{thmcase:formula-T/Sigma^S/Sigma}), which is (\ref{eq:C_TX-as-coidentifier}) in the case $T = \Com$ and $X = \Sigma\Delta$. An object of $\Com^{[1]}_{\Sigma}(\Sigma\Delta)$ is a morphism of $\tnb{S}(\Sigma\Delta)$ of the form
\[ (\rho,((1_{n_{\rho i}},1_{n_{\rho i}}))_i) : (n_{\rho 1},...,n_{\rho k}) \longrightarrow (n_1,...,n_k) \]
by \cite{Weber-OpPoly2Mnd} Lemma 5.7, and the corresponding component of $\beta$ is the isomorphism
\[ (\rho(1_{n_{\rho i}})_i,1) : n_{\rho 1} + ... + n_{\rho k} \longrightarrow n_1 + ... n_k \]
of $\Sigma\Delta$. Just considering the cases where the $n_i$'s are all $1$, all isomorphisms of $\Sigma\Delta$ arise as components of $\beta$. Thus $r_{\Sigma\Delta}$ could equally well be described as the universal functor out of $\Sigma\Delta$ which sends isomorphisms to identities. By Lemma \ref{lem:SigmaDelta-identify-isos}, its underlying category agrees with that of $\C$, and the underlying object of its universal monoid is $1$. The condition that $r_{\Sigma\Delta}$ be a symmetric strict monoidal functor, forces $(\Com/\Sigma)^{\Ass}$'s and $\C$'s commutative monoidal structures to agree also.
\end{exam}

\section*{Acknowledgements}
The author would like to acknowledge Michael Batanin, John Bourke, Richard Garner, Joachim Kock, Steve Lack and Ross Street for interesting discussions on the subject of this paper. The author would also like to acknowledge the financial support of the Australian Research Council grant No. DP130101172.


\end{document}